\documentclass[3p,10pt]{elsarticle}
\biboptions{sort&compress}
\usepackage{amsmath,amssymb,amsfonts,amsthm}
\usepackage{mathtools}
\usepackage{mathrsfs}
\usepackage{bm}
\usepackage{enumitem}
\usepackage{graphicx}
\usepackage{booktabs}
\usepackage{hyperref}

\hypersetup{
    colorlinks=true,
    linkcolor=blue,
    citecolor=blue,
    urlcolor=blue,
    pdfauthor={Jie Zhao and Minghua Chen}
}

\newtheorem{theorem}{Theorem}[section]
\newtheorem{lemma}[theorem]{Lemma}
\newtheorem{proposition}[theorem]{Proposition}
\newtheorem{corollary}[theorem]{Corollary}
\newtheorem*{coercivitylemma}{Lemma}

\theoremstyle{definition}
\newtheorem{definition}[theorem]{Definition}

\theoremstyle{remark}
\newtheorem{remark}[theorem]{Remark}
\newtheorem*{remark*}{Remark}

\begin{document}

\begin{frontmatter}

% =========================================================
% Title
% =========================================================
% \title{	Regularity up to the boundary for  tempered fractional Laplacian }
\title{Tempered fractional Laplacian: Boundary regularity}

% =========================================================
% Authors
% =========================================================
% 第一版本
\author[aff1]{Jie Zhao}
\ead{jiezhao2024@lzu.edu.cn}

\author[aff1]{Minghua Chen\corref{cor1}}
\ead{chenmh@lzu.edu.cn}

% \author[aff-ros1]{Xavier Ros-Oton\corref{cor1}}
% \ead{xros@icrea.cat}

\cortext[cor1]
{Corresponding author.}

% \author[aff1]{Jie Zhao}

% \cortext[cor1]{Corresponding author.}

% =========================================================
% Affiliations
% =========================================================
\affiliation[aff1]{
  organization={School of Mathematics and Statistics, Lanzhou University},
  city={Lanzhou},
  postcode={730000},
  country={P.R. China}
}

% Xavier Ros-Oton 的单位
% \affiliation[aff-ros1]{
%     organization={ICREA},
%     addressline={Pg. Lluís Companys 23},
%     city={Barcelona},
%     postcode={08010},
%     country={Spain}
% }

\date{}

\begin{abstract}

We study the boundary regularity of solutions for the tempered fractional Laplacian
\[
    (-\Delta)^s_\lambda u=f \quad \text{in }\Omega,
    \qquad
    u=0 \quad \text{in }\mathbb{R}^n\setminus\Omega,
\]
where \(0<s<1\), \(\lambda>0\), and
\(\Omega\subset\mathbb{R}^n\) is a bounded Lipschitz domain
satisfying the exterior ball condition.
The tempered operator is defined by
\[
    (-\Delta)^s_\lambda u(x)
    =
    c_{n,s}\operatorname{P.V.}
    \int_{\mathbb{R}^n}
    \frac{u(x)-u(y)}{|x-y|^{n+2s}}
    e^{-\lambda|x-y|}\,dy,
\]
which was originally proposed in [Koponen, Phys. Rev. E  (1995),  1197-1199].
% Exponential tempering preserves the leading singularity of
% the kernel but destroys its homogeneity, preventing a direct
% application of classical boundary regularity estimates.
% Extending these estimates [Ros-Oton and Serra, J. Math. Pures Appl. (2014), 275--302] to the tempered setting therefore
% requires overcoming the difficulties caused by the loss
% of scale invariance.
By decomposing the operator into the classical fractional
Laplacian and the correction operator \(\mathcal Q_\lambda^s\), and combining
approximation, localization, and iterative arguments
with interpolation inequalities, we establish the global estimate
\[
    \|u\|_{C^s(\mathbb{R}^n)}
    \le C\|f\|_{L^\infty(\Omega)}.
\]
For more regular data, we derive interior and
boundary-weighted Schauder estimates.
In particular, under suitable regularity assumptions on \(f\),
these estimates yield
\[
    |D^\nu u(x)|
    \le C\,\delta(x)^{s-|\nu|},
    \quad |\nu| > 0,
\]
where \(\nu\) is a multi-index and
\(\delta(x)=\operatorname{dist}(x,\partial\Omega)\).
These bounds quantify the possible growth of higher-order
derivatives near the boundary and provide a regularity
foundation for numerical error analysis of the tempered problem.
\end{abstract}

% =========================================================
% Keywords and MSC
% =========================================================

\end{frontmatter}

\section{Introduction and main results}\label{sec:introduction}
Anomalous diffusion occurs in a wide range of physical systems, with L\'evy flights providing a useful 
model for transport characterized by heavy-tailed displacements \cite{MetzlerKlafter2000}. 
In the isotropic \(2s\)-stable case, with \(0<s<1\), the corresponding L\'evy process has a 
jump measure with density proportional to \(|z|^{-n-2s}\), and its infinitesimal generator is \(-(-\Delta)^s\) \cite{MeerschaertSikorskii2019}. 
The power-law decay allows long-range jumps but yields infinite variance, which may limit the model's applicability when extremely large displacements are strongly suppressed. 
A natural modification is exponential tempering, which introduces the factor \(e^{-\lambda|z|}\), with \(\lambda>0\), into the stable L\'evy measure, 
yielding a density proportional to \(e^{-\lambda|z|}|z|^{-n-2s}\) \cite{Koponen1995,Rosinski2007}. 
This modification preserves the small-jump singularity of the stable model while suppressing large jumps and ensuring finite variance. 
The associated tempered stable process exhibits stable behavior in the short-time scaling limit and Gaussian behavior under long-time diffusive rescaling \cite{Rosinski2007,BaeumerMeerschaert2010}. 
This exponentially tempered L\'evy measure naturally leads to the tempered fractional Laplacian \cite{DengLiTianZhang2018}.
Further probabilistic and analytic foundations include random-walk limits
\cite{ChakrabartyMeerschaert2011}, tempered fractional calculus
\cite{SabzikarMeerschaertChen2015}, estimates of tempered stable densities
\cite{Sztonyk2010}, and gradient estimates for L\'evy semigroups
\cite{SchillingSztonykWang2012}.

In this paper, we investigate the regularity up to the boundary of solutions to the Dirichlet problem
\begin{equation}\label{eq:1.1}
\begin{cases}
(-\Delta)^s_\lambda u=f
& \text{in }\Omega,\\
u=0
& \text{in }\mathbb{R}^n\setminus\Omega,
\end{cases}
\end{equation}
where $0<s<1$, $\lambda>0$, $\Omega\subset\mathbb{R}^n$ is a bounded domain, and $f\in L^\infty(\Omega)$.

The tempered
fractional Laplacian is defined by\cite{DengLiTianZhang2018, MeerschaertSikorskii2019, WeiWangCaraballo2024}
\begin{equation}\label{eq:1.2}
\begin{aligned}
(-\Delta)^s_\lambda u(x)
&=
c_{n,s}\operatorname{P.V.}
\int_{\mathbb{R}^n}
\frac{u(x)-u(y)}
{|x-y|^{n+2s}}
e^{-\lambda|x-y|}
\,dy\\
&=
\frac{c_{n,s}}{2}
\int_{\mathbb{R}^n}
\frac{2u(x)-u(x+y)-u(x-y)}
{|y|^{n+2s}}
e^{-\lambda|y|}
\,dy, \quad x\in \Omega,
\end{aligned}
\end{equation}
where \(\operatorname{P.V.}\) denotes the Cauchy principal value and
\(c_{n,s}\) is the normalization constant for the classical
fractional Laplacian, given by
\[
c_{n,s}
=
\frac{s\,2^{2s}\Gamma\!\left(\frac{n+2s}{2}\right)}
{\pi^{n/2}\Gamma(1-s)}.
\]

Note that when \(\lambda=0\), 
problem~\eqref{eq:1.1}
reduces to the classical fractional Dirichlet problem\cite{RosOtonSerra2014}
  \begin{equation}\label{eq:1.3}
    \begin{cases}
      (-\Delta)^s u = g & \text{in }\Omega, \\[6pt]
      u=0 & \text{in }\mathbb{R}^n\setminus\Omega,
    \end{cases}
  \end{equation}
  % where
  % \begin{equation}\label{eq:1.4}
  % 	\begin{split}
  % (-\Delta)^s u(x) 
  % &= c_{n,s}\,\text{P.V.}\int_{\mathbb{R}^n} \frac{u(x) - u(y)}{|x - y|^{n+2s}} \,dy\\
  % &= \frac{1}{2} c_{n,s}
  %   \int_{\mathbb{R}^n}
  %   \frac{2u(x)-u(x+y) - u(x-y)}{|y|^{n+2s}}\,dy,
  % \end{split}
% \end{equation}
  \begin{remark}\label{remarkfly}
      For \(n=1\), \(s\neq\frac12\), we can check the tempered fractional Laplacian $(-\Delta)^s_\lambda$
      is characterized by
      \begin{align*}
        \mathcal{F}[(-\Delta)^s_\lambda u]
      =
          \frac{
        (\lambda^2+|\xi|^2)^s
        \cos\!\left(2s\arctan\frac{|\xi|}{\lambda}\right)
        -\lambda^{2s}
        }
        {\cos(\pi s)}\widehat{u}(\xi)
      =
        \frac{
        (\lambda+i|\xi|)^{2s}
        +
        (\lambda-i|\xi|)^{2s}
        -
        2\lambda^{2s}
        }{
        2\cos(\pi s)
        }\widehat{u}(\xi).
      \end{align*}
      When $\lambda=0$, using
      $(\pm i|\xi|)^{2s}=|\xi|^{2s}e^{\pm i\pi s}$,
      the above identity reduces to
      \begin{equation*}
            \mathcal{F}[(-\Delta)^s u]
          =
            |\xi|^{2s}\widehat{u}(\xi),
      \end{equation*}
      where $\mathcal{F}$ denotes the Fourier transform operator
      and $\widehat{u}(\xi)=\mathcal{F}(u)(\xi)$, i.e.,
      $
          \widehat{u}(\xi)
          =
          \int_{\mathbb{R}} e^{i\xi x}u(x)\,dx.
      $
  \end{remark}

{
  % \color{blue}
The untempered theory provides the background for our approach.
Its potential-theoretic and fractional Sobolev foundations are developed in
\cite{18,DiNezzaPalatucciValdinoci2012}, while the extension formulation
\cite{8} offers a complementary local interpretation.
Interior H\"older estimates \cite{Silvestre2006} and regularity theory for
fully nonlinear integro-differential equations \cite{CaffarelliSilvestre2009}
provide basic analytical tools.
Boundary behavior has been studied through the boundary Harnack principle
and the structure of $\alpha$-harmonic functions \cite{2,4}, Kelvin
transformations \cite{5}, and Dirichlet heat-kernel estimates
\cite{ref:classical_half_order,10}.
Explicit calculations for balls and power functions
\cite{ref:ball_torsion,Dyda2012}, as well as geometric properties of
ground-state eigenfunctions for the Cauchy process \cite{16}, further
illustrate the role of domain geometry.

Related nonlinear problems include fractional Fujita equations \cite{1},
mountain-pass solutions \cite{23}, and nonexistence results for fractional
elliptic boundary value problems \cite{12}.
Regularity also plays a central role in fractional free-boundary
\cite{7} and obstacle problems \cite{Silvestre2007}, as well as in density
estimates for nonlocal variational models \cite{22}.
In particular, the fractional Pohozaev identity \cite{19,20} highlights
the importance of controlling the boundary quotient $u/\delta^s$,
where $\delta(x)=\operatorname{dist}(x,\partial\Omega)$.

Numerical methods for the tempered fractional Laplacian have made
substantial progress~\cite{ZhangDengKarniadakis2018,SunNieDeng2021,DuoZhang2019}.
% Numerical methods for the tempered fractional Laplacian have made substantial progress
% \cite{ZhangDengKarniadakis2018,SunNieDeng2021,DuoZhang2019}.
However, the numerical error analyses in these works are generally carried out under prescribed regularity assumptions on the solution, rather than deriving such regularity directly from the equation and its data.
In particular, the numerical experiments of Zhang et al.
\cite{ZhangDengKarniadakis2018} indicate that solutions may exhibit limited regularity, and the authors pointed out 
that this phenomenon requires further theoretical regularity investigation.
Related numerical developments include high-order algorithms for
fractional substantial diffusion with truncated L\'evy flights
\cite{ChenDeng2015} and variational formulations for tempered fractional
problems \cite{DengZhang2018}.
% 正则性理论
Recently, regularity theory for the tempered fractional Laplacian~\eqref{eq:1.1} has been developed in several function space settings.
Wei, Wang, and Caraballo \cite{WeiWangCaraballo2024} introduced tempered fractional Sobolev spaces and established embedding theorems.
K\"uhn \cite{Kuhn2019,Kuhn2022} obtained whole-space and interior Schauder estimates for classes of L\'evy generators that include tempered stable operators.
For homogeneous exterior Dirichlet problems, Borthagaray, Li, and Nochetto
\cite{BorthagarayLiNochetto2024} established boundary regularity in Besov spaces, together with corresponding Sobolev estimates, for a class of operators encompassing tempered fractional Laplacians.
However, these regularity results~\cite{WeiWangCaraballo2024,Kuhn2019,Kuhn2022,BorthagarayLiNochetto2024} do not provide the boundary-weighted pointwise estimates for higher-order derivatives 
that are needed in numerical error analyses~\cite{ZhangDengKarniadakis2018,SunNieDeng2021,DuoZhang2019} for boundary singularities.
Thus, a gap remains between the available numerical analysis and a quantitative description of the boundary behavior of solutions.
This connection between boundary regularity and numerical approximation
motivates the present study of boundary-weighted regularity estimates
for the tempered fractional Laplacian.

% The classical fractional Laplacian provides a useful model for bridging this gap.
When $\lambda=0$, the tempered fractional Laplacian reduces to
the classical fractional Laplacian, for which existing results
illustrate how boundary-weighted regularity estimates can bridge
this gap.
Ros-Oton and Serra \cite{RosOtonSerra2014} established boundary regularity and weighted H\"older estimates, from which pointwise derivative bounds in terms of the distance to the boundary can be obtained under suitable assumptions on the data.
Boundary regularity has also been developed for general stable operators
\cite{RosOtonSerra2016} and fully nonlinear integro-differential equations
\cite{21}.
For finite element approximations of the fractional Laplacian, the link
between solution regularity and convergence is studied in
\cite{AcostaBorthagaray2017,BorthagarayCiarlet2019}.
Using the weighted regularity estimates of~\cite{RosOtonSerra2014}, Chen et al.~
\cite{ChenDengMinShiStynes2024,ChenHanShiYu2026} derived pointwise boundary-distance estimates for higher-order
derivatives, and employed these regularity
estimates in the error analysis of numerical schemes on graded
meshes.
Their work~\cite{ChenDengMinShiStynes2024,ChenHanShiYu2026} illustrates how a quantitative description of boundary singularities can support numerical error estimates and guide the choice of mesh grading.

However, extending this type of regularity analysis from the classical fractional Laplacian~\cite{RosOtonSerra2014} to the tempered setting is nontrivial.
Although exponential tempering preserves the leading singularity of the kernel near the origin, it destroys its homogeneity.
As illustrated in Remark~\ref{remarkfly}, the Fourier symbol no longer has the simple power-law form \(|\xi|^{2s}\), making explicit Fourier inversion and potential-kernel representations less straightforward.
Moreover, whole-space Fourier representations alone do not capture the effect of the exterior Dirichlet condition.
% Therefore, establishing boundary-weighted regularity estimates requires additional control of the tempering-dependent terms arising under localization and rescaling, together with quantitative estimates for the lower-order correction to the classical fractional Laplacian.
To overcome these difficulties, a key step in our analysis is the decomposition
\[
    (-\Delta)^s_\lambda u
    =(-\Delta)^s u-\mathcal Q_\lambda^s u.
\] 
We first establish boundedness and H\"older estimates for the
correction operator \(\mathcal Q_\lambda^s\), treating
the cases \(0<s<\tfrac12\) and \(\tfrac12\le s<1\) separately.
To prove global \(C^s\) regularity for weak solutions
of~\eqref{eq:1.1}, we combine a weak comparison principle
and a barrier argument with an approximation of the kernel
near the origin.
Classical estimates for the fractional Laplacian and an
absorption argument yield uniform \(C^s\) bounds for the
approximate solutions.
To derive the Schauder estimates, we combine the estimates
for \(\mathcal Q_\lambda^s\) with cutoff localization,
nonlocal tail estimates, H\"older interpolation,
and an iterative absorption argument.
A rescaling argument on balls with radii comparable to the
distance to the boundary then yields the weighted estimates.
These estimates extend to distributional solutions by
mollification and passage to the limit.
Under the stated assumptions on the data, the weighted
estimates give pointwise bounds for higher-order derivatives
in terms of the distance to the boundary.
We now state our main results.

% In this paper, we establish quantitative regularity estimates for the tempered fractional Laplacian with homogeneous exterior Dirichlet conditions.
% Under the assumptions of our main results, we first obtain a global \(C^s(\mathbb R^n)\) estimate for bounded right-hand sides and then derive interior and boundary-weighted Schauder estimates for more regular data.
% As a consequence, we obtain

% $$
%     |D^\nu u(x)|
%     \leq C\,\delta(x)^{s-|\nu|},
%     \qquad
%     \delta(x)=\operatorname{dist}(x,\partial\Omega),
% $$
% for the derivative orders allowed by the corresponding regularity assumptions on the data.
% These estimates quantitatively describe the possible growth of higher-order derivatives near the boundary and provide a regularity foundation for numerical error analyses of tempered problems that explicitly account for boundary singularities and require pointwise derivative bounds.
}

\subsection{Main results}

% Throughout the paper, a solution of~\eqref{eq:1.1} means a weak
% solution in the sense of Definition~\ref{def:weak-solution}.

The main contributions of this paper can be summarized as follows.
First, we establish an \(L^\infty\)
estimate and prove the optimal global regularity
\[
u\in C^s(\mathbb{R}^n).
\]
Second, we prove that \(u/\delta^s\) admits a H\"older continuous
extension to \(\overline{\Omega}\), thereby identifying
\(\delta^s\) as the natural boundary profile also in the tempered
setting. Third, we establish interior Schauder estimates and
higher-order weighted H\"older estimates. In particular, under
appropriate weighted regularity assumptions on \(f\), we obtain
\[
|D^\nu u(x)|
\le
C\,\delta(x)^{s-|\nu|}
\qquad
\text{for all }x\in\Omega,
\]
for every multi-index \(\nu\) covered by the corresponding assumptions.

% These results show that the principal interior and boundary regularity
% properties of the fractional Laplacian remain stable under exponential
% tempering, although the exact homogeneity and explicit structure of
% the classical operator are lost. They also provide the regularity
% information needed for the stability, convergence, and error analysis
% of numerical schemes for tempered fractional problems.

% \textcolor{red}
% We recall that a domain $\Omega$ satisfies the exterior ball condition if there exists a positive radius $\rho_0$ such that all the points on $\partial\Omega$ can be touched by some exterior ball of radius $\rho_0$. 
The main results in this work are following 
\begin{theorem}\label{pro:1.1}
  Let $0 < s < 1$ , $\lambda > 0$ and $\Omega \subset \mathbb{R}^n$ be bounded, Lipschitz and satisfy the exterior ball condition, $f\in L^\infty(\Omega)$, 
  and $u$ 
  be a solution of \eqref{eq:1.1}.
  Then, $u \in C^s(\mathbb{R}^n)$ and
  \[
  \|u\|_{C^s(\mathbb{R}^n)}\le C
  \|f\|_{L^\infty(\Omega)},
  \]
  where $C$ is a constant depending only on $n$, $s$, $\lambda$ and $\Omega$.
\end{theorem}

{
\begin{corollary}
\label{cor:tempered-boundary-quotient}
Let \(0<s<1\), \(\lambda>0\), and let
\(\Omega\subset\mathbb R^n\) be a bounded \(C^{1,1}\) domain. Suppose
that \(f\in L^\infty(\Omega)\) and that \(u\) is a solution of
\eqref{eq:1.1}. If
\(\delta(x):=\operatorname{dist}(x,\partial\Omega)\), then there exists
\(\alpha\in(0,\min\{s,1-s\})\) such that
\[
\frac{u}{\delta^s}\in C^\alpha(\overline\Omega)
\]
and
\[
\left\|\frac{u}{\delta^s}\right\|_{C^\alpha(\overline\Omega)}
\le C\|f\|_{L^\infty(\Omega)},
\]
where \(C\) is a constant depending only on \(n,s,\lambda\), and \(\Omega\).
\end{corollary}
}

Throughout the paper, and when no confusion is 
possible, we use the notation 
$C^\beta(U)$ with $\beta>0$ to 
refer to the space $C^{k,\beta'}(U)$, 
where $k$ is the greatest integer 
such that $k<\beta$ and where $\beta'=\beta-k$. 
This notation is especially appropriate when 
we work with $(-\Delta)^s_\lambda$ in order to avoid the splitting of different cases in the statements of regularity results. According to this, $[\cdot]_{C^\beta(U)}$ denotes the $C^{k,\beta'}(U)$ seminorm

\[
[u]_{C^\beta(U)}=[u]_{C^{k,\beta'}(U)}=\sup_{x,y\in U,\,x\ne y}\frac{|D^k u(x)-D^k u(y)|}{|x-y|^{\beta'}}.
\]

Moreover, given an open set $U\subset\mathbb{R}^n$ with $\partial U\ne\emptyset$, we will also denote
\[
d_x=\text{dist}(x,\partial U)\quad\text{and}\quad d_{x,y}=\min\{d_x,d_y\}.
\]

We recall first some well known interior 
regularity results for linear equations 
involving the operator $(-\Delta)^s$, 
defined by~\eqref{eq:1.2}. 
The first one states that 
$w\in C^{\beta+2s}(B_{1/2})$ whenever 
$w\in C^\beta(\mathbb{R}^n)$ and $(-\Delta)^s w\in C^\beta(B_1)$. Recall that, throughout this section and in all the paper, we denote by $C^\beta$, with $\beta>0$, the space $C^{k,\beta'}$, where $k$ is an integer, $\beta'\in(0,1]$, and $\beta=k+\beta'$.

We use the boundary-weighted H\"older norms of
\cite{RosOtonSerra2014}, following the classical elliptic framework
of~\cite{14}.

  \begin{definition}\label{Definition_tempered1.3}
      Let $\beta>0$ and $\sigma\ge-\beta$. 
      Let $\beta=k+\beta'$, 
      with $k$ integer and $\beta'\in(0,1]$. For $w\in C^\beta(U)=C^{k,\beta'}(U)$, define the seminorm
      \[
      [w]_{\beta;U}^{(\sigma)}=\sup_{x,y\in U}\left(d_{x,y}^{\beta+\sigma}\frac{|D^k w(x)-D^k w(y)|}{|x-y|^{\beta'}}\right).
      \]
      For $\sigma>-1$, we also define the norm $\|\cdot\|_{\beta;U}^{(\sigma)}$ as follows: in case that $\sigma\ge0$,
      \[
      \|w\|_{\beta;U}^{(\sigma)}=\sum_{l=0}^{k}\sup_{x\in U}\left(d_x^{l+\sigma}|D^l w(x)|\right)+[w]_{\beta;U}^{(\sigma)},
      \]
      while for $-1<\sigma<0$,
      \[
      \|w\|_{\beta;U}^{(\sigma)}=\|w\|_{C^{-\sigma}(U)}+\sum_{l=1}^{k}\sup_{x\in U}\left(d_x^{l+\sigma}|D^l w(x)|\right)+[w]_{\beta;U}^{(\sigma)}.
      \]

      % Note that $\sigma$ is the rescale order of the seminorm $[\cdot]_{\beta;U}^{(\sigma)}$, in the sense that $[w(\lambda\cdot)]_{\beta;U/\lambda}^{(\sigma)}=\lambda^\sigma[w]_{\beta;U}^{(\sigma)}$.

      % When $g$ is H\"older continuous, the next result provides optimal estimates for higher order H\"older norms of $u$ up to the boundary.
  \end{definition}

\begin{theorem}\label{Pro:1.4}
Let $0<s<1$,  $\Omega \subset \mathbb{R}^n$ be a bounded domain, 
and $\beta>0$ be such that neither $\beta$ nor $\beta+2s$ is an integer. 
Let $f\in C^\beta(\Omega)$ be such that $\|f\|_{\beta;\Omega}^{(s)}<\infty$, 
and $u\in C^s(\mathbb{R}^n)$ 
be a solution of~\eqref{eq:1.1}. Then, $u\in C^{\beta+2s}(\Omega)$ and
\[
\|u\|_{\beta+2s;\Omega}^{(-s)}\le C
% \,e^{\lambda \operatorname{diam}\Omega}
\left(\|u\|_{C^s(\mathbb{R}^n)}+\|f\|_{\beta;\Omega}^{(s)}\right),
\]
where $C$ is a constant depending only on $\Omega$, $s$, $\beta$ and $\lambda$.
\end{theorem}

\begin{theorem}\label{lemma1.6-fourth}
  Let $0<s<1$, $\lambda>0$, and let $\Omega\subset\mathbb{R}^n$
  be a bounded Lipschitz domain satisfying the exterior ball condition.
  Let $\beta=4-2s+\gamma$, where
    \[
        \begin{cases}
        0<\gamma<2s, & 0<s\le \dfrac12,\\[0.4em]
        0<\gamma<2s-1, & \dfrac12<s<1.
        \end{cases}
    \]
    Assume that \(f\in L^\infty(\Omega)\cap C^\beta(\Omega)\) satisfies
    $
        \|f\|_{\beta;\Omega}^{(s)}<\infty,
    $
    and let \(u\) be a solution of~\eqref{eq:1.1}. Then, for every
    multi-index \(\nu\) with \(0\le |\nu|\le 4\),
    \[
        |D^\nu u(x)|
        \le
        C\delta(x)^{s-|\nu|}
        \qquad \text{for }x\in\Omega.
    \]
    Moreover, for every multi-index \(\nu\) with \(0\le |\nu|\le 2\),
    \[
        |D^\nu f(x)|
        \le
        C\delta(x)^{-s-|\nu|}
        \qquad \text{for }x\in\Omega.
    \]
    Here \(C\) is a positive constant depending only on
    \(\Omega,n,s,\lambda,\beta\) and \(f\).
\end{theorem}

In the following result, the meaning of ``bounded solution'' is that of ``bounded weak solution'' (see Definition~\ref{def:weak-solution}). 
% By Remark~2.11, these two notions coincide. 
Also, by $f\in C_{\text{loc}}^{0,1}(\Omega\times\mathbb{R})$ we mean that $f$ is Lipschitz in every compact subset of $\Omega\times\mathbb{R}$.
Throughout the paper, by the notation $A \lesssim B$ we mean that 
$A$ can be bounded by a multiple of $B$, independently of the parameters 
they may depend on, while the expression $A \simeq B$ means that 
$A \lesssim B \lesssim A$.

The remainder of the paper is organized as follows. Section~\ref{sec:algebraic} records
the algebraic and variational properties of the tempered fractional
Laplacian. Section~\ref{sec:correction} studies the correction operator
\(\mathcal Q_\lambda^s\), which relates the tempered operator to the
classical fractional Laplacian. These tools are
used in Section~\ref{sec:global-regularity} to prove the global \(C^s\) estimate and the boundary
regularity of \(u/\delta^s\). Section~\ref{sec:schauder} develops the interior and weighted
Schauder estimates, and Section~\ref{sec:higher-order} applies them to derive the stated
higher-order pointwise bounds.

% \section{Basic properties of the tempered fractional Laplacian}

\section{Some properties of the tempered operator
\texorpdfstring{$(-\Delta)^s_\lambda$}{the tempered fractional Laplacian}}\label{sec:algebraic}

This section introduces the variational framework for the tempered
Dirichlet problem and records the structural properties of the tempered
fractional Laplacian used below. In particular, we define weak solutions
and collect the operator identities and comparison tools needed in
Sections~\ref{sec:global-regularity} and~\ref{sec:schauder}.

  For $s\in(0,1)$ and $\lambda>0$, 
  we define the tempered fractional Laplacian of $\eta:\mathbb R^n\to\mathbb R$ by
\begin{equation}
  (-\Delta)^s_\lambda \eta(x)
  := c_{n,s}\,\mathrm{P.V.}\int_{\mathbb R^n}
  \frac{\eta(x)-\eta(y)}{|x-y|^{n+2s}}\,e^{-\lambda|x-y|}\,dy,
  \qquad x\in\mathbb R^n.
\end{equation}

\begin{definition}\label{def:weak-solution}
Let \(\Omega\subset\mathbb R^n\) be a bounded open set, and fix
\(s\in(0,1)\) and \(\lambda>0\). Define
\[
\widetilde H^s(\Omega)
:=
\bigl\{
u\in H^s(\mathbb R^n):
u=0\ \text{a.e. in }\mathbb R^n\setminus\Omega
\bigr\}.
\]
For \(u,v\in H^s(\mathbb R^n)\), define the tempered bilinear form
\[
\mathcal E_\lambda(u,v)
:=
\frac{c_{n,s}}{2}
\iint_{\mathbb R^n\times\mathbb R^n}
\frac{(u(x)-u(y))(v(x)-v(y))}{|x-y|^{n+2s}}
e^{-\lambda|x-y|}
\,dx\,dy.
\]
Since \(0<e^{-\lambda|x-y|}\le1\), this form is well defined on
\(H^s(\mathbb R^n)\times H^s(\mathbb R^n)\). Given
\(f\in(\widetilde H^s(\Omega))'\), we say that
\(u\in\widetilde H^s(\Omega)\) is a weak solution of~\eqref{eq:1.1} if
\[
\mathcal E_\lambda(u,\varphi)
=
\langle f,\varphi\rangle
\qquad
\text{for every }\varphi\in\widetilde H^s(\Omega),
\]
where \(\langle\cdot,\cdot\rangle\) denotes the duality pairing between
\((\widetilde H^s(\Omega))'\) and \(\widetilde H^s(\Omega)\).
\end{definition}

\begin{proposition}
  Let $s\in(0,1)$ and $\lambda>0$.  
  Then the operator $(-\Delta)^s_\lambda$ defined in \eqref{eq:1.2}
  satisfies the following properties.

\begin{enumerate}[label=(\roman*)]
  \item For all $\alpha,\beta\in\mathbb R$ and 
  $\eta,\xi \in C_c^\infty(\mathbb R^n)$,
  \begin{equation}\label{xianxin}
    (-\Delta)^s_\lambda(\alpha \eta+\beta \xi)
    = \alpha(-\Delta)^s_\lambda \eta
      + \beta(-\Delta)^s_\lambda \xi
    \quad\text{in }\mathbb R^n.
  \end{equation}

  \item For every constant $c\in\mathbb R$,
  \[
    (-\Delta)^s_\lambda c \equiv 0
    \quad\text{in }\mathbb R^n.
  \]

  \item For any $x_0\in\mathbb R^n$ 
  and $\eta\in C_c^\infty(\mathbb R^n)$, define
  $\eta_{x_0}(x):=\eta(x+x_0)$. Then
  \[
      (-\Delta)^s_\lambda \eta_{x_0}(x)
    =  
       ((-\Delta)^s_\lambda \eta)(x+x_0),
      \qquad x\in\mathbb R^n .
  \]

  \item For any $r>0$ and 
  $\eta\in C_c^\infty(\mathbb R^n)$, define
  $\eta_r(x):=\eta(rx)$. Then
  \begin{equation*}
     (-\Delta)^s_\lambda \eta_r(x)
    = 
      r^{2s}\,((-\Delta)^s_{\lambda/r} \eta)(rx),
      \qquad x\in\mathbb R^n .
  \end{equation*}
\end{enumerate}
\end{proposition}

\begin{proof}
Properties \textit{(i)} and \textit{(ii)} follow directly from the linearity of the integral
in \eqref{eq:1.1}. 

For \textit{(iii)}, let $\eta_{x_0}(x)=\eta(x+x_0)$ and perform the change of variables
$z=y+x_0$ in \eqref{eq:1.1}; since the kernel
${e^{-\lambda|x-y|}}/({|x-y|^{n+2s}})$ depends only on $x-y$, one obtains
\begin{align*}
  (-\Delta)^s_\lambda \eta_{x_0}(x)
  &= C_{n,s}\,\mathrm{P.V.}\int_{\mathbb R^n}
    \frac{\eta(x+x_0)-\eta(z)}{|(x+x_0)-z|^{n+2s}}
    e^{-\lambda|(x+x_0)-z|}\,dz\\
  &= ((-\Delta)^s_\lambda \eta) (x+x_0).
\end{align*}

For \textit{(iv)}, let $\eta_r(x)=\eta(rx)$ and set $\zeta = ry$ in
\eqref{eq:1.1}. Then $y=\zeta/r$, $dy = r^{-n}\,d\zeta$ and
\[
  |x-y|
  = \left|x-\frac{\zeta}{r}\right|
  = \frac1r |rx-\zeta|.
\]
Hence
\[
\begin{aligned}
  (-\Delta)^s_\lambda \eta_r(x)
  &= C_{n,s}\,\mathrm{P.V.}\int_{\mathbb R^n}
     \frac{\eta(rx)-\eta(ry)}{|x-y|^{n+2s}}e^{-\lambda|x-y|}\,dy \\
  &= C_{n,s}\,\mathrm{P.V.}\int_{\mathbb R^n}
     \frac{\eta(rx)-\eta(\zeta)}{|x-\zeta/r|^{n+2s}}
     e^{-\lambda|x-\zeta/r|}\,r^{-n}\,d\zeta \\
  &= C_{n,s}\,r^{2s}\,\mathrm{P.V.}\int_{\mathbb R^n}
     \frac{\eta(rx)-\eta(\zeta)}{|rx-\zeta|^{n+2s}}
     e^{-(\lambda/r)|rx-\zeta|}\,d\zeta \\
  &= r^{2s}\,((-\Delta)^s_{\lambda/r}\eta) (rx),
\end{aligned}
\]
This completes the proof.
\end{proof}

\begin{proposition}
Let $s\in(0,1)$ and $\lambda>0$. For any orthogonal matrix $\mathbf M\in O(n)$ and any
$\eta\in C_c^\infty(\mathbb R^n)$, define
$\eta_{\mathbf M}(x):=\eta(\mathbf M x)$, $x\in\mathbb R^n.$
Then
\[
  (-\Delta)^s_\lambda \eta_{\mathbf M}(x)
=
  ((-\Delta)^s_\lambda \eta)(\mathbf M x),\qquad x\in\mathbb R^n.
\]
\end{proposition}

\begin{proof}
By definition,
\[
(-\Delta)^s_\lambda \eta_{\mathbf M}(x)
=
c_{n,s}\,\mathrm{P.V.}\int_{\mathbb R^n}
\frac{\eta(\mathbf M x)-\eta(\mathbf M y)}{|x-y|^{n+2s}}e^{-\lambda|x-y|}\,dy.
\]
Let $z=\mathbf M y$. Since $\mathbf M$ is orthogonal, we have $dy=dz$ and
\[
|x-y|=|\mathbf M x-\mathbf M y|=|\mathbf M x-z|.
\]
Hence
\[
  \begin{aligned}
      (-\Delta)^s_\lambda \eta_{\mathbf M}(x)
    &=
      c_{n,s}\,\mathrm{P.V.}\int_{\mathbb R^n}
      \frac{\eta(\mathbf M x)-\eta(z)}{|\mathbf M x-z|^{n+2s}}e^{-\lambda|\mathbf M x-z|}\,dz \\
    &=
      ((-\Delta)^s_\lambda \eta)(\mathbf M x).
    \end{aligned}
\]
  This proves the proposition.
\end{proof}

\begin{proposition}
Let $s\in(0, 1)$ and $\lambda>0$, and let $\eta \in C_c^\infty(\mathbb R^n)$.

\begin{enumerate}
\item If $x_0\in\mathbb R^n$ is a global maximum point of $\eta$, then
\[
(-\Delta)^s_\lambda \eta(x_0)\ge 0.
\]

\item If $x_0\in\mathbb R^n$ is a global minimum point of $\eta$, then
\[
(-\Delta)^s_\lambda \eta(x_0)\le 0.
\]
\end{enumerate}
\end{proposition}

\begin{proof}
Suppose first that $x_0$ is a global maximum point of $\eta$. Then
  \[
    \eta(x_0)-\eta(y)\ge 0 \qquad \text{for all } y\in\mathbb R^n.
  \]
Since the kernel
 ${e^{-\lambda|x_0-y|}}/({|x_0-y|^{n+2s}})$
 is nonnegative, the integrand in the definition of $(-\Delta)^s_\lambda \eta(x_0)$ is nonnegative, and therefore
  \[
    (-\Delta)^s_\lambda \eta(x_0)
  =
    c_{n,s}\,\mathrm{P.V.}\int_{\mathbb R^n}
    \frac{\eta(x_0)-\eta(y)}{|x_0-y|^{n+2s}}e^{-\lambda|x_0-y|}\,dy
  \ge 
    0.
  \]
The proof for a global minimum point is identical.
\end{proof}

\begin{remark}
Unlike the classical fractional Laplacian, which is homogeneous under dilations, 
the tempered fractional Laplacian is not homogeneous in the same sense
Its scaling law necessarily involves a simultaneous
rescaling of the tempering parameter, namely $\lambda \mapsto \lambda/r.$
More precisely,
$(-\Delta)^s_\lambda \eta_r(x)=r^{2s}((-\Delta)^s_{\lambda/r}\eta)(rx)$ 
where $\eta_r(x):=\eta(rx)$.
Thus the scaling changes not only the prefactor $r^{2s}$, but also the tempering parameter $\lambda$, which becomes $\lambda/r$.
\end{remark}

% \section{Some regularity of \texorpdfstring{$\mathcal{Q}^s_\lambda u$}{Q-lambda-s u}}
\section{Some properties of the correction operator
\texorpdfstring{$\mathcal Q_\lambda^s$}{Q}}\label{sec:correction}

This section studies the correction operator
\(\mathcal Q_\lambda^s\), which relates the tempered operator to the
classical fractional Laplacian. We define it at the weak level and
establish the pointwise, \(L^\infty\), and H\"older estimates used in
Sections~\ref{sec:global-regularity} and~\ref{sec:schauder}.

\begin{definition}
Let \(n\ge1\), \(0<s<1\), and \(\lambda>0\). For
\(u\in H^s(\mathbb R^n)\), define the \emph{correction operator}
\(\mathcal Q^s_\lambda u\) distributionally by
\begin{equation}\label{eq:Qlambda-weak}
\begin{aligned}
\bigl\langle \mathcal Q^s_\lambda u,\varphi\bigr\rangle
:={}&\frac{c_{n,s}}{2}
\iint_{\mathbb R^n\times\mathbb R^n}
\frac{(u(x)-u(y))(\varphi(x)-\varphi(y))}
{|x-y|^{n+2s}}\\
&\hspace{3.2cm}\times
\bigl(1-e^{-\lambda|x-y|}\bigr)\,dx\,dy,
\qquad \varphi\in H^s(\mathbb R^n).
\end{aligned}
\end{equation}
This pairing is well defined because
\(0\le1-e^{-\lambda|x-y|}\le1\), and the Cauchy--Schwarz inequality
gives
\[
\bigl|\langle \mathcal Q^s_\lambda u,\varphi\rangle\bigr|
\le C [u]_{H^s(\mathbb R^n)}[\varphi]_{H^s(\mathbb R^n)}.
\]

The correction operator also has an absolutely convergent pointwise
representation if either \(0<s<\tfrac12\) and
\(u\in L^\infty(\mathbb R^n)\), or \(\tfrac12\le s<1\) and
\(u\in L^\infty(\mathbb R^n)\cap C^\gamma(\mathbb R^n)\) for some
\(2s-1<\gamma\le1\):
\begin{align}\label{eq:Qlambda}
\mathcal{Q}^s_\lambda u(x)
:={}&c_{n,s}\int_{\mathbb{R}^n}
\frac{u(x)-u(y)}{|x-y|^{n+2s}}
\Bigl(1-e^{-\lambda|x-y|}\Bigr)\,dy\\
={}&c_{n,s}\int_{\mathbb{R}^n}
\frac{u(x)-u(x+z)}{|z|^{n+2s}}
\bigl(1-e^{-\lambda|z|}\bigr)\,dz .
\end{align}
For \(s<\tfrac12\), the kernel itself is locally integrable. For
\(s\ge\tfrac12\), the H\"older estimate for the increment of \(u\)
bounds the integrand near \(z=0\) by a constant times
\(|z|^{-n-2s+1+\gamma}\), which is integrable precisely when
\(\gamma>2s-1\). At infinity, boundedness of \(u\) gives integrability.

Whenever the operators admit pointwise representatives, one has
\begin{equation}\label{Q=g-f}
\mathcal{Q}^s_\lambda u
=(-\Delta)^s u-(-\Delta)^s_\lambda u.
\end{equation}
For a weak solution \(u\in\widetilde H^s(\Omega)\) of
\eqref{eq:1.1}, the corresponding identity is understood in
\(\mathcal D'(\Omega)\):
\begin{equation}\label{eqQ}
(-\Delta)^s u=f+\mathcal Q^s_\lambda u.
\end{equation}
The correction operator \(\mathcal Q^s_\lambda\) is of lower order: its
kernel is one power less singular near the diagonal, although the kernel
itself is locally integrable only when \(s<\tfrac12\).

When \(\lambda=0\), the factor \(1-e^{-\lambda|x-y|}\) vanishes, so
\(\mathcal Q^s_0u\equiv0\) and
\[
(-\Delta)^s_0u=(-\Delta)^su.
\]
\end{definition}

We next establish the required \(L^\infty\) and H\"older bounds for
\(\mathcal Q_\lambda^s\).

\begin{lemma}\label{lemmaQL}
    Let $n\ge 1$, $0<s<1$, and $\lambda>0$. Define
    \[
    \mathcal{Q}^s_\lambda  u(x)
    :=
    c_{n,s}\int_{\mathbb{R}^n}
    \frac{u(x)-u(x+z)}{|z|^{n+2s}}
    \left(1-e^{-\lambda |z|}\right)\,dz .
    \]
    Then the following assertions hold.

    \begin{itemize}
        \item[(i)] {If \(0<s<\frac12\) and
        \(u\in L^\infty(\mathbb R^n)\), then
        \(\mathcal Q_\lambda^s u\in L^\infty(\mathbb R^n)\) and}
        \begin{equation}\label{QL<UL}
          \|\mathcal{Q}^s_\lambda  u\|_{L^\infty(\mathbb{R}^n)}
           \le
          C\|u\|_{L^\infty(\mathbb{R}^n)}.
        \end{equation}

        \item[(ii)] {If
        \(\frac12\le s<1\), \(0<\gamma<1\),
        \(\gamma>2s-1\), and
        \(u\in L^\infty(\mathbb R^n)\cap C^\gamma(\mathbb R^n)\), then
        \(\mathcal Q_\lambda^s u\in L^\infty(\mathbb R^n)\) and, for
        every \(\varepsilon>0\),}
        \[
        \|\mathcal{Q}^s_\lambda  u\|_{L^\infty(\mathbb{R}^n)}
        \le
        C\left(
        \varepsilon^{\gamma+1-2s}
        [u]_{C^\gamma(\mathbb{R}^n)}
        +
        \varepsilon^{-2s}
        \|u\|_{L^\infty(\mathbb{R}^n)}
        \right).
        \]

        \item[(iii)] {In particular, if
        \(\frac12\le s<1\) and \(u\in C^s(\mathbb R^n)\), then, for
        every \(\varepsilon>0\),}
          \begin{equation}
                \|\mathcal{Q}^s_\lambda u\|_{L^\infty(\mathbb{R}^n)}
          \le
            C\left(
              \varepsilon^{1-s}[u]_{C^s(\mathbb{R}^n)}
          +
              \varepsilon^{-2s}\|u\|_{L^\infty(\mathbb{R}^n)}
          \right).
          \end{equation}
    \end{itemize}
    Here $C>0$ is a constant depending only on $n,s,\lambda$ and, in case $(ii)$, also on $\gamma$.
\end{lemma}

\begin{proof}
    We first consider the case $0<s<\frac12$. Set
    \[
    K^s_\lambda (z)
    :=
    c_{n,s}
    \frac{1-e^{-\lambda |z|}}{|z|^{n+2s}}.
    \]
    We prove that $K^s_\lambda \in L^1(\mathbb{R}^n)$. Split
    \[
      \int_{\mathbb{R}^n}K^s_\lambda (z)\,dz
    =
      \int_{|z|\le 1}K^s_\lambda (z)\,dz + \int_{|z|>1}K^s_\lambda (z)\,dz .
    \]
   
    For $|z|\le 1$, using $1-e^{-\lambda |z|}\le \lambda |z|$, we have
    \[
    K^s_\lambda (z)
    \le
    c_{n,s}\lambda |z|^{-n-2s+1}.
    \]
    Hence, by polar coordinates,
    \[
    \int_{|z|\le 1}K^s_\lambda (z)\,dz
    \le
    c_{n,s}\lambda |\mathbb{S}^{n-1}|
    \int_0^1 r^{-2s}\,dr
    =
    c_{n,s}\lambda |\mathbb{S}^{n-1}|
    \frac{1}{1-2s}.
    \]
    This is finite because $s<\frac12$.

    For $|z|>1$, since $1-e^{-\lambda |z|}\le 1$, we get
    \[
    K^s_\lambda (z)
    \le
    c_{n,s}|z|^{-n-2s}.
    \]
    Thus
    \[
    \int_{|z|>1}K^s_\lambda (z)\,dz
    \le
    c_{n,s}|\mathbb{S}^{n-1}|
    \int_1^\infty r^{-1-2s}\,dr
    =
    c_{n,s}|\mathbb{S}^{n-1}|\frac{1}{2s}.
    \]
    Therefore
    \[
    \|K^s_\lambda \|_{L^1(\mathbb{R}^n)}
    \le
    C_{n,s,\lambda}<\infty .
    \]
    Consequently,
    \[
    \begin{aligned}
    |\mathcal{Q}^s_\lambda  u(x)|
    &\le
    \int_{\mathbb{R}^n}
    |u(x)-u(x+z)|K^s_\lambda (z)\,dz        \\
    &\le
    2\|u\|_{L^\infty(\mathbb{R}^n)}
    \|K^s_\lambda \|_{L^1(\mathbb{R}^n)} .
    \end{aligned}
    \]
    Taking the supremum over $x\in\mathbb{R}^n$ gives
    \[
    \|\mathcal{Q}^s_\lambda  u\|_{L^\infty(\mathbb{R}^n)}
    \le
    C\|u\|_{L^\infty(\mathbb{R}^n)}.
    \]

    Next, assume that $\frac12\le s<1$, $0<\gamma<1$, $\gamma>2s-1$, and
    $u\in C^\gamma(\mathbb{R}^n)$. Fix $\varepsilon>0$. We split
    \begin{align*}
      \mathcal{Q}^s_\lambda  u(x)
    &=
     c_{n,s}\int_{|z|\le \varepsilon}
    \frac{u(x)-u(x+z)}{|z|^{n+2s}}
    \left(1-e^{-\lambda |z|}\right)\,dz\\
    &+
      c_{n,s}\int_{|z|>\varepsilon}
      \frac{u(x)-u(x+z)}{|z|^{n+2s}}
      \left(1-e^{-\lambda |z|}\right)\,dz .
    \end{align*}
    For the near-field part, using
    $
    |u(x)-u(x+z)|
    \le
    [u]_{C^\gamma(\mathbb{R}^n)}|z|^\gamma
    $
    and
    $
    1-e^{-\lambda |z|}
    \le
    \lambda |z|,
    $
    we obtain
    \begin{equation*}
      \int_{|z|\le \varepsilon}\frac{|u(x)-u(x+z)|}{|z|^{n+2s}}
       \left(1-e^{-\lambda |z|}\right)\,dz       
    \le
      \frac{\lambda |\mathbb{S}^{n-1}|}{\gamma+1-2s}\varepsilon^{\gamma+1-2s}[u]_{C^\gamma(\mathbb{R}^n)} .
    \end{equation*}
    For the far-field part, using
    $|u(x)-u(x+z)| \le 2\|u\|_{L^\infty(\mathbb{R}^n)}$
    and
    $1-e^{-\lambda |z|}\le1,$
    we get
    
    \begin{equation*}
      \int_{|z|>\varepsilon}
      \frac{|u(x)-u(x+z)|}{|z|^{n+2s}}
      \left(1-e^{-\lambda |z|}\right)\,dz      
    \le
      \frac{|\mathbb{S}^{n-1}|}{s}
      \varepsilon^{-2s}
      \|u\|_{L^\infty(\mathbb{R}^n)} .
    \end{equation*}
    
    % Combining the two estimates and taking the supremum over
    % $x\in\mathbb{R}^n$, we obtain
    % \[
    % \|\mathcal{Q}^s_\lambda  u\|_{L^\infty(\mathbb{R}^n)}
    % \le
    % C\left(
    % \varepsilon^{\gamma+1-2s}
    % [u]_{C^\gamma(\mathbb{R}^n)}
    % +
    % \varepsilon^{-2s}
    % \|u\|_{L^\infty(\mathbb{R}^n)}
    % \right).
    % \]
    % The proof is complete.
    Combining the two estimates and taking the supremum over
    $x\in\mathbb{R}^n$, we obtain $(ii)$.

    Finally, since
    $s>2s-1$ for every $s<1$,
    the choice $\gamma=s$ is admissible in $(ii)$. Substituting
    $\gamma=s$ gives
    $\varepsilon^{\gamma+1-2s}=\varepsilon^{1-s},$
    and therefore
    \[
    \|\mathcal{Q}^s_\lambda u\|_{L^\infty(\mathbb{R}^n)}
    \le
    C\left(
    \varepsilon^{1-s}[u]_{C^s(\mathbb{R}^n)}
    +
    \varepsilon^{-2s}\|u\|_{L^\infty(\mathbb{R}^n)}
    \right).
    \]
    This proves $(iii)$ and completes the proof.
\end{proof}

\begin{lemma}\label{QCA}
    Let $n\ge 1$, $0<s<1$, $\lambda > 0 $, and let
    \[
      \mathcal{Q}^s_\lambda  u(x)
    :=
      c_{n,s}\int_{\mathbb{R}^n}\frac{u(x)-u(x+z)}{|z|^{n+2s}}\left(1-e^{-\lambda |z|}\right)\,dz .
    \]
    Then the following assertions hold.

    \begin{itemize}
        \item[(i)] If $0<s<\frac12$, $0 < \alpha <1$, and
        $u\in C^\alpha(\mathbb{R}^n)$, then
        $\mathcal{Q}^s_\lambda  u\in C^\alpha(\mathbb{R}^n)$
        and
        \[
        \|\mathcal{Q}^s_\lambda  u\|_{C^\alpha(\mathbb{R}^n)}
        \le
        C\|u\|_{C^\alpha(\mathbb{R}^n)}.
        \]

        \item[(ii)] If $s=\frac12$, $0 < \sigma < \alpha < 1$, and
        $u\in C^\alpha(\mathbb{R}^n)$, then
        $\mathcal{Q}^{1/2}_\lambda u\in C^\sigma(\mathbb{R}^n)$
        and
        \[
        \|\mathcal{Q}^{1/2}_\lambda u\|_{C^\sigma(\mathbb{R}^n)}
        \le
        C\|u\|_{C^\alpha(\mathbb{R}^n)}.
        \]

        \item[(iii)] If $\frac12<s<1$, $0 < \alpha <1$, $\alpha -2s+1 > 0$, and
        $u\in C^\alpha(\mathbb{R}^n)$, then
        $\mathcal{Q}^s_\lambda  u\in C^{\alpha+1-2s}(\mathbb{R}^n)$
        and
        \[
        \|\mathcal{Q}^s_\lambda  u\|_{C^{\alpha-2s+1}(\mathbb{R}^n)}
        \le
        C\|u\|_{C^\alpha(\mathbb{R}^n)}.
        \]
        % \item[(iv)] In particular, if $\frac12\le s<1$,
        % % and $u\in C^s(\mathbb{R}^n)$, 
        %   and
        % $u\in C^s(\mathbb{R}^n)$, then
        % $\mathcal{Q}^s_\lambda  u\in C^{1-s}(\mathbb{R}^n)$
        % and
        % \[
        % \|\mathcal{Q}^s_\lambda  u\|_{C^{1-s}(\mathbb{R}^n)}
        % \le
        % C\|u\|_{C^s(\mathbb{R}^n)}.
        % \]
    \end{itemize}
    Here $C>0$ depends only on the indicated parameters.
\end{lemma}

\begin{proof}
    We write
    \[K^s_\lambda (z):=
    c_{n,s}
    \frac{1-e^{-\lambda |z|}}{|z|^{n+2s}}.
    \]

    For $h\in\mathbb{R}^n$, define
    $\delta_h u(x):=u(x+h)-u(x).$
    Then
    \[
    \mathcal{Q}^s_\lambda  u(x+h)-\mathcal{Q}^s_\lambda  u(x)
    =
    \int_{\mathbb{R}^n}
    \left(
    \delta_h u(x)-\delta_h u(x+z)
    \right)
    K^s_\lambda (z)\,dz .
    \]
    If $u\in C^\alpha(\mathbb{R}^n)$, then
    \[
    |\delta_h u(x)-\delta_h u(x+z)|
    \le
    2[u]_{C^\alpha(\mathbb{R}^n)}|h|^\alpha
    \]
    and also
    \[
    |\delta_h u(x)-\delta_h u(x+z)|
    \le
    2[u]_{C^\alpha(\mathbb{R}^n)}|z|^\alpha .
    \]
    Hence
    \[
    |\delta_h u(x)-\delta_h u(x+z)|
    \le
    2[u]_{C^\alpha(\mathbb{R}^n)}
    \min\{|h|^\alpha,|z|^\alpha\}.
    \]

    First, assume that $0<s<\frac12$. By the proof of
    Lemma~\ref{lemmaQL}, we have $K^s_\lambda \in L^1(\mathbb{R}^n).$
    Therefore
    \[
    \begin{aligned}
    |\mathcal{Q}^s_\lambda  u(x+h)-\mathcal{Q}^s_\lambda  u(x)|
    &\le
    2[u]_{C^\alpha(\mathbb{R}^n)}
    |h|^\alpha
    \int_{\mathbb{R}^n}K^s_\lambda (z)\,dz  \\
    &\le
    C[u]_{C^\alpha(\mathbb{R}^n)}|h|^\alpha .
    \end{aligned}
    \]
    Thus
    \[
    [\mathcal{Q}^s_\lambda  u]_{C^\alpha(\mathbb{R}^n)}
    \le
    C[u]_{C^\alpha(\mathbb{R}^n)}.
    \]
    Together with Lemma~\ref{lemmaQL}, this gives
    \[
    \|\mathcal{Q}^s_\lambda  u\|_{C^\alpha(\mathbb{R}^n)}
    \le
    C\|u\|_{C^\alpha(\mathbb{R}^n)}.
    \]

    Next, assume that $s=\frac12$. Let $\alpha > \sigma > 0$ and set
   $ \rho:=|h|.$
    It is enough to consider $0<\rho<1$, since the case $\rho\ge 1$ is controlled by the $L^\infty$ estimate.

    For $|z|\le 1$, we have
    \[
    K_\lambda^{1/2}(z)
    \le
    C|z|^{-n}.
    \]

    For $|z|>1$, we have
    \[
    K_\lambda^{1/2}(z)
    \le
    C|z|^{-n-1}.
    \]
    Using the preceding minimum estimate, we split the integral into three parts:
    \[
    |z|\le \rho < 1,\qquad
    \rho<|z|\le 1,\qquad
    |z|>1.
    \]
    Then
    \[
    \begin{aligned}
    |\mathcal{Q}^{1/2}_\lambda u(x+h)-\mathcal{Q}^{1/2}_\lambda u(x)|
    &\le
    C[u]_{C^\alpha}
    \left(
    \int_{|z|\le \rho}|z|^\alpha |z|^{-n}\,dz \right. \\
    &\qquad\left.
    +
    \rho^\alpha
    \int_{\rho<|z|\le 1}|z|^{-n}\,dz
    +
    \rho^\alpha
    \int_{|z|>1}|z|^{-n-1}\,dz
    \right).
    \end{aligned}
    \]
    By polar coordinates,
    \[
    \int_{|z|\le \rho}|z|^\alpha |z|^{-n}\,dz
    =
    |\mathbb{S}^{n-1}|
    \int_0^\rho r^{\alpha-1}\,dr
    =
    C\rho^\alpha,
    \]
    and
    \[
    \int_{\rho<|z|\le 1}|z|^{-n}\,dz
    =
    |\mathbb{S}^{n-1}|
    \int_\rho^1 r^{-1}\,dr
    =
    C\log\frac1\rho.
    \]
    Moreover,
    \[
    \int_{|z|>1}|z|^{-n-1}\,dz<\infty.
    \]
    Hence
    \[
    |\mathcal{Q}^{1/2}_\lambda u(x+h)-\mathcal{Q}^{1/2}_\lambda u(x)|
    \le
    C[u]_{C^\alpha(\mathbb{R}^n)}
    \rho^\alpha\left(1+\log\frac1\rho\right).
    \]
    Since $0<\sigma<\alpha$, there exists a constant $C_{\alpha,\sigma}>0$ such that
    \[
    \rho^\alpha\left(1+\log\frac1\rho\right)
    \le
    C_{\alpha,\sigma}\rho^\sigma,
    \qquad 0<\rho<1.
    \]
    Therefore
    \[
    [\mathcal{Q}^{1/2}_\lambda u]_{C^\sigma(\mathbb{R}^n)}
    \le
    C[u]_{C^\alpha(\mathbb{R}^n)}.
    \]
    The $L^\infty$ estimate follows from
    Proposition~\ref{lemmaQL} with $\gamma=\alpha$.
    Thus
    \[
    \|\mathcal{Q}^{1/2}_\lambda u\|_{C^\sigma(\mathbb{R}^n)}
    \le
    C\|u\|_{C^\alpha(\mathbb{R}^n)}.
    \]

    Finally, assume that $\frac12<s<1$ and $\alpha > 2s-1 $. Set
   $ m:=2s-1.$
    Then $0<m<1$ and $\alpha-m>0$. Again let $\rho:=|h|$.
    It is enough to consider $0<\rho<1$.

    For $|z|\le 1$, using $1-e^{-\lambda |z|}\le \lambda |z|$, we have
    \[
    K^s_\lambda (z)
    \le
    C|z|^{-n-m}.
    \]

    For $|z|>1$, using $1-e^{-\lambda |z|}\le 1$, we have
    \[
    K^s_\lambda (z)
    \le
    C|z|^{-n-2s}.
    \]
    % \color{red}
    Therefore,
    \[
    \begin{aligned}
    |\mathcal{Q}^s_\lambda  u(x+h)-\mathcal{Q}^s_\lambda  u(x)|
    &\le
    C[u]_{C^\alpha}
    \left(
    \int_{|z|\le \rho}|z|^\alpha |z|^{-n-m}\,dz \right. \\
    &\qquad\left.
    +
    \rho^\alpha
    \int_{\rho<|z|\le 1}|z|^{-n-m}\,dz
    +
    \rho^\alpha
    \int_{|z|>1}|z|^{-n-2s}\,dz
    \right).
    \end{aligned}
    \]
    Using polar coordinates,
    \[
    \int_{|z|\le \rho}|z|^\alpha |z|^{-n-m}\,dz
    =
    |\mathbb{S}^{n-1}|
    \int_0^\rho r^{\alpha-m-1}\,dr
    =
    C\rho^{\alpha-m},
    \]
    because $\alpha>m$. Also,
    \[
    \rho^\alpha
    \int_{\rho<|z|\le 1}|z|^{-n-m}\,dz
    =
    C\rho^\alpha
    \int_\rho^1 r^{-m-1}\,dr
    \le
    C\rho^{\alpha-m}.
    \]
    Finally,
    \[
    \rho^\alpha
    \int_{|z|>1}|z|^{-n-2s}\,dz
    \le
    C\rho^\alpha
    \le
    C\rho^{\alpha-m},
    \]
    since $0<\rho<1$ and $m>0$. Combining the three estimates gives
    \[
    |\mathcal{Q}^s_\lambda  u(x+h)-\mathcal{Q}^s_\lambda  u(x)|
    \le
    C[u]_{C^\alpha(\mathbb{R}^n)}
    |h|^{\alpha+1-2s}.
    \]
    Hence
    \[
    [\mathcal{Q}^s_\lambda  u]_{C^{\alpha+1-2s}(\mathbb{R}^n)}
    \le
    C[u]_{C^\alpha(\mathbb{R}^n)}.
    \]
    The $L^\infty$ estimate follows from
    Proposition~\ref{lemmaQL} with $\gamma=\alpha$.
    Therefore
    \[
    \|\mathcal{Q}^s_\lambda  u\|_{C^{\alpha+1-2s}(\mathbb{R}^n)}
    \le
    C\|u\|_{C^\alpha(\mathbb{R}^n)}.
    \]
    The proof is complete.
\end{proof}

      \begin{lemma}\label{thm:Q-subcritical-holder}
                Let $n\ge 1$, $0<s<\frac12$, $\lambda>0$, and define
                    \[
                      \mathcal{Q}^s_\lambda  u(x)
                    :=
                      c_{n,s}\int_{\mathbb{R}^n}\frac{u(x)-u(x+z)}{|z|^{n+2s}}\left(1-e^{-\lambda |z|}\right)\,dz .
                    \]
                Then for every $\alpha>1$,
                $
                \mathcal Q_\lambda^s:
                C^\alpha(\mathbb R^n)
                \longrightarrow
                C^\alpha(\mathbb R^n)
                $
            is a bounded linear operator. More precisely,
            \[
                     \|\mathcal Q_\lambda^s u\|_{C^\alpha(\mathbb R^n)}
              \le
                  C\|u\|_{C^\alpha(\mathbb R^n)},
            \]
            where $C>0$ depends only on $n,s,\lambda,\alpha$.

            Moreover, if $\alpha=k+\theta$, where $k\in\mathbb N_0$ and
            $0<\theta\le 1$, then for every $\nu$ with $|\nu|\le k$,
            \[
                D^\nu\mathcal Q_\lambda^s u
              =
                \mathcal Q_\lambda^s(D^\nu u)
            \]
            pointwise in $\mathbb R^n$.
      \end{lemma}

    \begin{proof}
        Throughout the proof, we use the standard H\"older convention
        \[
        C^\alpha(\mathbb R^n)=C^{k,\theta}(\mathbb R^n),
        \qquad
        \alpha=k+\theta,
        \]
        where $k\in\mathbb N_0$ and $0<\theta\le1$. Thus,
        \[
        \|u\|_{C^\alpha(\mathbb R^n)}
        =
        \sum_{|\mu|\le k}\|D^\mu u\|_{L^\infty(\mathbb R^n)}
        +
        \sum_{|\mu|=k}[D^\mu u]_{C^\theta(\mathbb R^n)}.
        \]

        We divide the proof into several steps.
       
        By Lemma~\ref{lemmaQL},  we first prove that
        $
        0 < \int_{\mathbb R^n}K_\lambda^s(z)\,dz<\infty.
        $
        % denote
        % $
        %     A_{s,\lambda}:=\int_{\mathbb R^n}K_\lambda^s(z)\,dz.
        % $
      Equivalently, since $K_\lambda^s\in L^1(\mathbb R^n)$, we may write
    \begin{equation}
        \mathcal Q_\lambda^s u(x)
      =
          u(x)\int_{\mathbb R^n}K_\lambda^s(z)\,dz-\int_{\mathbb R^n}u(x+z)K_\lambda^s(z)\,dz,
    \end{equation}
    this representation will be useful below.

    We next prove that the operator $\mathcal Q_\lambda^s$ commutes with derivatives.
    Assume first that $v\in C^1(\mathbb R^n)$ and
    \[
    \|v\|_{C^1(\mathbb R^n)}<\infty.
    \]
    We prove that for each $i=1,\dots,n$,
    \begin{equation}\label{step33}
       \partial_i\mathcal Q_\lambda^s v
      =
       \mathcal Q_\lambda^s(\partial_i v)
    \end{equation}
    pointwise in $\mathbb R^n$.

    Fix $x\in\mathbb R^n$ and $t\ne0$. We consider the difference quotient
    \[
    \frac{\mathcal Q_\lambda^s v(x+t e_i)-\mathcal Q_\lambda^s v(x)}{t}.
    \]
    By definition,
    \[
    \begin{aligned}
    &\frac{\mathcal Q_\lambda^s v(x+t e_i)-\mathcal Q_\lambda^s v(x)}{t}  \\
    &=
    \int_{\mathbb R^n}
    \frac{
    v(x+t e_i)-v(x+t e_i+z)
    -
    v(x)+v(x+z)
    }{t}
    K_\lambda^s(z)\,dz.
    \end{aligned}
    \]
    Set
    \[
    B_t(x,z)
    :=
    \frac{
    v(x+t e_i)-v(x)
    -
    v(x+t e_i+z)+v(x+z)
    }{t}.
    \]
    Then
    \[
    \frac{\mathcal Q_\lambda^s v(x+t e_i)-\mathcal Q_\lambda^s v(x)}{t}
    =
    \int_{\mathbb R^n}B_t(x,z)K_\lambda^s(z)\,dz.
    \]

    By the fundamental theorem of calculus,
    \[
    \frac{v(x+t e_i)-v(x)}{t}
    =
    \int_0^1\partial_i v(x+\tau t e_i)\,d\tau,
    \]
    and similarly,
    \[
    \frac{v(x+t e_i+z)-v(x+z)}{t}
    =
    \int_0^1\partial_i v(x+z+\tau t e_i)\,d\tau.
    \]

    Therefore
    \[
    B_t(x,z)
    =
    \int_0^1
    \left[
    \partial_i v(x+\tau t e_i)
    -
    \partial_i v(x+z+\tau t e_i)
    \right]\,d\tau.
    \]

    For each fixed $z\in\mathbb R^n$, since $\partial_i v$ is continuous,
    \[
    B_t(x,z)
    \to
    \partial_i v(x)-\partial_i v(x+z)
    \qquad\text{as }t\to0.
    \]
    Moreover,
    \[
    |B_t(x,z)|
    \le
    2\|\partial_i v\|_{L^\infty(\mathbb R^n)}.
    \]
    Thus,
    \[
    |B_t(x,z)|K_\lambda^s(z)
    \le
    2\|\partial_i v\|_{L^\infty(\mathbb R^n)}K_\lambda^s(z).
    \]
    Since $K_\lambda^s\in L^1(\mathbb R^n)$, the right-hand side is integrable in $z$.

    By the dominated convergence theorem,
    \[
    \begin{aligned}
    \partial_i\mathcal Q_\lambda^s v(x)
    &=
    \lim_{t\to0}
    \frac{\mathcal Q_\lambda^s v(x+t e_i)-\mathcal Q_\lambda^s v(x)}{t} \\
    &=
    \int_{\mathbb R^n}
    \bigl(\partial_i v(x)-\partial_i v(x+z)\bigr)
    K_\lambda^s(z)\,dz \\
    &=
    \mathcal Q_\lambda^s(\partial_i v)(x).
    \end{aligned}
    \]
    This proves the first-order commutation identity.
    
    Let $\nu$ be a multi-index with $|\nu|\le k$. We next prove that
    $
      D^\nu\mathcal Q_\lambda^s u
    =
      \mathcal Q_\lambda^s(D^\nu u).
    $

    If $|\nu|=0$, there is nothing to prove. Suppose $|\nu|\ge1$. Choose indices
    $i_1,\dots,i_{|\nu|}\in\{1,\dots,n\}$ such that
    \[
    D^\nu
    =
    \partial_{i_{|\nu|}}\cdots\partial_{i_1}.
    \]
    We apply the first-order commutation result repeatedly.

    Since $u\in C^{k,\theta}(\mathbb R^n)$, every derivative $D^\mu u$ with
    $|\mu|<k$ belongs to $C^1(\mathbb R^n)$ and has bounded first derivatives.
    Therefore the proof of~\eqref{step33} can be applied successively to
    \[
    u,\quad
    \partial_{i_1}u,\quad
    \partial_{i_2}\partial_{i_1}u,
    \quad \dots,\quad
    \partial_{i_{|\nu|-1}}\cdots\partial_{i_1}u.
    \]
    Consequently,
    \[
    \begin{aligned}
    D^\nu\mathcal Q_\lambda^s u
    &=
    \partial_{i_{|\nu|}}\cdots\partial_{i_1}\mathcal Q_\lambda^s u  \\
    &=
    \partial_{i_{|\nu|}}\cdots\partial_{i_2}
    \mathcal Q_\lambda^s(\partial_{i_1}u)  \\
    &=
    \cdots  \\
    &=
    \mathcal Q_\lambda^s
    \left(
    \partial_{i_{|\nu|}}\cdots\partial_{i_1}u
    \right) \\
    &=
    \mathcal Q_\lambda^s(D^\nu u).
    \end{aligned}
    \]
    Hence, for all $|\nu|\le k$,
    \[
    D^\nu\mathcal Q_\lambda^s u
    =
    \mathcal Q_\lambda^s(D^\nu u).
    \]

    By the commutation identity and the $L^\infty$ estimate
    \eqref{QL<UL}, applied to $D^\nu u$, we have
    \[
    \begin{aligned}
      \|D^\nu\mathcal Q_\lambda^s u\|_{L^\infty(\mathbb R^n)}
    =
      \|\mathcal Q_\lambda^s(D^\nu u)\|_{L^\infty(\mathbb R^n)}
    \le
        C\|D^\nu u\|_{L^\infty(\mathbb R^n)}.
    \end{aligned}
    \]
    Summing over all $|\nu|\le k$, we get
    \begin{equation}\label{lower-derivatives-estimate}
    \sum_{|\nu|\le k}
    \|D^\nu\mathcal Q_\lambda^s u\|_{L^\infty(\mathbb R^n)}
    \le
    C
    \sum_{|\nu|\le k}
    \|D^\nu u\|_{L^\infty(\mathbb R^n)}.
    \end{equation}

    Let $|\nu|=k$. We need to estimate
    \[
    [D^\nu\mathcal Q_\lambda^s u]_{C^\theta(\mathbb R^n)}.
    \]
    By the commutation identity,
    \[
    D^\nu\mathcal Q_\lambda^s u
    =
    \mathcal Q_\lambda^s(D^\nu u).
    \]
    Set
    $
    w:=D^\nu u,
    $
    then $w\in C^\theta(\mathbb R^n)$. For any $x,y\in\mathbb R^n$,
    \[
    \begin{aligned}
      \mathcal Q_\lambda^s w(x)-\mathcal Q_\lambda^s w(y)
    =
      \int_{\mathbb R^n}
      \Big[
          w(x)-w(x+z)-w(y)+w(y+z)
      \Big]K_\lambda^s(z)\,dz.
    \end{aligned}
    \]
    Taking absolute values gives
    \[
    \begin{aligned}
    |\mathcal Q_\lambda^s w(x)-\mathcal Q_\lambda^s w(y)|
    &\le
    \int_{\mathbb R^n}
    |w(x)-w(y)|K_\lambda^s(z)\,dz  \\
    &\quad+
    \int_{\mathbb R^n}
    |w(x+z)-w(y+z)|K_\lambda^s(z)\,dz.
    \end{aligned}
    \]
    Since $w\in C^\theta(\mathbb R^n)$,
    \[
    |w(x)-w(y)|,\ |w(x+z)-w(y+z)|
    \le
    [w]_{C^\theta(\mathbb R^n)}|x-y|^\theta .
    \]
    Therefore
    \[
    \begin{aligned}
    |\mathcal Q_\lambda^s w(x)-\mathcal Q_\lambda^s w(y)|
    &\le
    2[w]_{C^\theta(\mathbb R^n)}
    |x-y|^\theta
    \int_{\mathbb R^n}K_\lambda^s(z)\,dz  \\
    &\le
    C[w]_{C^\theta(\mathbb R^n)}|x-y|^\theta.
    \end{aligned}
    \]
    Thus we obtain
    \[
    [\mathcal Q_\lambda^s w]_{C^\theta(\mathbb R^n)}
    \le
    C[w]_{C^\theta(\mathbb R^n)},
    \]
    since $w=D^\nu u$, this becomes
    \[
    [D^\nu\mathcal Q_\lambda^s u]_{C^\theta(\mathbb R^n)}
    \le
    C
    [D^\nu u]_{C^\theta(\mathbb R^n)}.
    \]

    Summing over all $|\nu|=k$, we get
    \begin{equation}\label{top-holder-estimate}
    \sum_{|\nu|=k}
    [D^\nu\mathcal Q_\lambda^s u]_{C^\theta(\mathbb R^n)}
    \le
    C
    \sum_{|\nu|=k}
    [D^\nu u]_{C^\theta(\mathbb R^n)}.
    \end{equation}

    Combining \eqref{lower-derivatives-estimate} and \eqref{top-holder-estimate}, we obtain
    \[
    \begin{aligned}
    \|\mathcal Q_\lambda^s u\|_{C^\alpha(\mathbb R^n)}
    &=
    \sum_{|\nu|\le k}
    \|D^\nu\mathcal Q_\lambda^s u\|_{L^\infty(\mathbb R^n)}
    +
    \sum_{|\nu|=k}
    [D^\nu\mathcal Q_\lambda^s u]_{C^\theta(\mathbb R^n)}  \\
    &\le
    C
    \left(
    \sum_{|\nu|\le k}
    \|D^\nu u\|_{L^\infty(\mathbb R^n)}
    +
    \sum_{|\nu|=k}
    [D^\nu u]_{C^\theta(\mathbb R^n)}
    \right) \\
    &=
    C\|u\|_{C^\alpha(\mathbb R^n)}.
    \end{aligned}
    \]
      Therefore
      $
        \mathcal Q_\lambda^s u\in C^\alpha(\mathbb R^n)
      $
      and
      $
        \|\mathcal Q_\lambda^s u\|_{C^\alpha(\mathbb R^n)}
        \le
        C\|u\|_{C^\alpha(\mathbb R^n)}.
      $
    The proof is completed.
    \end{proof}

\section{Global \texorpdfstring{$C^s$}{Cs} regularity: proof of Theorem~\ref{pro:1.1} }\label{sec:global-regularity}
% \section{Global \texorpdfstring{$C^s$}{Cs} regularity}

This section proves the \(C^s\) estimate for the tempered
Dirichlet problem and then identifies the boundary profile of its
solutions. We first establish a weak comparison principle and an
\(L^\infty\) bound by means of a barrier and the coercivity of the
tempered energy. We then combine a regularized-kernel approximation with
the classical fractional estimate and the bounds for the correction
operator \(\mathcal Q_\lambda^s\) obtained in
Section~\ref{sec:correction}. The resulting global
control also permits us to transfer the classical boundary quotient
estimate to the tempered problem.

\begin{lemma}
    Let \(w \in H^s(\mathbb{R}^n)\), $0 < s < 1 $. Then $ w^+ := \max\{w,0\} \in H^s(\mathbb{R}^n).$
    Moreover, if \(w \in \widetilde{H}^s(\Omega)\), then \(w^+ \in \widetilde{H}^s(\Omega)\).
\end{lemma}

  \begin{proof}
      Consider the mapping $T(t):=t^+=\max\{t,0\}.$
      It is a \(1\)-Lipschitz function, namely, for every \(a,b\in\mathbb{R}\),
      \[
      |T(a)-T(b)|\le |a-b|.
      \]

      Hence, if \(w\in H^s(\mathbb{R}^n)\), then
      \[
      \iint_{\mathbb{R}^n\times\mathbb{R}^n}
      \frac{|w^+(x)-w^+(y)|^2}{|x-y|^{n+2s}}\,dx\,dy
      \le
      \iint_{\mathbb{R}^n\times\mathbb{R}^n}
      \frac{|w(x)-w(y)|^2}{|x-y|^{n+2s}}\,dx\,dy
      <\infty.
      \]
      On the other hand, since \(|w^+|\le |w|\), we also have \(w^+\in L^2(\mathbb{R}^n)\). Therefore,
      \[
      w^+\in H^s(\mathbb{R}^n).
      \]

      Now assume in addition that \(w\in \widetilde{H}^s(\Omega)\). Then \(w=0\) a.e. in \(\mathbb{R}^n\setminus\Omega\). It follows immediately that
      \[
      w^+=0 \qquad \text{a.e. in } \mathbb{R}^n\setminus\Omega.
      \]
      Since we already know that \(w^+\in H^s(\mathbb{R}^n)\), we conclude that
      $w^+\in \widetilde{H}^s(\Omega).$
    
      This completes the proof.
  \end{proof}

% [Weak comparison principle for the tempered operator]
% \begin{lemma}
%     Let \(u,v\in \widetilde H^s(\Omega)\). Assume that
%     \[
%     \mathcal E_\lambda(u,\varphi)\le \mathcal E_\lambda(v,\varphi)
%     \qquad
%     \text{for all } \varphi\in \widetilde H^s(\Omega)\text{ with }\varphi\ge 0
%     \ \text{a.e. in } \Omega.
%     \]
%     Then
%     \[
%     u\le v
%     \qquad
%     \text{a.e. in } \Omega.
%     \]
% \end{lemma}

% [Weak comparison principle]
\begin{lemma}\label{Weak-comparison-principle}
  Let \(0<s<1\), \(\lambda>0\), 
  and let \(\Omega\subset\mathbb R^n\) be a bounded open set. 
  Assume that \(u,v\in H^s(\mathbb R^n)\) 
  satisfy
  \[ u\le v \qquad\text{a.e. in }\mathbb R^n\setminus\Omega, \]
  and that \[ \mathcal E_\lambda(u,\varphi) \le \mathcal E_\lambda(v,\varphi) \] 
  for every nonnegative \(\varphi\in\widetilde H^s(\Omega)\). Then 
  \[ u\le v \qquad\text{a.e. in }\Omega. \] 
\end{lemma}

% \begin{lemma}[Weak comparison principle] \label{weak-comparison-principle} Let \(0<s<1\), \(\lambda>0\), and let \(\Omega\subset\mathbb R^n\) be a bounded open set. 
%   Assume that \(u,v\in H^s(\mathbb R^n)\) satisfy \[ u\le v \qquad\text{a.e. in }\mathbb R^n\setminus\Omega, \] and \[ \mathcal E_\lambda(u,\varphi) \le \mathcal E_\lambda(v,\varphi) \] for every nonnegative \(\varphi\in\widetilde H^s(\Omega)\). Then \[ u\le v \qquad\text{a.e. in }\Omega. \] \end{lemma}

  \begin{proof}
      Set $w:=(u-v)^+$.

      By assumption, \(w\in \widetilde H^s(\Omega)\), 
      and clearly \(w\ge 0\) a.e. in \(\Omega\). Hence \(w\) is an admissible function in the hypothesis, and we obtain
      \[
      \mathcal E_\lambda(u,w)\le \mathcal E_\lambda(v,w),
      \]
      that is,
      \[
      \mathcal E_\lambda(u-v,w)\le 0.
      \]

      We now estimate the left-hand side from below. 
      % By definition,
      % \[
      % \mathcal E_\lambda(u-v,w)
      % =
      % \frac{c_{n,s}}{2}
      % \iint_{\mathbb R^n\times \mathbb R^n}
      % \frac{\bigl((u-v)(x)-(u-v)(y)\bigr)\bigl(w(x)-w(y)\bigr)}{|x-y|^{n+2s}}
      % e^{-\lambda|x-y|}
      % \,dx\,dy.
      % \]
      For brevity, let
      $a:=(u-v)(x),\,b:=(u-v)(y)$.

      Then
      $w(x)=a^+,\,w(y)=b^+.$
      We use the elementary pointwise inequality
      \[
      (a-b)(a^+-b^+)\ge |a^+-b^+|^2
      \qquad \text{for all } a,b\in \mathbb R.
      \]
      Indeed, writing
      \[
      a=a^+-a^-,\qquad b=b^+-b^-,
      \]
      we have
      \[
      a-b=(a^+-b^+)-(a^--b^-),
      \]
      and therefore
      \[
      (a-b)(a^+-b^+)
      =
      |a^+-b^+|^2-(a^--b^-)(a^+-b^+).
      \]
      Since \(a^+a^-=b^+b^-=0\), a direct expansion yields
      \[
      (a^--b^-)(a^+-b^+)=-a^-b^+-b^-a^+\le 0.
      \]
      Hence
      \[
      (a-b)(a^+-b^+)\ge |a^+-b^+|^2.
      \]

      Substituting this into the integrand, we obtain
      \[
      \bigl((u-v)(x)-(u-v)(y)\bigr)\bigl(w(x)-w(y)\bigr)
      \ge
      |w(x)-w(y)|^2.
      \]
      Consequently,
      \[
      \mathcal E_\lambda(u-v,w)
      \ge
      \frac{c_{n,s}}{2}
      \iint_{\mathbb R^n\times \mathbb R^n}
      \frac{|w(x)-w(y)|^2}{|x-y|^{n+2s}}
      e^{-\lambda|x-y|}
      \,dx\,dy
      \ge 0.
      \]

      Combining this with \(\mathcal E_\lambda(u-v,w)\le 0\), we conclude that
      \[
      \iint_{\mathbb R^n\times \mathbb R^n}
      \frac{|w(x)-w(y)|^2}{|x-y|^{n+2s}}
      e^{-\lambda|x-y|}
      \,dx\,dy
      =0.
      \]
      Since
      \[
      \frac{e^{-\lambda|x-y|}}{|x-y|^{n+2s}}>0
      \qquad \text{for all } x\neq y,
      \]
      it follows that
      \[
      w(x)-w(y)=0
      \qquad \text{for a.e. } (x,y)\in \mathbb R^n\times \mathbb R^n.
      \]
      Therefore \(w\) is a.e. equal to a constant in \(\mathbb R^n\). On the other hand, \(w\in \widetilde H^s(\Omega)\), so
      \[
      w=0
      \qquad \text{a.e. in } \mathbb R^n\setminus \Omega.
      \]
      Hence this constant must be zero, and thus
      \[
      w=(u-v)^+=0
      \qquad \text{a.e. in } \mathbb R^n.
      \]
      Equivalently,
      \[
      u\le v
      \qquad \text{a.e. in } \Omega.
      \]
      This completes the proof.
  \end{proof}

 % u L无穷估计
% \begin{lemma}\label{lem:uLfL}
%     Let $0 < s < 1$ , $\lambda > 0$ and $\Omega \subset \mathbb{R}^n$ be a bounded Lipschitz domain satisfying the exterior ball condition and let $f \in L^\infty(\Omega)$. Let $u$ be the solution of 
%     \eqref{eq:1.1}. Then, $u \in L^\infty(\Omega)$ and
%     \begin{equation}\label{ULWUQIONG}
%     \|u\|_{L^\infty(\mathbb{R}^n)} 
%     \leq C
%     % \,e^{\lambda (\mathrm{diam}\,\Omega)} 
%     (\mathrm{diam}\,\Omega)^{2s} \|f\|_{L^\infty(\Omega)}, 
%     \end{equation}
%     where $C=C(n,s,\lambda,\Omega)$ is a constant depending only on $n$, $s$, $\lambda$ and $\Omega$.
% \end{lemma}

\begin{lemma}\label{lem:uLfL} 
  Let \(0<s<1\), \(\lambda>0\), 
  and let \(\Omega\subset\mathbb R^n\) be a bounded open set. 
  Let \(f\in L^\infty(\Omega)\), 
  and let \(u\in\widetilde H^s(\Omega)\) be a weak solution of \eqref{eq:1.1}. 
  Then $ u\in L^\infty(\mathbb R^n)$, and 
  \begin{equation}\label{ULWUQIONG} 
    \|u\|_{L^\infty(\mathbb R^n)} 
    \le 
    C_{n,s} e^{5\lambda \operatorname{diam}\Omega} \operatorname{diam}\Omega^{2s} \|f\|_{L^\infty(\Omega)},
  \end{equation} 
  where \(C_{n,s}>0\) depends only on \(n\) and \(s\). 
\end{lemma}

\begin{proof}
    Fix \(x_0\in\Omega\), and set
    $
        D:=\operatorname{diam}(\Omega).
    $
    Since \(x_0\in\Omega\), we have
    \[
        \Omega\subset \overline B_D(x_0).
    \]

    Choose a function \(\eta\in C_c^\infty(\mathbb R^n)\) such that
    \[
        0\le\eta\le1,
        \qquad
        \eta\equiv1
        \quad\text{in }B_{2D}(x_0),
        \qquad
        \eta\equiv0
        \quad\text{in }\mathbb R^n\setminus B_{3D}(x_0).
    \]
    In particular,
    \[
        \eta\equiv1
        \qquad\text{in a neighborhood of }\overline\Omega.
    \]

    We first establish a strictly positive lower bound for
    \((-\Delta)^s_\lambda\eta\) in \(\Omega\). Let \(x\in\Omega\).
    Since \(\eta(x)=1\) and \(0\le\eta\le1\), we have
    \[
        \eta(x)-\eta(y)=1-\eta(y)\ge0
        \qquad\text{for every }y\in\mathbb R^n.
    \]
    Hence
    \[
        (-\Delta)^s_\lambda\eta(x)
        =
        c_{n,s}
        \int_{\mathbb R^n}
        \frac{1-\eta(y)}
        {|x-y|^{n+2s}}
        e^{-\lambda|x-y|}
        \,dy.
    \]

    Define the annulus
    \[
        A:=
        B_{4D}(x_0)\setminus B_{3D}(x_0).
    \]
    Since \(\eta\equiv0\) in
    \(\mathbb R^n\setminus B_{3D}(x_0)\), we have
    \[
        \eta(y)=0
        \qquad\text{for every }y\in A.
    \]
    Therefore,
    \[
    \begin{aligned}
        (-\Delta)^s_\lambda\eta(x)
        &\ge
        c_{n,s}
        \int_A
        \frac{e^{-\lambda|x-y|}}
        {|x-y|^{n+2s}}
        \,dy.
    \end{aligned}
    \]

    For \(x\in\Omega\subset\overline B_D(x_0)\) and \(y\in A\),
    the triangle inequality gives
    \[
        |x-y|
        \le
        |x-x_0|+|y-x_0|
        \le
        D+4D
        =
        5D.
    \]
    Consequently,
    \[
        e^{-\lambda|x-y|}
        \ge
        e^{-5\lambda D},
    \]
    and
    \[
        \frac{1}{|x-y|^{n+2s}}
        \ge
        \frac{1}{(5D)^{n+2s}}.
    \]
    It follows that
    \[
    \begin{aligned}
        (-\Delta)^s_\lambda\eta(x)
        &\ge
        c_{n,s}
        \frac{e^{-5\lambda D}}
        {(5D)^{n+2s}}
        |A|.
    \end{aligned}
    \]
    Since
    \[
        |A|
        =
        |B_1|
        \bigl(4^n-3^n\bigr)D^n,
    \]
    we obtain
    \[
        (-\Delta)^s_\lambda\eta(x)
        \ge
        \kappa_{n,s}
        e^{-5\lambda D}
        D^{-2s}
        \qquad\text{for every }x\in\Omega,
    \]
    where
    \[
        \kappa_{n,s}
        :=
        c_{n,s}|B_1|
        \bigl(4^n-3^n\bigr)
        5^{-n-2s}
        >0.
    \]
    Set
    \[
        m_0:=
        \kappa_{n,s}
        e^{-5\lambda D}
        D^{-2s}.
    \]
    Thus
    \begin{equation}\label{eta-positive-lower-bound}
        (-\Delta)^s_\lambda\eta(x)
        \ge m_0
        \qquad\text{for every }x\in\Omega.
    \end{equation}

    We now define
    \[
        M:=
        \frac{\|f\|_{L^\infty(\Omega)}}{m_0},
        \qquad
        V:=M\eta.
    \]
    Since \(V\in C_c^\infty(\mathbb R^n)\), we have
    \(V\in H^s(\mathbb R^n)\). Moreover, by
    \eqref{eta-positive-lower-bound},
    \[
        (-\Delta)^s_\lambda V
        =
        M(-\Delta)^s_\lambda\eta
        \ge
        Mm_0
        =
        \|f\|_{L^\infty(\Omega)}
        \qquad\text{in }\Omega.
    \]

    Let \(\varphi\in\widetilde H^s(\Omega)\) be nonnegative.
    By the bilinear identity, first for smooth test functions and then
    by density,
    \[
    \begin{aligned}
        \mathcal E_\lambda(V,\varphi)
        &=
        \int_{\mathbb R^n}
        (-\Delta)^s_\lambda V(x)\varphi(x)\,dx
        \\
        &=
        \int_\Omega
        (-\Delta)^s_\lambda V(x)\varphi(x)\,dx
        \\
        &\ge
        \|f\|_{L^\infty(\Omega)}
        \int_\Omega\varphi(x)\,dx.
    \end{aligned}
    \]
    On the other hand, since \(u\) is a weak solution,
    \[
    \begin{aligned}
        \mathcal E_\lambda(u,\varphi)
        &=
        \int_\Omega f(x)\varphi(x)\,dx
        \\
        &\le
        \|f\|_{L^\infty(\Omega)}
        \int_\Omega\varphi(x)\,dx.
    \end{aligned}
    \]
    Therefore,
    \[
        \mathcal E_\lambda(u,\varphi)
        \le
        \mathcal E_\lambda(V,\varphi)
    \]
    for every nonnegative
    \(\varphi\in\widetilde H^s(\Omega)\).

    Furthermore, since
    \[
        u=0
        \qquad\text{a.e. in }\mathbb R^n\setminus\Omega
    \]
    and \(V\ge0\) in \(\mathbb R^n\), we have
    \[
        u\le V
        \qquad\text{a.e. in }\mathbb R^n\setminus\Omega.
    \]
    Hence, by the weak comparison principle,
    \[
        u\le V
        \qquad\text{a.e. in }\Omega.
    \]

    Applying the same argument to \(-u\), we also obtain
    \[
        -u\le V
        \qquad\text{a.e. in }\Omega.
    \]
    Consequently,
    \[
        |u|\le V
        \qquad\text{a.e. in }\Omega.
    \]

    Since \(\eta\equiv1\) in \(\Omega\), we have
    \[
        V=M
        \qquad\text{in }\Omega.
    \]
    Thus
    \[
        |u(x)|\le M
        \qquad\text{for a.e. }x\in\Omega.
    \]
    Because \(u=0\) almost everywhere in
    \(\mathbb R^n\setminus\Omega\), it follows that
    \[
        \|u\|_{L^\infty(\mathbb R^n)}
        \le M.
    \]
    By the definition of \(M\) and \(m_0\),
    \[
    \begin{aligned}
        M
        &=
        \kappa_{n,s}^{-1}
        e^{5\lambda D}
        D^{2s}
        \|f\|_{L^\infty(\Omega)}.
    \end{aligned}
    \]
    Therefore,
    \[
        \|u\|_{L^\infty(\mathbb R^n)}
        \le
        C_{n,s}
        e^{5\lambda D}
        D^{2s}
        \|f\|_{L^\infty(\Omega)},
    \]
    where
    \[
        C_{n,s}
        :=
        \kappa_{n,s}^{-1}
        =
        \frac{5^{n+2s}}
        {c_{n,s}|B_1|(4^n-3^n)}.
    \]
    This completes the proof.
\end{proof}

\begin{coercivitylemma}[Coercivity of the tempered form]
Let \(0<s<1\), \(\lambda>0\), and let
\(\Omega\subset\mathbb R^n\) be a bounded open set. There exist
constants \(c,C>0\), depending only on \(n,s,\lambda\), and \(\Omega\),
such that
\[
c\|v\|_{H^s(\mathbb R^n)}^2
\le \mathcal E_\lambda(v,v)
\le C\|v\|_{H^s(\mathbb R^n)}^2
\qquad
\text{for every }v\in\widetilde H^s(\Omega).
\]
Consequently, \(\mathcal E_\lambda\) is coercive on
\(\widetilde H^s(\Omega)\), and the weak Dirichlet problem has at most
one solution.
\end{coercivitylemma}

\begin{proof}
Choose \(x_0\in\Omega\) and \(R>0\) such that
\(\Omega\subset B_R(x_0)\), and set
\(A:=B_{3R}(x_0)\setminus B_{2R}(x_0)\). Since
\(A\subset\mathbb R^n\setminus\Omega\), the mixed interaction between
\(\Omega\) and \(A\) gives
\[
\mathcal E_\lambda(v,v)
\ge c_0\|v\|_{L^2(\mathbb R^n)}^2,
\]
where \(c_0>0\) depends only on \(n,s,\lambda\), and \(\Omega\). Indeed,
for \(x\in\Omega\) and \(y\in A\), one has
\(R\le |x-y|\le4R\) and
\(e^{-\lambda|x-y|}\ge e^{-4\lambda R}\).
Splitting the Gagliardo seminorm into the regions
\(|x-y|\le4R\) and \(|x-y|>4R\), the first part is controlled by
\(e^{4\lambda R}\mathcal E_\lambda(v,v)\), while the second is bounded
by \(CR^{-2s}\|v\|_{L^2(\mathbb R^n)}^2\). This proves the lower bound;
the upper bound follows from \(e^{-\lambda|x-y|}\le1\). Uniqueness is
an immediate consequence of coercivity.
\end{proof}

\begin{proposition}

{\cite[Proposition~1.1]{RosOtonSerra2014}}
    Let \(\Omega\) be a bounded Lipschitz domain satisfying the exterior ball condition, \(g \in L^\infty(\Omega)\),
    and let \(u\in\widetilde H^s(\Omega)\) be a weak solution of
    \[
    \begin{cases}
    (-\Delta)^s u=g & \text{in }\Omega,\\
    u=0 & \text{in }\mathbb R^n\setminus\Omega.
    \end{cases}
    \]
    Then,
    $u \in C^s(\mathbb{R}^n)$
    and
    \[
    \|u\|_{C^s(\mathbb{R}^n)} \le C \|g\|_{L^\infty(\Omega)},
    \]
    where \(C\) is a constant depending only on \(\Omega\) and \(s\).
\end{proposition}

\begin{proof}[Proof of Theorem~\ref{pro:1.1}]
    For every \(\varepsilon\in(0,1)\), define
    \[
        a_{\lambda,\varepsilon}(z)
        :=
        e^{-\lambda|z|}
        +
        \bigl(1-e^{-\lambda|z|}\bigr)
        \mathbf 1_{\{|z|\le\varepsilon\}}.
    \]
    Equivalently,
    \[
        a_{\lambda,\varepsilon}(z)
        =
        \begin{cases}
            1, & |z|\le\varepsilon,\\
            e^{-\lambda|z|}, & |z|>\varepsilon.
        \end{cases}
    \]
    We introduce the approximating bilinear form
    \[
    \begin{aligned}
        \mathcal E_{\lambda,\varepsilon}(v,\varphi)
        &:=
        \frac{c_{n,s}}{2}
        \iint_{\mathbb R^n\times\mathbb R^n}
        \frac{
        (v(x)-v(y))(\varphi(x)-\varphi(y))
        }{|x-y|^{n+2s}}
        \\
        &\hspace{3cm}\times
        a_{\lambda,\varepsilon}(x-y)
        \,dx\,dy.
    \end{aligned}
    \]
    Since
    \[
        e^{-\lambda|z|}
        \le
        a_{\lambda,\varepsilon}(z)
        \le1,
    \]
    we have
    \[
        \mathcal E_\lambda(v,v)
        \le
        \mathcal E_{\lambda,\varepsilon}(v,v)
        \le
        \mathcal E_0(v,v)
    \]
    for every \(v\in\widetilde H^s(\Omega)\), where
    \(\mathcal E_0\) denotes the bilinear form of the classical
    fractional Laplacian.

    By the coercivity estimate above, the form
    \(\mathcal E_{\lambda,\varepsilon}\) is coercive with a coercivity
    constant independent of \(\varepsilon\); it is also continuous
    because \(a_{\lambda,\varepsilon}\le1\). Since
    \(f\in L^\infty(\Omega)\) defines a continuous functional on
    \(\widetilde H^s(\Omega)\), the Lax--Milgram theorem gives a unique
    \(u_\varepsilon\in\widetilde H^s(\Omega)\) such that
    \begin{equation}\label{approximate-weak-equation}
        \mathcal E_{\lambda,\varepsilon}
        (u_\varepsilon,\varphi)
        =
        \int_\Omega f\varphi\,dx
    \end{equation}
    for every
    \(\varphi\in\widetilde H^s(\Omega)\).

    We first establish the uniform estimate
    \begin{equation}\label{uniform-Linfty-approx}
        \|u_\varepsilon\|_{L^\infty(\mathbb R^n)}
        \le
        C_0\|f\|_{L^\infty(\Omega)},
    \end{equation}
    where \(C_0=C_0(n,s,\lambda,\Omega)\) is independent of
    \(\varepsilon\).

    Indeed, the smooth barrier constructed in
    Lemma~\ref{lem:uLfL} also applies to the approximating form.
    More precisely, let \(V=M\eta\) be the barrier used there, with
    \[
        V=M\quad\text{in }\Omega,
        \qquad
        0\le V\le M\quad\text{in }\mathbb R^n.
    \]
    For every nonnegative
    \(\varphi\in\widetilde H^s(\Omega)\), using that \(V=M\) in
    \(\Omega\) and that \(\varphi=0\) in \(\mathbb R^n\setminus\Omega\),
    the mixed interaction terms give
    \[
    \begin{aligned}
    &\mathcal E_{\lambda,\varepsilon}(V,\varphi)
    -\mathcal E_\lambda(V,\varphi)\\
    &\quad=
    c_{n,s}\int_\Omega\varphi(x)
    \int_{\mathbb R^n\setminus\Omega}
    \frac{M-V(y)}{|x-y|^{n+2s}}
    \bigl(a_{\lambda,\varepsilon}(x-y)
    -e^{-\lambda|x-y|}\bigr)\,dy\,dx
    \ge0.
    \end{aligned}
    \]
    Therefore the same choice of \(M\) as in
    Lemma~\ref{lem:uLfL} gives
    \[
        \mathcal E_{\lambda,\varepsilon}(V,\varphi)
        \ge
        \|f\|_{L^\infty(\Omega)}
        \int_\Omega\varphi\,dx.
    \]
    The proof of Lemma~\ref{Weak-comparison-principle} applies verbatim
    to \(\mathcal E_{\lambda,\varepsilon}\), since its kernel is
    symmetric and strictly positive. Applying this comparison principle first to
    \(u_\varepsilon\) and \(V\), and then to
    \(-u_\varepsilon\) and \(V\), yields
    \[
        |u_\varepsilon|\le V
        \qquad\text{a.e. in }\Omega.
    \]
    Hence \eqref{uniform-Linfty-approx} follows, with a constant
    independent of \(\varepsilon\).

    For fixed \(\varepsilon>0\), we next prove that
    \(u_\varepsilon\in C^s(\mathbb R^n)\). Define
    \[
        \mathcal Q_{\lambda,\varepsilon}^sv(x)
        :=
        c_{n,s}
        \int_{|z|>\varepsilon}
        \frac{v(x)-v(x+z)}
        {|z|^{n+2s}}
        \bigl(1-e^{-\lambda|z|}\bigr)
        \,dz.
    \]
    For fixed \(\varepsilon>0\), the kernel in this integral belongs
    to \(L^1(\mathbb R^n)\). Therefore
    \[
    \begin{aligned}
        \|\mathcal Q_{\lambda,\varepsilon}^sv\|_{L^\infty(\mathbb R^n)}
        &\le
        2c_{n,s}
        \|v\|_{L^\infty(\mathbb R^n)}
        \int_{|z|>\varepsilon}
        \frac{1-e^{-\lambda|z|}}
        {|z|^{n+2s}}
        \,dz
        \\
        &<\infty
    \end{aligned}
    \]
    whenever \(v\in L^\infty(\mathbb R^n)\).

    Moreover, if
    \(v\in H^s(\mathbb R^n)\cap L^\infty(\mathbb R^n)\), then for every
    \(\varphi\in\widetilde H^s(\Omega)\), symmetrization gives
    \[
        \mathcal E_0(v,\varphi)
        =
        \mathcal E_{\lambda,\varepsilon}(v,\varphi)
        +
        \int_{\mathbb R^n}
        \mathcal Q_{\lambda,\varepsilon}^sv(x)
        \varphi(x)\,dx.
    \]
    Applying this identity to \(v=u_\varepsilon\), and using
    \eqref{approximate-weak-equation}, we obtain
    \[
        \mathcal E_0(u_\varepsilon,\varphi)
        =
        \int_\Omega
        \left(
        f+\mathcal Q_{\lambda,\varepsilon}^su_\varepsilon
        \right)\varphi\,dx.
    \]
    Thus \(u_\varepsilon\) is a weak solution of
    \[
    \begin{cases}
        (-\Delta)^su_\varepsilon
        =
        f+\mathcal Q_{\lambda,\varepsilon}^su_\varepsilon
        &\text{in }\Omega,\\
        u_\varepsilon=0
        &\text{in }\mathbb R^n\setminus\Omega.
    \end{cases}
    \]
    Since the right-hand side belongs to \(L^\infty(\Omega)\), the
    classical global \(C^s\)-estimate gives
    \begin{equation}\label{classical-Cs-approx}
    \begin{aligned}
        \|u_\varepsilon\|_{C^s(\mathbb R^n)}
        &\le
        C_1
        \left(
        \|f\|_{L^\infty(\Omega)}
        +
        \|\mathcal Q_{\lambda,\varepsilon}^su_\varepsilon
        \|_{L^\infty(\mathbb R^n)}
        \right),
    \end{aligned}
    \end{equation}
    where \(C_1=C_1(n,s,\Omega)\) is independent of
    \(\varepsilon\).

    Notice that no circular argument occurs here: for fixed
    \(\varepsilon>0\), the boundedness of
    \(\mathcal Q_{\lambda,\varepsilon}^su_\varepsilon\) follows only
    from the already established \(L^\infty\)-bound, and not from
    \(C^s\)-regularity.

    We now derive a bound for the truncated correction operator that is uniform
    in \(\varepsilon\). Let \(v\in C^s(\mathbb R^n)\), and let
    \(\rho\in(0,1)\). We claim that
    \begin{equation}\label{uniform-truncated-Q}
        \|\mathcal Q_{\lambda,\varepsilon}^sv
        \|_{L^\infty(\mathbb R^n)}
        \le
        C_2
        \left(
        \rho^{1-s}[v]_{C^s(\mathbb R^n)}
        +
        \rho^{-2s}\|v\|_{L^\infty(\mathbb R^n)}
        \right),
    \end{equation}
    where \(C_2=C_2(n,s,\lambda)\) is independent of
    \(\varepsilon\).

    Fix \(x\in\mathbb R^n\). For the part with
    \(\varepsilon<|z|\le\rho\), using
    \[
        |v(x)-v(x+z)|
        \le
        [v]_{C^s(\mathbb R^n)}|z|^s
    \]
    and
    \[
        1-e^{-\lambda|z|}
        \le
        \lambda|z|,
    \]
    we obtain
    \[
    \begin{aligned}
        &c_{n,s}
        \int_{\{\varepsilon<|z|\le\rho\}}
        \frac{|v(x)-v(x+z)|}
        {|z|^{n+2s}}
        \bigl(1-e^{-\lambda|z|}\bigr)
        \,dz
        \\
        &\qquad\le
        c_{n,s}\lambda
        [v]_{C^s(\mathbb R^n)}
        \int_{|z|\le\rho}
        |z|^{-n-s+1}\,dz
        \\
        &\qquad=
        \frac{
        c_{n,s}\lambda|\mathbb S^{n-1}|
        }{1-s}
        \rho^{1-s}
        [v]_{C^s(\mathbb R^n)}.
    \end{aligned}
    \]
    If \(\varepsilon\ge\rho\), this part is empty, and the same bound
    remains valid.

    For the remaining part, using
    \[
        |v(x)-v(x+z)|
        \le
        2\|v\|_{L^\infty(\mathbb R^n)}
    \]
    and \(1-e^{-\lambda|z|}\le1\), we obtain
    \[
    \begin{aligned}
        &c_{n,s}
        \int_{\{|z|>\max\{\varepsilon,\rho\}\}}
        \frac{|v(x)-v(x+z)|}
        {|z|^{n+2s}}
        \bigl(1-e^{-\lambda|z|}\bigr)
        \,dz
        \\
        &\qquad\le
        2c_{n,s}
        \|v\|_{L^\infty(\mathbb R^n)}
        \int_{|z|>\rho}|z|^{-n-2s}\,dz
        \\
        &\qquad=
        \frac{
        c_{n,s}|\mathbb S^{n-1}|
        }{s}
        \rho^{-2s}
        \|v\|_{L^\infty(\mathbb R^n)}.
    \end{aligned}
    \]
    Combining the two estimates and taking the supremum over \(x\)
    proves \eqref{uniform-truncated-Q}.

    Applying \eqref{uniform-truncated-Q} to \(v=u_\varepsilon\) and
    inserting it into \eqref{classical-Cs-approx}, we obtain
    \[
    \begin{aligned}
        \|u_\varepsilon\|_{C^s(\mathbb R^n)}
        &\le
        C_1\|f\|_{L^\infty(\Omega)}
        +
        C_1C_2\rho^{1-s}
        \|u_\varepsilon\|_{C^s(\mathbb R^n)}
        \\
        &\quad+
        C_1C_2\rho^{-2s}
        \|u_\varepsilon\|_{L^\infty(\mathbb R^n)}.
    \end{aligned}
    \]
    Choose \(\rho_0\in(0,1)\), depending only on
    \(n,s,\lambda,\Omega\), such that
    \[
        C_1C_2\rho_0^{1-s}
        \le\frac12.
    \]
    Then
    \[
    \begin{aligned}
        \frac12
        \|u_\varepsilon\|_{C^s(\mathbb R^n)}
        &\le
        C_1\|f\|_{L^\infty(\Omega)}
        +
        C_1C_2\rho_0^{-2s}
        \|u_\varepsilon\|_{L^\infty(\mathbb R^n)}.
    \end{aligned}
    \]
    Using the uniform \(L^\infty\)-estimate
    \eqref{uniform-Linfty-approx}, we conclude that
    \begin{equation}\label{uniform-Cs-approx}
        \|u_\varepsilon\|_{C^s(\mathbb R^n)}
        \le
        C_3\|f\|_{L^\infty(\Omega)},
    \end{equation}
    where
    \[
        C_3=C_3(n,s,\lambda,\Omega)
    \]
    is independent of \(\varepsilon\).

    It remains to pass to the limit as \(\varepsilon\to0\). Testing
    \eqref{approximate-weak-equation} with
    \(\varphi=u_\varepsilon\), and using the uniform coercivity of
    \(\mathcal E_{\lambda,\varepsilon}\), we obtain
    \[
        \|u_\varepsilon\|_{\widetilde H^s(\Omega)}
        \le
        C\|f\|_{L^\infty(\Omega)}
    \]
    with \(C\) independent of \(\varepsilon\).
    Hence, along a sequence \(\varepsilon_k\to0\),
    \[
        u_{\varepsilon_k}
        \rightharpoonup \widetilde u
        \qquad\text{weakly in }\widetilde H^s(\Omega).
    \]
    By the compact embedding of
    \(\widetilde H^s(\Omega)\) into \(L^2(\Omega)\),
    after passing to a further subsequence,
    \[
        u_{\varepsilon_k}
        \longrightarrow \widetilde u
        \qquad\text{strongly in }L^2(\Omega).
    \]

    We next show that \(\widetilde u\) solves the original tempered
    problem. For
    \(\varphi\in\widetilde H^s(\Omega)\), write
    \[
    \begin{aligned}
        \mathcal E_{\lambda,\varepsilon_k}
        (u_{\varepsilon_k},\varphi)
        -
        \mathcal E_\lambda
        (u_{\varepsilon_k},\varphi)
        &=
        \frac{c_{n,s}}2
        \iint_{\{|x-y|\le\varepsilon_k\}}
        \frac{
        (u_{\varepsilon_k}(x)-u_{\varepsilon_k}(y))
        (\varphi(x)-\varphi(y))
        }
        {|x-y|^{n+2s}}
        \\
        &\hspace{2cm}\times
        \bigl(1-e^{-\lambda|x-y|}\bigr)
        \,dx\,dy.
    \end{aligned}
    \]
    By the Cauchy--Schwarz inequality and
    \[
        1-e^{-\lambda|x-y|}
        \le
        \lambda|x-y|
        \le
        \lambda\varepsilon_k
        \qquad\text{when }|x-y|\le\varepsilon_k,
    \]
    we have
    \[
    \begin{aligned}
        &\left|
        \mathcal E_{\lambda,\varepsilon_k}
        (u_{\varepsilon_k},\varphi)
        -
        \mathcal E_\lambda
        (u_{\varepsilon_k},\varphi)
        \right|
        \\
        &\qquad\le
        C\lambda\varepsilon_k
        \|u_{\varepsilon_k}\|_{H^s(\mathbb R^n)}
        \|\varphi\|_{H^s(\mathbb R^n)}
        \longrightarrow0.
    \end{aligned}
    \]
    Since
    \(u_{\varepsilon_k}\rightharpoonup\widetilde u\) weakly in
    \(H^s(\mathbb R^n)\), the continuity of
    \(\mathcal E_\lambda\) gives
    \[
        \mathcal E_\lambda
        (u_{\varepsilon_k},\varphi)
        \longrightarrow
        \mathcal E_\lambda(\widetilde u,\varphi).
    \]
    Passing to the limit in
    \eqref{approximate-weak-equation}, we therefore obtain
    \[
        \mathcal E_\lambda(\widetilde u,\varphi)
        =
        \int_\Omega f\varphi\,dx
    \]
    for every
    \(\varphi\in\widetilde H^s(\Omega)\). Thus \(\widetilde u\) is a
    weak solution of the original tempered problem.

    The uniqueness consequence of the coercivity estimate above yields
    \[
        \widetilde u=u.
    \]

    Finally, by \eqref{uniform-Cs-approx}, the family
    \(\{u_{\varepsilon_k}\}\) is uniformly bounded and uniformly
    \(s\)-H\"older continuous in \(\mathbb R^n\). Taking their
    continuous representatives, they vanish outside the fixed compact
    set \(\overline\Omega\). The Arzel\`a--Ascoli theorem therefore
    yields, after passing to a subsequence, uniform convergence to some
    \(w\in C^s(\mathbb R^n)\). Since the same subsequence converges
    strongly to \(u\) in \(L^2(\Omega)\), we have \(w=u\) a.e.; taking
    \(w\) as the continuous representative of \(u\), we obtain
    \[
        u_{\varepsilon_k}\longrightarrow u
        \qquad\text{uniformly in }\mathbb R^n.
    \]
    For all \(x,y\in\mathbb R^n\),
    \[
    \begin{aligned}
        |u(x)-u(y)|
        &=
        \lim_{k\to\infty}
        |u_{\varepsilon_k}(x)-u_{\varepsilon_k}(y)|
        \\
        &\le
        C_3\|f\|_{L^\infty(\Omega)}
        |x-y|^s.
    \end{aligned}
    \]
    Also,
    \[
        \|u\|_{L^\infty(\mathbb R^n)}
        \le
        C_3\|f\|_{L^\infty(\Omega)}.
    \]
    Consequently,
    \[
        u\in C^s(\mathbb R^n)
    \]
    and
    \[
        \|u\|_{C^s(\mathbb R^n)}
        \le
        C\|f\|_{L^\infty(\Omega)},
    \]
    where \(C=C(n,s,\lambda,\Omega)>0\).
    This completes the proof.
\end{proof}

{
\subsection{Boundary regularity of the quotient
\texorpdfstring{$u/\delta^s$}{u over delta to the s}: proof of Corollary~\ref{cor:tempered-boundary-quotient}}

With the global \(C^s\) estimate at hand, we now turn to the boundary
regularity of the normalized quotient \(u/\delta^s\).
The local antecedent is Krylov's boundary regularity theory \cite{17};
see also the expositions in \cite{15,6}.
Here we use its fractional counterpart from~\cite{RosOtonSerra2014}.

\begin{proof}[Proof of Corollary~\ref{cor:tempered-boundary-quotient}]
By Theorem~\ref{pro:1.1},
\[
\|u\|_{C^s(\mathbb R^n)}
\le C\|f\|_{L^\infty(\Omega)}.
\]
If \(0<s<\frac12\), Lemma~\ref{lemmaQL}(i) applies, whereas if
\(\frac12\le s<1\), Lemma~\ref{lemmaQL}(iii), with
\(\varepsilon=1\), gives
\[
\|\mathcal Q_\lambda^s u\|_{L^\infty(\mathbb R^n)}
\le C\|u\|_{C^s(\mathbb R^n)}
\le C\|f\|_{L^\infty(\Omega)}.
\]
Set \(g:=f+\mathcal Q_\lambda^s u\). By~\eqref{eqQ}, \(u\) solves
\[
(-\Delta)^s u=g\quad\text{in }\Omega,
\qquad
u=0\quad\text{in }\mathbb R^n\setminus\Omega,
\]
and
\(\|g\|_{L^\infty(\Omega)}
\le C\|f\|_{L^\infty(\Omega)}\).
The boundary regularity theorem for the classical fractional
Laplacian~\cite{RosOtonSerra2014} therefore yields some
\(\alpha\in(0,\min\{s,1-s\})\) for which
\[
\frac{u}{\delta^s}\in C^\alpha(\overline\Omega),
\qquad
\left\|\frac{u}{\delta^s}\right\|_{C^\alpha(\overline\Omega)}
\le C\|f\|_{L^\infty(\Omega)}.
\]
\end{proof}
}

% \begin{lemma}\label{lem:RO2.8}
%   [Ros-Oton, Claim 2.8]Let $\Omega$ be a bounded domain and let $g \in L^\infty(\Omega)$. Let $u$ be the solution of 
%   \eqref{eq:1.3}. Then,
%   \[
%   \|u\|_{L^\infty(\mathbb{R}^n)} 
%   \leq C
%   (\mathrm{diam}\,\Omega)^{2s} \|g\|_{L^\infty(\Omega)},
%   \]
%   where $C$ is a constant depending only on $n$ and $s$.
% \end{lemma}

% \section{Proof of Theorem~\ref{Pro:1.4}}
\section{Interior Schauder estimates: proof of Theorem~\ref{Pro:1.4}}\label{sec:schauder}

This section develops the Schauder estimates required for higher-order
regularity. Since exponential tempering does not alter the order of the
singularity near the diagonal, we first establish the basic H\"older
mapping properties and transfer the classical interior estimate through
the decomposition involving \(\mathcal Q_\lambda^s\). Cutoff and tail
estimates then give localized bounds on nested balls, which are rescaled
to obtain the weighted estimate used in
Section~\ref{sec:higher-order}.

\begin{lemma}\label{prop:2.5}
  Let $0 < s < 1$, $u\in C^{0,\alpha}(\mathbb R^n)$ for some $\alpha\in(0,1]$ and assume $\alpha>2s>0$.
  Then $(-\Delta)^s_\lambda u\in C^{0,\alpha-2s}(\mathbb R^n)$ and
  \begin{equation}\label{eq:prop2.5-est}
  \bigl[(-\Delta)^s_\lambda u\bigr]_{C^{0,\alpha-2s}(\mathbb R^n)}
  \le C\,[u]_{C^{0,\alpha}(\mathbb R^n)},
  \end{equation}
  where $C$ depends only on $\alpha$, $s$, and $n$.
\end{lemma}

  \begin{proof}
      Let us estimate the difference $\bigl|(-\Delta)^s_\lambda u(x_1)-(-\Delta)^s_\lambda u(x_2)\bigr|$ for
      $x_1,x_2\in\mathbb R^n$. Using the singular integral representation, we write
      \begin{align*}
      \bigl|(-\Delta)^s_\lambda u(x_1)-(-\Delta)^s_\lambda u(x_2)\bigr|
      % &=
      %   C_{n,s}\left|
      %   \int_{\mathbb R^n}
      %   \frac{u(x_1)-u(x_1+y)-u(x_2)+u(x_2+y)}{|y|^{n+2s}}e^{-\lambda|y|}\,dy
      %   \right|
      % \\
      % &\le
      %     C_{n,s}\int_{\mathbb R^n}
      %     \frac{\left|u(x_1)-u(x_1+y)-u(x_2)+u(x_2+y)\right|}{|y|^{n+2s}}e^{-\lambda|y|}\,dy
      % \\
      &\le I_1+I_2,
      \end{align*}
      where, for a parameter $r>0$ to be chosen,
      \begin{align*}
      I_1
      &:=
      C_{n,s}
      \int_{B_r}
      \frac{\left|u(x_1)-u(x_1+y)-u(x_2)+u(x_2+y)\right|}{|y|^{n+2s}}e^{-\lambda|y|}\,dy,
      \\
      I_2
      &:=
      C_{n,s}
      \int_{\mathbb R^n\setminus B_r}
      \frac{\left|u(x_1)-u(x_1+y)-u(x_2)+u(x_2+y)\right|}{|y|^{n+2s}}e^{-\lambda|y|}\,dy.
      \end{align*}

      For $I_1$, we use that $|u(x_i)-u(x_i+y)|\le [u]_{C^{0,\alpha}}|y|^\alpha$ for $i=1,2$.
      Therefore,
      \[
      I_1
      \le
      C_{n,s}\int_{B_r}\frac{2[u]_{C^{0,\alpha}}|y|^\alpha}{|y|^{n+2s}}\,dy
      \le
      C\,[u]_{C^{0,\alpha}}\,r^{\alpha-2s}.
      \]

      For $I_2$, we use that $|u(x_1+y)-u(x_2+y)|\le [u]_{C^{0,\alpha}}|x_1-x_2|^\alpha$.
      Therefore,
      \[
      I_2
      \le
      C_{n,s}\int_{\mathbb R^n\setminus B_r}
      \frac{2[u]_{C^{0,\alpha}}|x_1-x_2|^\alpha}{|y|^{n+2s}}\,dy
      \le
      C\,[u]_{C^{0,\alpha}}\,r^{-2s}|x_1-x_2|^\alpha.
      \]

      Choosing $r=|x_1-x_2|$ and adding the bounds for $I_1$ and $I_2$, we obtain
      \[
      \bigl|(-\Delta)^s_\lambda u(x_1)-(-\Delta)^s_\lambda u(x_2)\bigr|
      \le
      C\,[u]_{C^{0,\alpha}}\,|x_1-x_2|^{\alpha-2s},
      \]
      which implies \eqref{eq:prop2.5-est} and hence $(-\Delta)^s_\lambda u\in C^{0,\alpha-2s}(\mathbb R^n)$.
  \end{proof}

\begin{lemma}\label{prop:2.6}
     Let $0<s<1$ and $u\in C^{1,\alpha}(\mathbb R^n)$ for some $\alpha\in(0,1]$.
      Then, for the tempered fractional Laplacian
      $(-\Delta)^s_\lambda$ with tempering parameter $\lambda>0$, 
      one has
      \begin{itemize}
          \item If $\alpha>2s$, then $(-\Delta)^s_\lambda u\in C^{1,\alpha-2s}(\mathbb R^n)$ and
      \begin{equation}\label{eq:prop2.6-case1}
        \bigl[(-\Delta)^s_\lambda u\bigr]_{C^{1,\alpha-2s}(\mathbb R^n)}
        \le
        C\,[u]_{C^{1,\alpha}(\mathbb R^n)},
      \end{equation}
      where $C$ depends only on $\alpha$, $s$, and $n$.

      \item If $2s-1<\alpha<2s$, then $(-\Delta)^s_\lambda u\in C^{0,\alpha-2s+1}(\mathbb R^n)$ and
      \begin{equation}\label{eq:prop2.6-case2}
      \bigl[(-\Delta)^s_\lambda u\bigr]_{C^{0,\alpha-2s+1}(\mathbb R^n)}
      \le
      C\,[u]_{C^{1,\alpha}(\mathbb R^n)},
      \end{equation}
      where $C$ depends only on $\alpha$, $s$, and $n$.
      \end{itemize}
\end{lemma}

  \begin{proof}
      The first part follows directly from Proposition~\ref{prop:2.5} together with the fact that
      $(-\Delta)^s_\lambda$ commutes with differentiation.

      For the second part, we assume $s<\tfrac12$ . Proceeding as in the proof of
      Proposition~\ref{prop:2.5}, we obtain
      \[
      \bigl|(-\Delta)^s_\lambda u(x_1)-(-\Delta)^s_\lambda u(x_2)\bigr|\le I_1+I_2,
      \]
      for the same $I_1$ and $I_2$ as before. To estimate $I_1$, we use that since $u\in C^{1,\alpha}$,
      \begin{align*}
      \bigl|u(x_1)-u(x_1+y)-u(x_2)+u(x_2+y)\bigr|
      &\le \bigl|(\nabla u(x_1)-\nabla u(x_2))\cdot y\bigr| + [u]_{C^{1,\alpha}}|y|^{1+\alpha}
      \\
      &\le [u]_{C^{1,\alpha}}\bigl(|y|\,|x_1-x_2|^{\alpha}+|y|^{1+\alpha}\bigr),
      \end{align*}
      and hence
      \[
      I_1\le C\,[u]_{C^{1,\alpha}}
      \Bigl(r^{1-2s}|x_1-x_2|^{\alpha}+r^{1+\alpha-2s}\Bigr).
      \]
      Furthermore, to estimate $I_2$,
      \begin{align*}
      \bigl|u(x_1)-u(x_2)-(u(x_1+y)-u(x_2+y))\bigr|
      &\le C\,[u]_{C^{1,\alpha}}\,|x_1-x_2|\,|y|^{\alpha}.
      \end{align*}
      and hence
      \[
      I_2
      \le
      C\,[u]_{C^{1,\alpha}}\,|x_1-x_2|\,r^{\alpha-2s}.
      \]

      Choosing $r=|x_1-x_2|$, we obtain
      \[
      \bigl|(-\Delta)^s_\lambda u(x_1)-(-\Delta)^s_\lambda u(x_2)\bigr|
      \le
      C\,[u]_{C^{1,\alpha}}\,|x_1-x_2|^{\alpha-2s+1},
      \]
      which implies \eqref{eq:prop2.6-case2} and hence $(-\Delta)^s_\lambda u\in C^{0,\alpha-2s+1}(\mathbb R^n)$.

      In the case $s\ge\tfrac12$, to overcome the divergence of the singular integral kernel at the origin, one must exploit the radial symmetry of the kernel,
      $
      {e^{-\lambda |y|}}/({|y|^{n+2s}})
      $ is even in y.
      and rewrite the operator in the principal value form involving the second-order
      difference:
      \[
      (-\Delta)^s_\lambda u(x)
      =
      -\frac{1}{2}c_{n,s}\int_{\mathbb{R}^n}
      \frac{u(x+y)+u(x-y)-2u(x)}{|y|^{n+2s}}e^{-\lambda |y|}\,dy.
      \]

      Denote the second-order difference operator by
      \[
      \delta^2 u(x,y)=u(x+y)+u(x-y)-2u(x).
      \]

      Then, for any $x_1,x_2\in\mathbb{R}^n$, one has
      \[
      \left|(-\Delta)^s_\lambda u(x_1)-(-\Delta)^s_\lambda u(x_2)\right|
      \le I_3+I_4,
      \]
      where
      \[
      I_3
      \le
      \frac{1}{2}C_{n,s}\int_{B_r}
      \frac{|\delta^2 u(x_1,y)|+|\delta^2 u(x_2,y)|}{|y|^{n+2s}}e^{-\lambda |y|}\,dy,
      \]
      and
      \[
      I_4
      \le
      \frac{1}{2}C_{n,s}\int_{\mathbb{R}^n\setminus B_r}
      \frac{|\delta^2 u(x_1,y)-\delta^2 u(x_2,y)|}{|y|^{n+2s}}e^{-\lambda |y|}\,dy.
      \]

      For $I_1$, by the fundamental theorem of calculus in several variables, the second-order
      difference $\delta^2 u(x,y)$ can be expanded in terms of the gradient as
      \[
      u(x+y)-u(x)=\int_0^1 \nabla u(x+ty)\cdot y\,dt,
      \]
      and
      \[
      u(x-y)-u(x)=\int_0^1 \nabla u(x-ty)\cdot(-y)\,dt
      =-\int_0^1 \nabla u(x-ty)\cdot y\,dt.
      \]
      Adding the above two identities and taking absolute values, we obtain
      \[
      |\delta^2 u(x,y)|
      \le
      \int_0^1 |\nabla u(x+ty)-\nabla u(x-ty)|\,|y|\,dt.
      \]
      Since $\nabla u\in C^{0,\alpha}(\mathbb{R}^n)$, it follows from the H\"older
      continuity of $\nabla u$ that
      \[
      |\delta^2 u(x,y)|
      \le
      \int_0^1 [u]_{C^{1,\alpha}(\mathbb{R}^n)}\,|2ty|^\alpha |y|\,dt
      =
      [u]_{C^{1,\alpha}(\mathbb{R}^n)}\,2^\alpha |y|^{1+\alpha}\int_0^1 t^\alpha\,dt,
      \]
      and hence
      \[
      |\delta^2 u(x,y)|
      \le
      C[u]_{C^{1,\alpha}(\mathbb{R}^n)}\,|y|^{1+\alpha}.
      \]

      Substituting the above bound into $I_3$, and using the elementary estimate
      $
      e^{-\lambda |y|}\le 1,
      $
      together with the polar coordinate formula
      \[
      dy=|\mathbb{S}^{n-1}|\,\rho^{n-1}d\rho,
      \]
      we obtain
      \[
      I_3
      \le
      C_{n,s}|\mathbb{S}^{n-1}|
      \frac{2^\alpha}{1+\alpha}
      [u]_{C^{1,\alpha}(\mathbb{R}^n)}
      \int_0^r
      \rho^{1+\alpha}\rho^{-n-2s}\rho^{n-1}\,d\rho.
      \]
      Since $\alpha>2s-1$, we have $\alpha-2s>-1$, and hence the above integral is finite.

      Consequently,
      \[
      I_3
      \le
      C [u]_{C^{1,\alpha}(\mathbb{R}^n)}\, r^{\alpha-2s+1}.
      \]

      For $I_4$, we estimate the increment of the second-order difference
      between two spatial points. For each fixed $y\in\mathbb{R}^n$, define
      \[
      h(x)=\delta^2u(x,y)=u(x+y)+u(x-y)-2u(x).
      \]
      Then, we have
      \[
      |\delta^2u(x_1,y)-\delta^2u(x_2,y)|
      =
      |h(x_1)-h(x_2)|
      \le
      \|\nabla_x h\|_{L^\infty(\mathbb{R}^n)}\,|x_1-x_2|.
      \]
      Moreover,
      \[
      \nabla_x h(x)
      =
      \nabla u(x+y)+\nabla u(x-y)-2\nabla u(x).
      \]

      Using the H\"older continuity of $\nabla u$, we obtain
      \[
      |\nabla_x h(x)|
      \le
      |\nabla u(x+y)-\nabla u(x)|
      +
      |\nabla u(x-y)-\nabla u(x)|
      \le
      2 [u]_{C^{1,\alpha}(\mathbb{R}^n)} |y|^\alpha.
      \]
      Therefore,
      \[
      |\delta^2u(x_1,y)-\delta^2u(x_2,y)|
      \le
      C [u]_{C^{1,\alpha}(\mathbb{R}^n)} |x_1-x_2|\,|y|^\alpha.
      \]

      Substituting this estimate into $I_4$, and using again that
      $
      e^{-\lambda|y|}\le 1,
      $
      we obtain
      \begin{align*}
      I_4
      &\le
      C [u]_{C^{1,\alpha}(\mathbb{R}^n)} |x_1-x_2|
      \int_{\mathbb{R}^n\setminus B_r}
      |y|^{\alpha-n-2s}\,dy
      \\&\le
      C [u]_{C^{1,\alpha}(\mathbb{R}^n)} |x_1-x_2|
      \int_r^\infty \rho^{\alpha-2s-1}\,d\rho.
      \end{align*}

      Since $\alpha<2s$, we have $\alpha-2s-1<-1$, and hence the above improper
      integral converges. More precisely,
      \[
      \int_r^\infty \rho^{\alpha-2s-1}\,d\rho
      =
      \frac{r^{\alpha-2s}}{2s-\alpha}.
      \]
      Consequently,
      \[
        I_4
      \le
        C [u]_{C^{1,\alpha}(\mathbb{R}^n)}|x_1-x_2|\, r^{\alpha-2s}.
      \]

      Choosing$r=|x_1-x_2|$,
      we arrive at
      \[
        I_4
      \le
        C [u]_{C^{1,\alpha}(\mathbb{R}^n)}|x_1-x_2|^{1+\alpha-2s}.
      \]

      Combining this with the estimate for $I_3$, we conclude that
      \[
        |(-\Delta)^s_\lambda u(x_1)-(-\Delta)^s_\lambda u(x_2)|
      \le
      C [u]_{C^{1,\alpha}(\mathbb{R}^n)}|x_1-x_2|^{1+\alpha-2s}.
      \]
      Hence,
      \[
      \bigl[(-\Delta)^s_\lambda u\bigr]_{C^{0,\alpha-2s+1}(\mathbb{R}^n)}
      \le
      C [u]_{C^{1,\alpha}(\mathbb{R}^n)},
      \]
      and
      \[
      (-\Delta)^s_\lambda u\in C^{0,\alpha-2s+1}(\mathbb{R}^n).
      \]
      This completes the proof.
  \end{proof}

Iterating the last two lemmas we get the following result:
\begin{lemma}
      Let $u\in C^{k,\alpha}(\mathbb{R}^n)$ and suppose that $k+\alpha-2s>0 $ is not an integer. Then
        \[
        (-\Delta)^s_\lambda u \in C^{l, l^{'}}(\mathbb{R}^n),
        \]
      where $l$ is the integer part of $k+\alpha-2s$ and $l^{'} = k+\alpha-2s-l.$
\end{lemma}

\begin{proposition}\label{RO2.2}
  {\cite[Proposition~2.2]{RosOtonSerra2014}}
    Let $n \ge 1$, $0 < s < 1$. 
    Assume that \(w \in C^\infty(\mathbb{R}^n)\) solves
    $(-\Delta)^s w = h\,\text{in } B_1$
    and that neither \(\beta\) nor \(\beta+2s\) is an integer. Then,
    \[
    \|w\|_{C^{\beta+2s}(\overline{B}_{1/2})}
    \le
    C\Bigl(
    \|w\|_{C^\beta(\mathbb{R}^n)}
    +
    \|h\|_{C^\beta(\overline{B}_1)}
    \Bigr),
    \]
    where \(C\) is a constant depending only on \(n\), \(s\), and \(\beta\).
\end{proposition}

\begin{lemma}\label{tempered-RO2.2}
    Let $n \ge 1$, $0 < s < 1$, $ \beta > 0$, and $\lambda > 0$.  
    Assume that \(v \in C^\infty(\mathbb{R}^n) \cap C^\beta(\mathbb{R}^n)  \) solves
    $(-\Delta)^s_\lambda v = F\,\text{in } B_1$
    and that neither \(\beta\) nor \(\beta+2s\) is an integer. Then,
      \[
        \|v\|_{C^{\beta+2s}(\overline{B}_{1/2})}
        \le
        C\Bigl(
        \|v\|_{C^\beta(\mathbb{R}^n)}+\|F\|_{C^\beta(\overline{B}_1)}
        \Bigr),
      \]
    where \(C\) is a constant depending only on \(n\), \(s\), \(\beta\), and $\lambda$.
\end{lemma}

    \begin{proof}
      For  $0 < s < \tfrac12$, the stated estimate follows directly by combining equations~\eqref{Q=g-f} and \eqref{eqQ}, Lemma~\ref{QCA} and~\ref{thm:Q-subcritical-holder} with
        Proposition~\ref{RO2.2}.

        For $\tfrac12 \le s < 1$,
          % we write
          % $
          % \gamma:=2s-1.
          % $
          % Since \(\tfrac12 \le s<1\), 
          we have
          $
          0\le 2s-1<1.
          $
          Choose \(\theta\) such that
          $2s-1<\theta<1,$
          and, if necessary, slightly change \(\theta\) so that \(\beta+\theta\notin\mathbb N\).
          Then
          \[
          \beta<\beta+\theta<\beta+2s.
          \]

We use the following standard translation estimate in H\"older spaces. Write
\[
\beta=m+\sigma,
\qquad
m:=\lfloor \beta\rfloor,
\qquad
0<\sigma<1.
\]
Since \(\beta\notin\mathbb N\), such a decomposition is well-defined. We claim that
\begin{equation}\label{delta_z}
    \|\delta_z v\|_{C^\beta(\overline B_r)}
    \le
    C|z|^\theta
    \|v\|_{C^{\beta+\theta}(\overline B_R)},
    \qquad |z|<d:=R-r<1,
\end{equation}
where
$\delta_z v(x):=v(x)-v(x+z).$
Indeed, since \(x\in B_r\) and \(|z|<d=R-r\), one has
$x+z\in B_R.$
Thus all the quantities below are well-defined in \(\overline B_R\).

For every  \(\nu\) with \(|\nu|\le m\), since \(v\in C^\infty(\mathbb R^n)\), we have
$
D^\nu\delta_z v
=
\delta_z(D^\nu v).
$
We first consider the lower-order derivatives \(|\nu|<m\). 
By the mean value theorem,
\[
\begin{aligned}
\|D^\nu\delta_z v\|_{L^\infty(\overline B_r)}
&=
\|\delta_z(D^\nu v)\|_{L^\infty(\overline B_r)} \\
&\le
C|z|
\|v\|_{C^{\beta+\theta}(\overline B_R)}.
\end{aligned}
\]
Since \(|z|<1\) and \(0<\theta<1\), we have
$
|z|\le |z|^\theta.
$
Therefore,
\[
\|D^\nu\delta_z v\|_{L^\infty(\overline B_r)}
\le
C|z|^\theta
\|v\|_{C^{\beta+\theta}(\overline B_R)},
\qquad |\nu|<m.
\]

It remains to estimate the derivatives \(|\nu|=m\). Let
$
G:=D^\nu v.
$
Since
\[
v\in C^{\beta+\theta}(\overline B_R)
=
C^{m+\sigma+\theta}(\overline B_R),
\]
we have
$
G\in C^{\sigma+\theta}(\overline B_R).
$
We shall prove
\begin{equation}\label{top_delta_z}
    \|\delta_z G\|_{C^\sigma(\overline B_r)}
    \le
    C|z|^\theta
    \|G\|_{C^{\sigma+\theta}(\overline B_R)}.
\end{equation}

First, we estimate the \(L^\infty\)-part. If \(0<\sigma+\theta<1\), then
\[
|\delta_z G(x)|
=
|G(x)-G(x+z)|
\le
[G]_{C^{\sigma+\theta}(\overline B_R)}|z|^{\sigma+\theta}
\le
C|z|^\theta
\|G\|_{C^{\sigma+\theta}(\overline B_R)}.
\]
If \(1<\sigma+\theta<2\), then \(DG\in C^{\sigma+\theta-1}(\overline B_R)\), and hence
\[
|\delta_z G(x)|
\le
|z|\|DG\|_{L^\infty(\overline B_R)}
\le
C|z|^\theta
\|G\|_{C^{\sigma+\theta}(\overline B_R)},
\]
again because \(|z|<1\) and \(0<\theta<1\). Thus, in both cases,
    \begin{equation}\label{GLwuqiong}
        \|\delta_z G\|_{L^\infty(\overline B_r)}
        \le
        C|z|^\theta
        \|G\|_{C^{\sigma+\theta}(\overline B_R)}.
    \end{equation}

We now estimate the H\"older seminorm. Let \(x,y\in \overline B_r\), \(x\neq y\), 
First suppose that \(0<\sigma+\theta<1\). 

If \(|x-y|\le |z|\), then
\[
\begin{aligned}
|\delta_z G(x)-\delta_z G(y)|
&\le
|G(x)-G(y)|
+
|G(x+z)-G(y+z)| \\
&\le
2[G]_{C^{\sigma+\theta}(\overline B_R)}|x-y|^{\sigma+\theta} \\
&\le
C\|G\|_{C^{\sigma+\theta}(\overline B_R)}
|z|^\theta |x-y|^\sigma .
\end{aligned}
\]

If \(|x-y|>|z|\), then
\[
\begin{aligned}
|\delta_z G(x)-\delta_z G(y)|
&\le
|G(x)-G(x+z)|
+
|G(y)-G(y+z)| \\
&\le
2[G]_{C^{\sigma+\theta}(\overline B_R)}|z|^{\sigma+\theta} \\
&\le
C\|G\|_{C^{\sigma+\theta}(\overline B_R)}
|z|^\theta |x-y|^\sigma .
\end{aligned}
\]
Therefore, when \(0<\sigma+\theta<1\),
  \begin{equation}\label{GHolder1}
      [\delta_z G]_{C^\sigma(\overline B_r)}
      \le
      C|z|^\theta
      \|G\|_{C^{\sigma+\theta}(\overline B_R)}.
  \end{equation}

Now suppose that \(1<\sigma+\theta<2\). Put
$\kappa:=\sigma+\theta-1\in(0,1).$
Then $DG\in C^\kappa(\overline B_R).$

If \(|x-y|\le |z|\), then
\[
\begin{aligned}
\delta_z G(x)-\delta_z G(y)
&=G(x)-G(x+z)-G(y)+G(y+z) \\
&=
\int_0^1
\left[
DG(y+t(x-y))-DG(y+z+t(x-y))
\right]\cdot (x-y)\,dt .
\end{aligned}
\]
Hence
\[
\begin{aligned}
|\delta_z G(x)-\delta_z G(y)|
&\le
[DG]_{C^\kappa(\overline B_R)}
|x-y| |z|^\kappa \\
&=
[DG]_{C^\kappa(\overline B_R)}
|x-y| |z|^{\sigma+\theta-1}.
\end{aligned}
\]
Since \(|x-y|\le |z|\), we have
$
|x-y|^{1-\sigma}\le |z|^{1-\sigma}.
$

Therefore,
\begin{align*}
  |x-y| |z|^{\sigma+\theta-1}
&=
  |x-y|^\sigma |x-y|^{1-\sigma}|z|^{\sigma+\theta-1}\\
&\le
  |x-y|^\sigma |z|^{1-\sigma}|z|^{\sigma+\theta-1}\\
&=
  |z|^\theta|x-y|^\sigma.
\end{align*}

Thus
\[
|\delta_z G(x)-\delta_z G(y)|
\le
C\|G\|_{C^{\sigma+\theta}(\overline B_R)}
|z|^\theta|x-y|^\sigma .
\]

If \(|x-y|>|z|\), then
\[
\begin{aligned}
\delta_z G(x)-\delta_z G(y)
&=
-\int_0^1
\left[
DG(x+tz)-DG(y+tz)
\right]\cdot z\,dt .
\end{aligned}
\]
Hence
\[
\begin{aligned}
|\delta_z G(x)-\delta_z G(y)|
&\le
[DG]_{C^\kappa(\overline B_R)}
|z||x-y|^\kappa \\
&=
[DG]_{C^\kappa(\overline B_R)}
|z||x-y|^{\sigma+\theta-1}.
\end{aligned}
\]
Since \(|z|<|x-y|\), we have
$
|z|^{1-\theta}\le |x-y|^{1-\theta}.
$

Therefore,
\begin{align*}
  |z||x-y|^{\sigma+\theta-1}
&=
  |z|^\theta |z|^{1-\theta}|x-y|^{\sigma+\theta-1}\\
&\le
  |z|^\theta |x-y|^{1-\theta}|x-y|^{\sigma+\theta-1}\\
&=
  |z|^\theta|x-y|^\sigma.
\end{align*}

Thus
\[
|\delta_z G(x)-\delta_z G(y)|
\le
C\|G\|_{C^{\sigma+\theta}(\overline B_R)}
|z|^\theta|x-y|^\sigma .
\]
Consequently, when \(1<\sigma+\theta<2\),
  \begin{equation}\label{GHolder2}
      [\delta_z G]_{C^\sigma(\overline B_r)}
      \le
      C|z|^\theta
      \|G\|_{C^{\sigma+\theta}(\overline B_R)}.
  \end{equation}

Combining~\eqref{GLwuqiong}, \eqref{GHolder1}and~\eqref{GHolder2}, this gives
\[
\|\delta_z G\|_{C^\sigma(\overline B_r)}
\le
C|z|^\theta
\|G\|_{C^{\sigma+\theta}(\overline B_R)}.
\]
Since \(G=D^\nu v\), for every \(|\nu|=m\) we obtain
\[
\|D^\nu\delta_z v\|_{C^\sigma(\overline B_r)}
\le
C|z|^\theta
\|v\|_{C^{\beta+\theta}(\overline B_R)}.
\]

Combining this estimate for all \(|\nu|=m\) with  \(|\nu|<m\), we conclude that
\[
\|\delta_z v\|_{C^\beta(\overline B_r)}
\le
C|z|^\theta
\|v\|_{C^{\beta+\theta}(\overline B_R)}.
\]
This proves \eqref{delta_z}.

\color{black}
         Next we split
          $\mathcal{Q}^s_\lambda v=I_1+I_2,$
          where
          \[
          I_1(x):=
          c_{n,s}\int_{|z|<d}\delta_z v(x)
          \frac{1-e^{-\lambda |z|}}{|z|^{n+2s}}\,dz,
          \]
          and
          \[
          I_2(x):=
          c_{n,s}\int_{|z|\ge d}\delta_z v(x)
          \frac{1-e^{-\lambda |z|}}{|z|^{n+2s}}\,dz.
          \]

          Therefore, for $I_1$,
          \[
          \begin{aligned}
          \|I_1\|_{C^\beta(\overline B_r)}
          &\le
              C\int_{|z|<d}
              \|\delta_z v\|_{C^\beta(\overline B_r)}
              |z|^{-n-2s+1}\,dz  \\
          &\le
              C\|v\|_{C^{\beta+\theta}(\overline B_R)}
              \int_{|z|<d}|z|^{\theta-n-2s+1}\,dz  \\
          &=
              C\|v\|_{C^{\beta+\theta}(\overline B_R)}
              \int_0^d \rho^{\theta-2s}\,d\rho.\\
          &\le
             C\|v\|_{C^{\beta+\theta}(\overline B_R)}
          \end{aligned}
          \]
  
          For \(I_2\), the kernel is integrable on \({|z|\ge d}\), and
          $
          \|\delta_z v\|_{C^\beta(\overline B_r)}
          \le
          C\|v\|_{C^\beta(\mathbb R^n)},
          $
          we have 
          \[
          \|I_2\|_{C^\beta(\overline B_r)}
          \le
          C d^{-N}
          \|v\|_{C^\beta(\mathbb R^n)}
          \]
          for any \(N \ge 2 > 2s\). 
          
          Consequently,
          \[
          \|\mathcal{Q}^s_\lambda v\|_{C^\beta(\overline B_r)}
          \le
          C\|v\|_{C^{\beta+\theta}(\overline B_R)}
          +
          C d^{-N}\|v\|_{C^\beta(\mathbb R^n)}.
          \]

          \color{black}
          Using the interpolation inequality
          \[
          \|v\|_{C^{\beta+\theta}(\overline B_R)}
          \le
          \varepsilon
          \|v\|_{C^{\beta+2s}(\overline B_R)}
          +
          C(\varepsilon)
          \|v\|_{C^\beta(\overline B_R)},
          \]
          we obtain
          \begin{equation}\label{qepsilon}
            \|\mathcal{Q}^s_\lambda v\|_{C^\beta(\overline B_r)}
              \le
              C\varepsilon
              \|v\|_{C^{\beta+2s}(\overline B_R)}
              +
              C_\varepsilon d^{-N}
              \|v\|_{C^\beta(\mathbb R^n)}.
          \end{equation}
          with  
          $C(\varepsilon)=\eta\left(\frac{1}{1-\eta}\right)^{\frac{\eta-1}{\eta}}
          \left({\varepsilon}\right)^{\frac{-\theta}{2s-\theta}}, 
          \eta=\frac{2s-\theta}{2s},  2s-1<\theta<1$.
          Since \(0<\eta<1\), then
          \[
          \eta\left(\frac{1}{1-\eta}\right)^{\frac{\eta-1}{\eta}}
          =
          \eta\left(\frac{1}{1-\eta}\right)^{-\frac{1-\eta}{\eta}}
          =
          \eta(1-\eta)^{\frac{1-\eta}{\eta}}.
          \]
          Since
          \[
          0<1-\eta<1
          \quad\text{and}\quad
          \frac{1-\eta}{\eta}>0,
          \]
          we have
          \[
          0<(1-\eta)^{\frac{1-\eta}{\eta}}\le 1.
          \]
          Therefore,
          \[
          0<
          \eta(1-\eta)^{\frac{1-\eta}{\eta}}
          \le
          \eta
          \le 1.
          \]
          Hence
          \[
          \eta\left(\frac{1}{1-\eta}\right)^{\frac{\eta-1}{\eta}}
          \le 1,
          \qquad 0<\eta<1.
          \]
          In particular, the quantity is uniformly bounded on \((0,1)\). 
          So we write $C(\varepsilon) \le \varepsilon^{-q}, q=\frac{\theta}{2s-\theta}$.
          
          Now, since
          \[
          (-\Delta)^s v
          =
          (-\Delta)^s_\lambda v+\mathcal{Q}^s_\lambda v,
          \]
          and since
          \[
          (-\Delta)^s_\lambda v=F
          \qquad\text{in }B_1,
          \]
          we have
          \begin{equation}\label{1234}
             (-\Delta)^s v=F+\mathcal{Q}^s_\lambda v
              \qquad\text{in }B_1.
          \end{equation}

          Let
          \[
          A(\rho):=
          \|v\|_{C^{\beta+2s}(\overline B_\rho)},
          \]
          and
          \[
          M:=
          \|v\|_{C^\beta(\mathbb R^n)}
          +
          \|F\|_{C^\beta(\overline B_1)}.
          \]

          For every $1/2\le r<R<1$, follow the proof of Proposition~\ref{RO2.2} and use \eqref{1234} , 
          one has
          \begin{align}\label{HolderrR}
          \|v\|_{C^{\beta+2s}(\overline B_r)}
              &\le
              C\Bigl(
              \|v\|_{C^\beta(\mathbb{R}^n)}+\|F+\mathcal{Q}^s_\lambda v\|_{C^\beta(\overline{B}_{(r+R)/2})}
              \Bigr)\\
                &\le
               C
              \left(
             M+ \|\mathcal{Q}^s_\lambda v\|_{C^\beta(\overline B_{(r+R)/2})}
          \right).
          \end{align}
          Applying the lower-order perturbation estimate to
          $
          B_{(r+R)/2} \subset B_R,
          $
          by~\eqref{qepsilon} and let $r=(r+R)/2 $, we get
          \begin{equation}\label{R+r}
            \|\mathcal{Q}^s_\lambda v\|_{C^\beta(\overline B_{(r+R)/2})}
          \le
          C\varepsilon \|v\|_{C^{\beta+2s}(\overline B_R)}
          +
          \varepsilon^{-q} (R-r)^{-N}
          \|v\|_{C^\beta(\mathbb R^n)}.
          \end{equation}
          %   \[
          % \|\mathcal{Q}^s_\lambda v\|_{C^\beta(\overline B_{(r+R)/2})}
          % \le
          % \varepsilon \|v\|_{C^{\beta+2s}(\overline B_R)}
          % +
          % C_\varepsilon (R-r)^{-N}
          % \|v\|_{C^\beta(\mathbb R^n)}.
          % \]

          Hence by \eqref{HolderrR} and~\eqref{R+r}, we get
          \[
          \|v\|_{C^{\beta+2s}(\overline B_r)}
          \le
          C(R-r)^{-N}M
          +
          C(R-r)^{-N}\varepsilon \|v\|_{C^{\beta+2s}(\overline B_R)}
          +
          C\varepsilon^{-q} (R-r)^{-N}M.
          \]

Now assume that, for \(1/2\le r<R<1\) and fix $ N = N_{0}$, we have obtained

\begin{eqnarray}\label{gujiC}
  \|v\|_{C^{\beta+2s}(\overline B_r)}
\le
C_0(R-r)^{-N_0}M
+
C_1(R-r)^{-N_0}\varepsilon \|v\|_{C^{\beta+2s}(\overline B_R)}
+
C_2\varepsilon^{-q}(R-r)^{-N_0}M,
\end{eqnarray}
where \(q>0\), \(N_0>0\), and \(C_0,C_1,C_2\) are independent of
\(r,R,\varepsilon,v,F\).

For a fixed number \(\delta\in(0,1)\), choose
\[
\varepsilon:=\frac{\delta}{2C_1}d^{N_0},
\]
then
\begin{equation}\label{C_1}
  C_1d^{-N_0}\varepsilon
=
    \frac{\delta}{2}
    \le \delta.
\end{equation}  
Moreover,
\[
\varepsilon^{-q}
=
\left(\frac{2C_1}{\delta}\right)^q
d^{-qN_0}.
\]
Hence
\begin{equation}\label{C2}
  C_2\varepsilon^{-q}d^{-N_0}M
=
C_2\left(\frac{2C_1}{\delta}\right)^q
d^{-(1+q)N_0}M.
\end{equation}

Since \(0<d<1\), we also have
\begin{equation}\label{C_0}
  C_0d^{-N_0}M
\le
  C_0d^{-(1+q)N_0}M.
\end{equation}

Therefore, by~\eqref{gujiC},~\eqref{C_1}, \eqref{C2} and \eqref{C_0}, we get

\begin{equation}\label{Czong}
  \|v\|_{C^{\beta+2s}(\overline B_r)}
\le
\delta \|v\|_{C^{\beta+2s}(\overline B_R)}
+
C_\delta d^{-\widetilde N}M,
\end{equation}
where
$
\widetilde N:=(1+q)N_0
$
and
$
C_\delta
:=
C_0
+
C_2\left(\frac{2C_1}{\delta}\right)^q.
$
Thus,
\[
\|v\|_{C^{\beta+2s}(\overline B_r)}
\le
\delta \|v\|_{C^{\beta+2s}(\overline B_R)}
+
C_\delta (R-r)^{-\widetilde N}M,
\qquad
\frac12\le r<R<1.
\]

Notice that \(C_\delta\) may blow up as \(\delta\downarrow0\). This is not a problem, because we do not let
\(\delta\) tend to zero. We now fix one positive value of \(\delta\). Choose
\[
\delta_0:=2^{-(\widetilde N+1)}.
\]
Then
\[
\delta_0 2^{\widetilde N}
=
\frac12<1.
\]
For this fixed \(\delta_0>0\), the constant
\[
C_{\delta_0}
=
C_0
+
C_2\left(\frac{2C_1}{\delta_0}\right)^q
\]
is finite. Hence by~\eqref{Czong}, we have
\[
\|v\|_{C^{\beta+2s}(\overline B_r)}
\le
\delta_0 \|v\|_{C^{\beta+2s}(\overline B_R)}
+
C_{\delta_0}(R-r)^{-\widetilde N}M,
\qquad
\frac12\le r<R<1.
\]

Now define
\[
r_k:=\frac34-\frac14\,2^{-k},
\qquad k=0,1,2,\dots.
\]
Then
\[
r_0=\frac12,\qquad r_k\uparrow\frac34,
\]
and
\[
r_{k+1}-r_k=2^{-k-3}.
\]
Applying the preceding estimate with \(r=r_k\) and \(R=r_{k+1}\), we get
\[
A(r_k)
\le
\delta_0 A(r_{k+1})
+
C_{\delta_0}(r_{k+1}-r_k)^{-\widetilde N}M.
\]
Since
\[
(r_{k+1}-r_k)^{-\widetilde N}
=
2^{\widetilde N(k+3)}
=
2^{3\widetilde N}2^{\widetilde N k},
\]
we obtain
\[
A(r_k)
\le
\delta_0 A(r_{k+1})
+
C_{\delta_0}2^{3\widetilde N}2^{\widetilde N k}M.
\]

Iterating this inequality gives
\[
A(r_0)
\le
\delta_0^m A(r_m)
+
C_{\delta_0}2^{3\widetilde N}M
\sum_{k=0}^{m-1}
\delta_0^k2^{\widetilde N k}.
\]
Since
\[
\delta_0 2^{\widetilde N}
=
\frac12,
\]
the series
\[
\sum_{k=0}^{\infty}
\delta_0^k2^{\widetilde N k}
=
\sum_{k=0}^{\infty}
\left(\delta_0 2^{\widetilde N}\right)^k
\]
converges.

Moreover, since \(v\in C^\infty(\mathbb R^n)\), we have
\[
A(3/4)<\infty.
\]
Because \(r_m<3/4\), it follows that
\[
\|v\|_{C^{\beta+2s}(\overline B_{r_m})}\le \|v\|_{C^{\beta+2s}(\overline B_{3/4})}.
\]
Therefore,
\[
\delta_0^m \|v\|_{C^{\beta+2s}(\overline B_{r_m})}
\le
\delta_0^m \|v\|_{C^{\beta+2s}(\overline B_{3/4})}
\to0
\qquad\text{as }m\to\infty.
\]
Letting \(m\to\infty\), we obtain
\[
A(r_0)
\le
C_{\delta_0}2^{3\widetilde N}M
\sum_{k=0}^{\infty}
\left(\delta_0 2^{\widetilde N}\right)^k.
\]
Since \(r_0=1/2\), this gives
\[
\|v\|_{C^{\beta+2s}(\overline B_{1/2})}
\le 
C\left(
\|v\|_{C^\beta(\mathbb R^n)}
+
\|F\|_{C^\beta(\overline B_1)}
\right),
\]
where
\[
C:=
C_{\delta_0}2^{3\widetilde N}
\sum_{k=0}^{\infty}
\left(\delta_0 2^{\widetilde N}\right)^k
<\infty.
\]
% {
% \color{blue}

% For this fixed value of \(\delta_0\), the constant \(C_{\delta_0}\) is finite.
% Indeed, from the previous construction,
% \[
% C_{\delta_0}
% =
% C_0+
% C_2\left(\frac{2C_1}{\delta_0}\right)^q.
% \]
% where \(C_0,C_1,C_2,q\) are fixed positive constants. 

% Since \(\delta_0>0\), this gives
% $
% C_{\delta_0}<\infty.
% $
% More explicitly, because
% $
% \delta_0^{-1}=2^{\widetilde N+1},
% $
% we have
% \[
% C_{\delta_0}
% =
% C_0+
% C_2(2C_1)^q2^{q(\widetilde N+1)}
% <\infty.
% \]

% Moreover, since \(\widetilde N<\infty\),
% $
% 2^{3\widetilde N}<\infty.
% $, we have $C_{\delta_0}2^{3\widetilde N}<\infty.$

% Furthermore,
% \[
% \sum_{k=0}^{\infty}
% \left(\delta_0 2^{\widetilde N}\right)^k
% =
% \sum_{k=0}^{\infty}
% \left(\frac12\right)^k
% =
% 2.
% \]
% Consequently,
% \[
% C:=
% C_{\delta_0}2^{3\widetilde N}
% \sum_{k=0}^{\infty}
% \left(\delta_0 2^{\widetilde N}\right)^k
% =
% 2C_{\delta_0}2^{3\widetilde N}
% <\infty.
% \]
% }
    \end{proof}

  \begin{lemma}\label{tempered2.4}
    Let \(n \ge 1\), \(0 < s < 1\), \( \beta > 0 \), and \(\lambda > 0\). 
    Assume that \(w \in C^\infty(\mathbb{R}^n)\) solves
    $
    (-\Delta)^s_{\lambda} w = h \, \text{in } B_2,
    $
    and that \(\beta\notin\mathbb N, \beta + 2s \notin \mathbb{N}\). Then,
    
    \[
    \|w\|_{C^{\beta+2s}(\overline{B}_{1/2})}
    \le
    C\left(
    \left\|(1+|x|)^{-n-2s} w(x)\right\|_{L^1(\mathbb{R}^n)}
    +
    \|w\|_{C^\beta(\overline{B}_2)}
    +
    \|h\|_{C^\beta(\overline{B}_2)}
    \right),
    \]
    where \(C\) depends only on \(n\), \(s\), \(\beta\), and \(\lambda\).
  \end{lemma}

    \begin{proof}
        Let \(\eta \in C_c^\infty(\mathbb{R}^n)\) be such that
        \[
        \eta \equiv 0 \quad \text{in } \mathbb{R}^n \setminus B_2,
        \qquad
        \eta \equiv 1 \quad \text{in } B_{3/2},
        \qquad
        0 \le \eta \le 1.
        \]

        Then $ \widetilde{w} := \eta w\in C^\infty(\mathbb{R}^n)$ and \(\widetilde{w} = w\) in \(B_{3/2}\). 
        Hence,  in \(B_1\), we have
        \[
        (-\Delta)^s_\lambda \widetilde{w}=
        \widetilde{h}:=
        (-\Delta)^s_\lambda w - (-\Delta)^s_\lambda\bigl(w(1-\eta)\bigr)=
        h - (-\Delta)^s_\lambda\bigl(w(1-\eta)\bigr).
        \]
          Since \(\eta\) is smooth and supported in \(B_2\), we clearly have
        \[
        \|\widetilde{w}\|_{C^\beta(\mathbb{R}^n)}
        \le
        C \|w\|_{C^\beta(\overline{B}_2)}.
        \]
        Note that for \(x \in B_{3/2}\), we have
        \[
        (-\Delta)^s_\lambda\bigl(w(1-\eta)\bigr)(x)
        =
        c_{n,s}\int_{\mathbb{R}^n}
        \frac{-\bigl(w(1-\eta)\bigr)(y)}{|x-y|^{n+2s}}e^{-\lambda|x-y|}\,dy.
        \]

        Since \(\operatorname{supp}(1-\eta)\subset \mathbb{R}^n\setminus B_{3/2}\), when \(x\in B_1\) we have
        $|x-y|\ge \tfrac12.$
        Therefore, 
        \[
        \left|(-\Delta)^s_\lambda\bigl(w(1-\eta)\bigr)(x)\right|
        \le
        C\int_{\mathbb{R}^n\setminus B_{3/2}}
        \frac{|w(y)|}{|x-y|^{n+2s}}\,dy.
        \]

        For \(x\in B_1\) and \(y\in \mathbb{R}^n\setminus B_{3/2}\), we have
        \begin{align*}
              \left|(-\Delta)^s_\lambda\bigl(w(1-\eta)\bigr)(x)\right|
          &\le
              C\int_{\mathbb{R}^n\setminus B_{3/2}}
              \frac{|w(y)|}{(1+|y|)^{n+2s}}\,dy\\
          &\le
              C\int_{\mathbb{R}^n}
              \frac{|w(y)|}{(1+|y|)^{n+2s}}\,dy.
      \end{align*}
        Thus
        \[
          \left\|(-\Delta)^s_\lambda\bigl(w(1-\eta)\bigr)\right\|_{L^\infty(B_1)}
          \le
          C\left\|(1+|y|)^{-n-2s}w(y)\right\|_{L^1(\mathbb{R}^n)}.
        \]
        % Then take \(x,x' \in B_1\).
        % Since \(\eta \equiv 1\) in \(B_{3/2}\), it follows that \(w(1-\eta) \equiv 0\) in \(B_{3/2}\). Therefore, for any \(x \in B_1\), we have
        % \[
        % (-\Delta)^s_\lambda\bigl(w(1-\eta)\bigr)(x)
        % =
        % -c_{n,s}\int_{\mathbb{R}^n\setminus B_{3/2}}
        % \frac{w(y)(1-\eta(y))}{|x-y|^{n+2s}}e^{-\lambda|x-y|}\,dy.
        % \]

        % Hence, for \(x,x' \in B_1\), we may write
        % \[
        % (-\Delta)^s_\lambda\bigl(w(1-\eta)\bigr)(x)-(-\Delta)^s_\lambda\bigl(w(1-\eta)\bigr)(x')
        % \]
        % \[
        % =
        % -c_{n,s}\int_{\mathbb{R}^n\setminus B_{3/2}}
        % w(y)(1-\eta(y))
        % \left(
        % \frac{e^{-\lambda|x-y|}}{|x-y|^{n+2s}}
        % -
        % \frac{e^{-\lambda|x'-y|}}{|x'-y|^{n+2s}}
        % \right)\,dy.
        % \]

        % Therefore,
        % \[
        % \left|(-\Delta)^s_\lambda\bigl(w(1-\eta)\bigr)(x)-(-\Delta)^s_\lambda\bigl(w(1-\eta)\bigr)(x')\right|
        % \le
        % c_{n,s}\int_{\mathbb{R}^n\setminus B_{3/2}}
        % |w(y)|
        % \left|
        % \frac{e^{-\lambda|x-y|}}{|x-y|^{n+2s}}
        % -
        % \frac{e^{-\lambda|x'-y|}}{|x'-y|^{n+2s}}
        % \right|\,dy.
        % \]
        Next we shall prove
          \[
            \left\|
            (-\Delta)^s_\lambda\bigl(w(1-\eta)\bigr)
            \right\|_{C^\beta(\overline B_1)}
        \le
            C\left\|(1+|y|)^{-n-2s}w(y)\right\|_{L^1(\mathbb R^n)}.
          \]

        Write
        $\beta=m+\theta$, $m\in\mathbb N_0$, $0<\theta<1,$
        thus $C^\beta=C^{m,\theta}.$ First, let \(\mu\) be a multi-index with \(|\mu|\le m\). Since
        \(x\in B_1\) and \(y\in\mathbb R^n\setminus B_{3/2}\), we have
        $|x-y|\ge |y|-|x|\ge \frac32-1=\frac12.$
        
        Therefore, no singularity of the kernel is encountered in the region of
        integration. Consequently, we can obtain
        \[
          D_x^\mu\left[(-\Delta)^s_\lambda\bigl(w(1-\eta)\bigr)(x)\right]
        =
          -c_{n,s}
          \int_{\mathbb R^n\setminus B_{3/2}}
          w(y)(1-\eta(y))D_x^\mu (\frac{e^{-\lambda|x-y|}}{|x-y|^{n+2s}})\,dy .
        \]

        For every multi-index \(\mu\), there exists \(C_\mu>0\) such that      
        \begin{equation}
            |D_x^\mu (\frac{e^{-\lambda|x-y|}}{|x-y|^{n+2s}})|
          \le
            C_\mu (1+|x-y|)^{-n-2s},
            \qquad |x-y|\ge \frac12.
        \end{equation}

        Since \(x\in B_1\), we also have
        \[
            1+|y|
        \le
            1+|x-y|+|x|
        \le
            2+|x-y|
        \le
            2(1+|x-y|).
        \]
        Therefore,
        \begin{equation}\label{daoshu}
          |D_x^\mu (\frac{e^{-\lambda|x-y|}}{|x-y|^{n+2s}})|
        \le
            C_\mu(1+|x-y|)^{-n-2s}
        \le
            C_\mu(1+|y|)^{-n-2s}.
        \end{equation}
        It follows that
          \[
            \begin{aligned}
            |D_x^\mu\left[(-\Delta)^s_\lambda\bigl(w(1-\eta)\bigr)(x)\right]|
            &\le
            C
            \int_{\mathbb R^n\setminus B_{3/2}}
            |w(y)|(1-\eta(y))(1+|y|)^{-n-2s}\,dy \\
            &\le
            C
            \int_{\mathbb R^n}
            |w(y)|(1+|y|)^{-n-2s}\,dy .
            \end{aligned}
          \]
        Taking the supremum over \(x\in B_1\), we obtain that for all \(|\mu|\le m\),
        \begin{equation}\label{wuqiong}
        \|D^\mu (-\Delta)^s_\lambda\bigl(w(1-\eta)\bigr)\|_{L^\infty(B_1)}
        \le
        C
        \left\|(1+|y|)^{-n-2s}w(y)\right\|_{L^1(\mathbb R^n)}.
        \end{equation}

        It remains to estimate the H\"older seminorm of the derivatives of order \(m\).
        Let \(|\mu|=m\), and take \(x,x'\in B_1\), \(x\neq x'\). Then

        \[
        \begin{aligned}
          &D^\mu (-\Delta)^s_\lambda\bigl(w(1-\eta)\bigr)(x)
          -D^\mu (-\Delta)^s_\lambda\bigl(w(1-\eta)\bigr)(x')
        \\&=
          -c_{n,s}
        \int_{\mathbb R^n\setminus B_{3/2}}
        w(y)(1-\eta(y))
        \left[
        D_x^\mu (\frac{e^{-\lambda|x-y|}}{|x-y|^{n+2s}})
        -
        D_x^\mu (\frac{e^{-\lambda|x'-y|}}{|x'-y|^{n+2s}})
        \right]\,dy .
        \end{aligned}
        \]

        % \color{red}
        % By the smoothness of the kernel with respect to \(x\), together with the fact that \(|x-y|\ge \tfrac12\), we obtain
        % \[
        % \left|
        % \frac{e^{-\lambda|x-y|}}{|x-y|^{n+2s}}
        % -
        % \frac{e^{-\lambda|x'-y|}}{|x'-y|^{n+2s}}
        % \right|
        % \le
        % C|x-x'|(1+|y|)^{-n-2s-1}.
        % \]

        % Since \(|x-x'|\le 2\), it follows that $ |x-x'|\le C|x-x'|^\beta.$
        % Therefore,
        % \[
        %   \left|(-\Delta)^s_\lambda \phi(x)-(-\Delta)^s_\lambda \phi(x')\right|
        %   \le
        %   C|x-x'|^\beta
        %   \int_{\mathbb{R}^n}
        %   |w(y)|(1+|y|)^{-n-2s}\,dy
        %   \]
        %   \[
        %   \frac{\left|(-\Delta)^s_\lambda \phi(x)-(-\Delta)^s_\lambda \phi(x')\right|}{|x-x'|^\beta}
        %   \le
        %   C\left\|(1+|y|)^{-n-2s}w(y)\right\|_{L^1(\mathbb{R}^n)}
        % \]

        % \[
        %     \bigl[(-\Delta)^s_\lambda\bigl(w(1-\eta)\bigr)\bigr]_{C^\beta(B_1)}
        %   \le
        %     C\left\|(1+|y|)^{-n-2s}w(y)\right\|_{L^1(\mathbb{R}^n)}.
        % \]
  
        For every point \(\xi\) on the segment joining \(x\) and \(x'\), since \(B_1\) is
        convex, we have \(\xi\in B_1\). Hence, for
        \(y\in\mathbb R^n\setminus B_{3/2}\),
        $|\xi-y|\ge \frac12.$
        By the mean value theorem,
        \[
        \begin{aligned}
        \left|
        D_x^\mu (\frac{e^{-\lambda|x-y|}}{|x-y|^{n+2s}})
        -
        D_x^\mu (\frac{e^{-\lambda|x'-y|}}{|x'-y|^{n+2s}})
        \right| 
        \le
        |x-x'|
        \sup_{\xi\in[x,x']}
        \left|\nabla_xD_x^\mu (\frac{e^{-\lambda|\xi-y|}}{|\xi-y|^{n+2s}})\right|.
        \end{aligned}
        \]
        Using the preceding kernel derivative estimate~\eqref{daoshu} with \(|\mu|+1\), we get
        \[
        \sup_{\xi\in[x,x']}
        \left|\nabla_xD_x^\mu (\frac{e^{-\lambda|\xi-y|}}{|\xi-y|^{n+2s}})\right|
        \le
        C(1+|y|)^{-n-2s}.
        \]
        Thus
        \begin{align*}
          \left|
              D_x^\mu (\frac{e^{-\lambda|x-y|}}{|x-y|^{n+2s}})
            -
              D_x^\mu (\frac{e^{-\lambda|x'-y|}}{|x'-y|^{n+2s}})
          \right|
        &\le
            C|x-x'|(1+|y|)^{-n-2s}\\
        &\le
              C|x-x'|^\theta(1+|y|)^{-n-2s}.
        \end{align*}
        % Since \(x,x'\in B_1\), we have \(|x-x'|\le2\). Because \(0<\theta<1\),
        % \[
        % |x-x'|
        % =
        % |x-x'|^\theta |x-x'|^{1-\theta}
        % \le
        % 2^{1-\theta}|x-x'|^\theta.
        % \]
        % Therefore,
      %   \[
      %   \left|
      %  D_x^\mu (\frac{e^{-\lambda|x-y|}}{|x-y|^{n+2s}})
      %   -
      %   D_x^\mu (\frac{e^{-\lambda|x'-y|}}{|x'-y|^{n+2s}})
      %   \right|
      %   \le
      %   C|x-x'|^\theta(1+|y|)^{-n-2s}.
      %   \]

        So we obtain that 
        % \[
        % \begin{aligned}
        % |D^\mu F(x)-D^\mu F(x')|
        % &\le
        % C|x-x'|^\theta
        % \int_{\mathbb R^n\setminus B_{3/2}}
        % |w(y)|(1-\eta(y))(1+|y|)^{-n-2s}\,dy  \\
        % &\le
        % C|x-x'|^\theta
        % \left\|(1+|y|)^{-n-2s}w(y)\right\|_{L^1(\mathbb R^n)}.
        % \end{aligned}
        % \]
        % 
        % Therefore, 
        for all \(|\mu|=m\), 
        \begin{equation}\label{banfanshu}
            [D^\mu (-\Delta)^s_\lambda\bigl(w(1-\eta)\bigr)]_{C^\theta(B_1)}
        \le
            C\left\|(1+|y|)^{-n-2s}w(y)\right\|_{L^1(\mathbb R^n)}.
        \end{equation}

        Combining~\eqref{wuqiong} and~\eqref{banfanshu}, we conclude that
        \[
        \left\|
        (-\Delta)^s_\lambda\bigl(w(1-\eta)\bigr)
        \right\|_{C^\beta(\overline B_1)}
        \le
        C
        \left\|(1+|y|)^{-n-2s}w(y)\right\|_{L^1(\mathbb R^n)}.
        \]

  \color{black}
        From expression above we obtain that,
        \[
        \|\widetilde{h}\|_{C^\beta(\overline{B}_1)}
        \le
        C\left(
        \|h\|_{C^\beta(\overline{B}_2)}
        +
        \left\|(1+|y|)^{-n-2s}w(y)\right\|_{L^1(\mathbb{R}^n)}
        \right).
        \]

        Applying Lemma~\ref{tempered-RO2.2} to \(\widetilde{w}\), we obtain
        \[
        \|\widetilde{w}\|_{C^{\beta+2s}(\overline{B}_{1/2})}
        \le
        C\left(
        \|\widetilde{w}\|_{C^\beta(\mathbb{R}^n)}
        +
        \|\widetilde{h}\|_{C^\beta(\overline{B}_1)}
        \right)
        \]
        \[
        \le
        C\left(
        \|w\|_{C^\beta(\overline{B}_2)}
        +
        \|h\|_{C^\beta(\overline{B}_2)}
        +
        \left\|(1+|y|)^{-n-2s}w(y)\right\|_{L^1(\mathbb{R}^n)}
        \right).
        \]
      Since \(\widetilde{w}=w\) in \(B_{1/2}\), it follows that
        \[
        \|w\|_{C^{\beta+2s}(\overline{B}_{1/2})}
        \le
        C\left(
        \left\|(1+|y|)^{-n-2s}w(y)\right\|_{L^1(\mathbb{R}^n)}
        +
        \|w\|_{C^\beta(\overline{B}_2)}
        +
        \|h\|_{C^\beta(\overline{B}_2)}
        \right).
        \]

        This completes the proof.
    \end{proof}

\begin{lemma}\label{tempered2.4-general}
    Let \(n\ge1\), \(0<s<1\), \(\beta>0\), and \(\lambda>0\).
    Let \(\frac{1}{2} \le r < R \le 2\). Assume that \(h\in C^\beta(\overline B_R)\) and that
    \(w\in C^\infty(\mathbb R^n)\) solves
    $(-\Delta)^s_\lambda w=h$ in  $B_R.$
    Assume also that
    $\beta\notin\mathbb N$, $\beta+2s\notin\mathbb N.$
    Then
    \[
    \|w\|_{C^{\beta+2s}(\overline B_r)}
    \le
    C(R-r)^{-N}\left(
    \left\|(1+|x|)^{-n-2s}w(x)\right\|_{L^1(\mathbb R^n)}
    +
    \|w\|_{C^\beta(\overline B_R)}
    +
    \|h\|_{C^\beta(\overline B_R)}
    \right),
    \]
    where \(C=C(n,s,\beta,\lambda)\),  and  $N=N(n, s, \beta)$.
\end{lemma}

\begin{proof}
    Set
    $d:=R-r>0,$
    $R_0:=r+\frac d3,$
    $R_1:=r+\frac{2d}{3}.$
    Then
    $
    r<R_0<R_1<R,
    $
    and
    $
    R_0-r=R_1-R_0=R-R_1=\frac d3:=d_1.
    $

    Choose \(\eta\in C_c^\infty(\mathbb R^n)\) such that
    \begin{equation}\label{etachoose}
    0\le \eta\le1\quad\text{in } \mathbb R^n,
    \qquad
    \eta\equiv1\quad\text{in }B_{R_1},
    \qquad
    \eta\equiv0\quad\text{in }\mathbb R^n\setminus B_R.
    \end{equation}

    We now construct the cut-off function~\eqref{etachoose} explicitly, we use the method of mollifiers; see, 
    for example, \cite[p.~713]{EvansPDE}.
    Define \(\psi:\mathbb R^n\to\mathbb R\) by
    \[
    \psi(x):=
    \begin{cases}
    D\exp\left(\dfrac{1}{\|x\|^2-1}\right),
    & \text{if } \|x\|<1,\\[6pt]
    0,
    & \text{if } \|x\|\ge 1,
    \end{cases}
    \]
    where \(D>0\) is selected so that
    $
    \int_{\mathbb R^n}\psi(x)\,dx=1.
    $
    The mollifier \(\psi\in C^\infty\) is positive with support in the unit ball about the origin. 

For $\varepsilon:=\frac{d_1}{4},$
set
$
    \psi_\varepsilon(x)
:=
    \frac{1}{\varepsilon^n}
    \psi\left(\frac{x}{\varepsilon}\right).
$
Define
\[
\eta(x)
:=
\bigl(\psi_\varepsilon * \mathbf 1_{B_{R_1+\varepsilon}}\bigr)(x)
=
\int_{B_{R_1+\varepsilon}}
\psi_\varepsilon(x-y)\,dy .
\]
Then \(\eta\in C_c^\infty(\mathbb R^n)\) and \(0\le \eta\le 1\).

We claim that
\[
\eta\equiv 1 \quad \text{in } B_{R_1},
\qquad
\eta\equiv 0 \quad \text{in } \mathbb R^n\setminus B_R.
\]
Indeed, if \(x\in B_{R_1}\), then
\[
B_\varepsilon(x)\subset B_{R_1+\varepsilon}.
\]
Since \(\operatorname{supp}\psi_\varepsilon\subset B_\varepsilon\), we have
\[
\eta(x)
=
\int_{\mathbb R^n}\psi_\varepsilon(x-y)\,dy
=
1.
\]
On the other hand, if \(\eta(x)\neq0\), then there exists
\(y\in B_{R_1+\varepsilon}\) such that
$
|x-y|<\varepsilon.
$
Hence
\[
|x|
\le |y|+|x-y|
<
R_1+2\varepsilon
=
R_1+\frac{d_1}{2}
<
R_1+d_1
=
R.
\]
Therefore,
  $\operatorname{supp}\eta\subset B_R,$
which gives
  $
    \eta\equiv0
    \quad \text{in } \mathbb R^n\setminus B_R.
  $

Finally, for every multi-index \(\mu\), we have
$
D^\mu\eta
=
D^\mu\psi_\varepsilon * \mathbf 1_{B_{R_1+\varepsilon}}.
$
Thus
\[
\|D^\mu\eta\|_{L^\infty(\mathbb R^n)}
\le
\|D^\mu\psi_\varepsilon\|_{L^1(\mathbb R^n)}.
\]
Since
    $
        D^\mu\psi_\varepsilon(x)
    =
        \varepsilon^{-n-|\mu|}
       (D^\mu\psi)\left(\frac{x}{\varepsilon}\right),
    $
it follows that
\[
\|D^\mu\psi_\varepsilon\|_{L^1(\mathbb R^n)}
=
\varepsilon^{-|\mu|}
\|D^\mu\psi\|_{L^1(\mathbb R^n)}.
\]
Therefore,
\[
\|D^\mu\eta\|_{L^\infty(\mathbb R^n)}
% \le
% C_{\mu} \varepsilon^{-|\mu|}
% =
% C_\mu d_1^{-|\mu|}
\le
C_\mu d^{-|\mu|},
\]
where \(C_\mu\) is independent of \(R\), and \(r\).
    
    Thus we  choose \(\eta\) so that for every multi-index \(\mu\),
    \begin{equation}\label{derivative}
    \|D^\mu\eta\|_{L^\infty(\mathbb R^n)}
    \le
    C_\mu d^{-|\mu|}.
    \end{equation}

    Define
    $
    \widetilde w:=\eta w =w-w(1-\eta),
    $
    then
    $
    \widetilde w\in C^\infty(\mathbb R^n),
    $
    $
    \widetilde w=w$ in $B_{R_1}.$
    By~\eqref{xianxin} , we have
    \[
    (-\Delta)^s_\lambda\widetilde w
    =
    (-\Delta)^s_\lambda w
    -
    (-\Delta)^s_\lambda\bigl(w(1-\eta)\bigr).
    \]
    Since \(B_{R_0}\subset B_R\) and \(w\) solves
    $(-\Delta)^s_\lambda w=h$ in $B_R,$
    we obtain, in \(B_{R_0}\),
    \[
    (-\Delta)^s_\lambda\widetilde w
    =
    \widetilde h
    :=
    h-
    (-\Delta)^s_\lambda\bigl(w(1-\eta)\bigr).
    \]

    Since \(\widetilde w=\eta w\) and \(\eta\) is supported in \(B_R\), the
    standard product estimate in H\"older spaces gives
    \[
    \|\widetilde w\|_{C^\beta(\mathbb R^n)}
    \le
      \|\eta\|_{C^\beta(\mathbb R^n)}
      \|w\|_{C^\beta(\overline B_R)}.
    \]

    Using~\eqref{derivative}, we have
    \[
    \|\eta\|_{C^\beta(\mathbb R^n)}
    \le
    C (R-r)^{-\beta}.
    \]
    
    % Therefore,
    % \[
    % \|\widetilde w\|_{C^\beta(\mathbb R^n)}
    % \le
    % C (R-r)^{-\beta}
    % \|w\|_{C^\beta(\overline B_R)}.
    % \]

    Consequently,
    \[
    \|\widetilde w\|_{C^\beta(\mathbb R^n)}
    \le
    C (R-r)^{-N}
    \|w\|_{C^\beta(\overline B_R)},
    \]
    where \(N=N(\beta)\).

    % We first estimate \(\widetilde w\), since \(\eta\) is smooth and supported in \(B_R\),
    % \[
    % \|\widetilde w\|_{C^\beta(\mathbb R^n)}
    % \le
    % C\|w\|_{C^\beta(\overline B_R)},
    % \]
    % where \(C\) depends on \(n,\beta,r,R\). More explicitly, the dependence on
    % \(R-r\) comes from the derivative bounds of the cut-off function.

    Next we estimate $
    \left\|
        (-\Delta)^s_\lambda\bigl(w(1-\eta)\bigr)
    \right\|_{C^\beta(\overline B_{R_0})}$.
    Let $F(x):=(-\Delta)^s_\lambda\bigl(w(1-\eta)\bigr)(x)$, and  $K_1(z):=\frac{e^{-\lambda|z|}}{|z|^{n+2s}}$.
    For every \(x\in B_{R_0}\),
    \[
      F(x)
    = 
      (-\Delta)^s_\lambda\bigl(w(1-\eta)\bigr)(x)
    =
      -c_{n,s}
      \int_{\mathbb R^n\setminus B_{R_1}}
      w(y)(1-\eta(y))
      \frac{e^{-\lambda|x-y|}}{|x-y|^{n+2s}}\,dy.
    \]

    For \(x\in B_{R_0}\) and \(y\in\mathbb R^n\setminus B_{R_1}\), we have
    $
    |x-y|
    \ge
    R_1-R_0
    =
    \frac d3,
    $
    hence no singularity of the kernel is encountered in the region of integration.

      For every multi-index \(\mu\),
      there exists \(C_\mu>0\), depending only on \(n,s,\lambda,\mu\), such that
      % \[
      %     |D^\mu K_1(x-y)|
      % \le
      %   C_\mu d^{-n-2s-|\mu|}(1+|y|)^{-n-2s}.
      % \]

% For every multi-index \(\mu\), there exists \(C_\mu>0\), depending only on
% \(n,s,\lambda,\mu\), such that
% \[
% |D^\mu K_1(z)|
% \le
% C_\mu |z|^{-{(n+2s)}-|\mu|},
% \qquad z\neq0.
% \]
% Hence
% \[
% |D_x^\mu K_1(x-y)|
% \le
% C_\mu |x-y|^{-{(n+2s)}-|\mu|}.
% \]

% Since \(x\in B_{R_0}\) and
% \(y\in\mathbb R^n\setminus B_{R_1}\), we have
% \[
% |x-y|\ge R_1-R_0=\frac d3.
% \]
% Moreover, since \(R_0<R \le 2\),
% \[
% 1+|y|
% \le
% 1+|x-y|+|x|
% \le
% 3(1+|x-y|).
% \]
% Therefore,
% \[
% \begin{aligned}
% |x-y|^{-{(n+2s)}-|\mu|}(1+|y|)^{(n+2s)}
% &\le
% C |x-y|^{-{(n+2s)}-|\mu|}(1+|x-y|)^{n+2s}        \\
% &=
% C |x-y|^{-|\mu|}
% \left(1+\frac1{|x-y|}\right)^{n+2s}.
% \end{aligned}
% \]
% Using \(|x-y|\ge d/3\) and \(0<d \le 3/2\), we get
% \[
% |x-y|^{-|\mu|}
% \le
% C d^{-|\mu|},
% \qquad
% \left(1+\frac1{|x-y|}\right)^{n+2s}
% \le
% C d^{-{(n+2s)}}.
% \]

% Consequently,
% \[
% |x-y|^{-{(n+2s)}-|\mu|}(1+|y|)^{(n+2s)}
% \le
% C d^{-{(n+2s)}-|\mu|}.
% \]
% Equivalently,
% \[
% |x-y|^{-{(n+2s)}-|\mu|}
% \le
% C d^{-{(n+2s)}-|\mu|}(1+|y|)^{-{(n+2s)}}.
% \]
% Thus
% \[
% |D_x^\mu K_1(x-y)|
% \le
% C_\mu d^{-n-2s-|\mu|}
% (1+|y|)^{-n-2s}.
% \]
    % Consequently, for every \(|\mu|\le m\),
    \begin{equation}\label{Lbounds}
        \|D^\mu F\|_{L^\infty(B_{R_0})}
    \le
        C d^{-n-2s-m}
        \left\|(1+|y|)^{-n-2s}w(y)\right\|_{L^1(\mathbb R^n)}.
    \end{equation}     

    It remains to estimate the H\"older seminorms. Let \(|\mu|=m\) and
    \(x,x'\in B_{R_0}\). By the mean value theorem,
    \[
    \begin{aligned}
      \left|
      D^\mu K_1(x-y)-D^\mu K_1(x'-y)
      \right|
    \le
      |x-x'|
      \sup_{\xi\in[x,x']}
      |\nabla D^\mu K_1(\xi-y)|.
    \end{aligned}
    \]
    Since \(\xi\in B_{R_0}\), we again have
    $
        |\xi-y|\ge \frac d3.
    $
    Thus
    \[
    |\nabla D^\mu K_1(\xi-y)|
    \le
    C d^{-n-2s-m-1}(1+|y|)^{-n-2s}.
    \]
    Using also \(R_0<2\), we have
    \[
    |x-x'|\le 4^{1-\theta}|x-x'|^\theta.
    \]

    Hence
    \[
    \left|
    D^\mu K_1(x-y)-D^\mu K_1(x'-y)
    \right|
    \le
    C |x-x'|^\theta
    d^{-n-2s-m-1}(1+|y|)^{-n-2s}.
    \]
    Substituting this into the integral representation of
    \(D^\mu F(x)-D^\mu F(x')\), we obtain
    \begin{equation}\label{FHolderseminorm}
        [D^\mu F]_{C^\theta(B_{R_0})}
    \le
         C d^{-n-2s-m-1}
        \left\|(1+|y|)^{-n-2s}w(y)\right\|_{L^1(\mathbb R^n)}.
    \end{equation}
    Combining~\eqref{Lbounds} and~\eqref{FHolderseminorm} gives
    \[
        \|(-\Delta)^s_\lambda\bigl(w(1-\eta)\bigr)\|_{C^\beta(\overline B_{R_0})}
    \le
        C d^{-N_1}
        \left\|(1+|y|)^{-n-2s}w(y)\right\|_{L^1(\mathbb R^n)},
    \]
    where one may take
    $N_1=n+2s+m+1$, $m=\lfloor\beta\rfloor.$

    Hence,
    \[
    \begin{aligned}
    \|\widetilde h\|_{C^\beta(\overline B_{R_0})}
    &\le
    \|h\|_{C^\beta(\overline B_{R_0})}
    +
    \left\|
    (-\Delta)^s_\lambda\bigl(w(1-\eta)\bigr)
    \right\|_{C^\beta(\overline B_{R_0})} \\
    &\le
    C d^{-N}
      \left(
      \|h\|_{C^\beta(\overline B_R)}
    +
      \left\|(1+|y|)^{-n-2s}w(y)\right\|_{L^1(\mathbb R^n)}
      \right).
    \end{aligned}
    \]
 
    Then we obtain modified Lemma~\ref{tempered-RO2.2}, that is 

    % Now we apply the scaled interior Schauder estimate corresponding to
    % Lemma~\ref{tempered-RO2.2}, we have
    \[
        \|\widetilde w\|_{C^{\beta+2s}(\overline B_r)}
    \le
    C d^{-N}
      \left(
        \|\widetilde w\|_{C^\beta(\mathbb R^n)}
      +
        \|\widetilde h\|_{C^\beta(\overline B_{R_0})}
      \right),
    \]
    where \(C=C(n,s,\beta,\lambda)\),  and  $N=N(n, s, \beta)$.

    Using the estimates above, we obtain
    \[
    \begin{aligned}
      \|\widetilde w\|_{C^{\beta+2s}(\overline B_r)}
    &\le
      C(R-r)^{-N}\left(
      \left\|(1+|y|)^{-n-2s}w(y)\right\|_{L^1(\mathbb R^n)}
    +
      \|w\|_{C^\beta(\overline B_R)}
    +
      \|h\|_{C^\beta(\overline B_R)}
    \right).
    \end{aligned}
    \]

    Finally, since \(\widetilde w=w\) in \(B_r\), we conclude that
    \[
    \|w\|_{C^{\beta+2s}(\overline B_r)}
    \le
    C(R-r)^{-N}\left(
      \left\|(1+|y|)^{-n-2s}w(y)\right\|_{L^1(\mathbb R^n)}
    +
      \|w\|_{C^\beta(\overline B_R)}
    +
      \|h\|_{C^\beta(\overline B_R)}
    \right).
    \]
    The proof is complete.
\end{proof}

\begin{lemma}\label{tempered2.10}
    Let \(U \subset \mathbb{R}^n\) be a bounded open set with nonempty boundary. Let
    $ 0<s<1$,  $0<\alpha<1$, $\alpha<2s,$ $\beta > 0 $
    and assume that neither $\beta$ nor $\beta+2s$ is an integer.
    Suppose that \(w\in C^\alpha(\mathbb R^n)\), \(h\in C^\beta(U)\),
    and
    \[
        (-\Delta)^s_\lambda w=h
        \qquad\text{in }\mathcal D'(U).
    \]
    If
    \[
        \|h\|^{(2s-\alpha)}_{\beta;U}<\infty,
    \]
    then there exists a constant
    $C=C(n,s,\alpha,\beta,\lambda,U)>0$
    such that
    \begin{equation}\label{eq:tempered2.7}
        \|w\|^{(-\alpha)}_{\beta+2s;U}
        \le
        C\left(
        \|w\|_{C^\alpha(\mathbb{R}^n)}
        +
        \|h\|^{(2s-\alpha)}_{\beta;U}
        \right).
    \end{equation}
\end{lemma}

  \begin{proof}
        We first prove the estimate under the additional assumption
        \(w\in C^\infty(\mathbb R^n)\). The removal of this auxiliary
        smoothness assumption is carried out at the end of the proof.
        We first estimate the $C^{\beta+2s}$ norm of \(w\) in balls \(B_R(x_0)\).
        
        Fix \(x_0 \in U\) and set$R:=\frac{d_{x_0}}{2},$
        so that \(B_R(x_0)\subset U\). Define $\widetilde{w}(y):=w(x_0+Ry)-w(x_0).$
        We claim that
        \[
        \|\widetilde{w}\|_{C^\alpha(B_1)}
        \le
        R^\alpha [w]_{C^\alpha(\mathbb{R}^n)}.
        \]

        Indeed, for any \(y_1,y_2\in B_1\),
        $
        \widetilde{w}(y_1)-\widetilde{w}(y_2)
        =
        w(x_0+Ry_1)-w(x_0+Ry_2),
        $
        and therefore
        \[
        |w(x_0+Ry_1)-w(x_0+Ry_2)|
        \le
        [w]_{C^\alpha(\mathbb{R}^n)}\,|Ry_1-Ry_2|^\alpha
        =
        [w]_{C^\alpha(\mathbb{R}^n)}\,R^\alpha |y_1-y_2|^\alpha.
        \]
        Hence,
        \begin{equation}\label{wwai1}
        [\widetilde{w}]_{C^\alpha(B_1)}
        \le
        R^\alpha [w]_{C^\alpha(\mathbb{R}^n)}.
        \end{equation}
        Moreover, for every \(y\in B_1\),
        \[
        |\widetilde{w}(y)|
        =
        |w(x_0+Ry)-w(x_0)|
        \le
        [w]_{C^\alpha(\mathbb{R}^n)}\,|Ry|^\alpha
        \le
        R^\alpha [w]_{C^\alpha(\mathbb{R}^n)}.
        \]
        Thus,
        \begin{equation}\label{wwai2}
        \sup_{y\in B_1} |\widetilde{w}(y)|
        \le
        R^\alpha [w]_{C^\alpha(\mathbb{R}^n)}.
        \end{equation}

        Combining the above two estimates~\eqref{wwai1} and ~\eqref{wwai2}, we conclude that
        \begin{equation}\label{wwaialpfa}
          \|\widetilde{w}\|_{C^\alpha(B_1)}
          \le
          R^\alpha [w]_{C^\alpha(\mathbb{R}^n)}
        \end{equation}
        and 
        \[
        \left\|(1+|y|)^{-n-2s}\widetilde{w}(y)\right\|_{L^1(\mathbb{R}^n)}
        \le
        C(n,s)\,R^\alpha [w]_{C^\alpha(\mathbb{R}^n)}.
        \]

        Note that, for \(y\in B_1\),
        \[
        (-\Delta)^s_{\lambda R}\widetilde{w}(y)
        =
        R^{2s}(-\Delta)^s_{\lambda}w(x_0+Ry).
        \]

        Hence
        \[
        \|(-\Delta)^s_{\lambda R}\widetilde{w}\|_{L^\infty(B_1)}
        =
        R^{2s}\|(-\Delta)^s_{\lambda}w\|_{L^\infty(B_R(x_0))}
        \le
        R^\alpha \|(-\Delta)^s_{\lambda}w\|^{(2s-\alpha)}_{\beta;U},
        \]
        and
        \[
        [(-\Delta)^s_{\lambda R}\widetilde{w}]_{C^\beta(B_1)}
        =
        R^{2s+\beta}[(-\Delta)^s_{\lambda}w]_{C^\beta(B_R(x_0))}
        \le
        R^\alpha \|(-\Delta)^s_{\lambda}w\|^{(2s-\alpha)}_{\beta;U}.
        \]

        Therefore,
        \[
        \|(-\Delta)^s_{\lambda R}\widetilde{w}\|_{C^\beta(B_1)}
        \le
        C R^\alpha \|(-\Delta)^s_{\lambda}w\|^{(2s-\alpha)}_{\beta;U}.
        \]

        If $\beta \le \alpha$, 
        then the embedding
        \(C^\alpha(B_1)\hookrightarrow C^\beta(B_1)\), together with
        Lemma~\ref{tempered2.4}, gives

        % using Lemma~\ref{tempered2.4} and  the estimate \(\widetilde{w}\in C^\beta(B_1)\) follows 
        % directly from \(\widetilde{w}\in C^\alpha(B_1)\),
        % we obtain that 
        \[
        \|\widetilde{w}\|_{C^{\beta+2s}(\overline{B}_{1/2})}
        \le
        C R^\alpha
        \left(
        [w]_{C^\alpha(\mathbb{R}^n)}
        +
        \|(-\Delta)^s_\lambda w\|^{(2s-\alpha)}_{\beta;U}
        \right),
        \]

        % When \(\beta>\alpha\), the \(C^\beta\) regularity of \(\widetilde{w}\) is obtained by a finite bootstrap argument. More precisely, choose a finite sequence
        % $  \alpha=\sigma_0<\sigma_1<\cdots<\sigma_m=\beta$
        % such that
        % $\sigma_{j+1}<\sigma_j+2s,\,\text{for } j=0,1,\dots,m-1.$
        % By iteratively applying Proposition~\ref{tempered-RO2.2} on a nested family of concentric balls, 
        % one improves the regularity successively from \(C^{\sigma_j}\) to \(C^{\sigma_j+2s}\).  
        % Thus, using Lemma~\ref{tempered2.4} we obtain that

        If \(\beta>\alpha\), we set
    \[
    M_0:=
    [w]_{C^\alpha(\mathbb R^n)}
    +
    \|(-\Delta)^s_\lambda w\|^{(2s-\alpha)}_{\beta;U} 
    \]
and 
    \[
    A(\rho):=
    \|\widetilde w\|_{C^{\beta+2s}(\overline B_\rho)}.
    \]
    For every
    \[
    \frac12\le r<R<2,
    \]
    apply Lemma~\ref{tempered2.4-general} in the balls
    \(B_r \subset B_R\), gives
    \begin{equation}\label{local form of lemma 4.11}
    \|\widetilde w\|_{C^{\beta+2s}(\overline B_r)}
    \le
    C(R-r)^{-N}
    \left(
    R^\alpha M_0
    +
    \|\widetilde w\|_{C^\beta(\overline B_R)}
    \right),
    \end{equation}
where \(C>0\) and \(N>0\) depend only on the fixed parameters.

Since
$
\alpha<\beta<\beta+2s,
$
the interpolation inequality gives, for every \(\varepsilon>0\),
    \begin{equation}\label{wwaichazhi}
      \|\widetilde w\|_{C^\beta(\overline B_R)}
      \le
      \varepsilon
      \|\widetilde w\|_{C^{\beta+2s}(\overline B_R)}
      +
      C_\varepsilon
      \|\widetilde w\|_{C^\alpha(\overline B_R)}.
    \end{equation}

    By the same  scaling argument as~\eqref{wwaialpfa}, we obtain
    \begin{equation}\label{wwaiBR}
        \|\widetilde w\|_{C^\alpha(\overline B_R)}
    \le
        C R^\alpha [w]_{C^\alpha(\mathbb{R}^n)}
    \le
        C R^\alpha M_0,
    \end{equation}
    combining~\eqref{wwaichazhi} and~\eqref{wwaiBR}, we obtain
    \[
    \|\widetilde w\|_{C^\beta(\overline B_R)}
    \le
    \varepsilon \|\widetilde w\|_{C^{\beta+2s}(\overline B_R)}
    +
    C_\varepsilon R^\alpha M_0.
    \]
    Consequently, the local estimate in~\eqref{local form of lemma 4.11} yields
    % Therefore, by~\eqref{local form of lemma 4.11}, we have
      \[
      \|\widetilde w\|_{C^{\beta+2s}(\overline B_r)}
      \le
      C(R-r)^{-N}\varepsilon \|\widetilde w\|_{C^{\beta+2s}(\overline B_R)}
      +
      C(R-r)^{-N}C_\varepsilon R^\alpha M_0.
      \]

      Now assume that, for \(1/2\le r<R<2\) and fix $ N = N_{0}$, we have obtained
      \begin{equation}\label{houCzong}
      A(r)
      \le
      C_{1}(R-r)^{-N_{0}}\varepsilon A(R)
      +
      C_{2}(R-r)^{-N_{0}}C_\varepsilon R^\alpha M_0.
      \end{equation}
      where \(\varepsilon>0\), \(N_0>0\), and \(C_1,C_2\) 
      are independent of fixed parameters.

      For a fixed number \(\delta\in(0,1)\), choose
      \begin{equation}\label{varepsilonxuanqu}
          \varepsilon
        =
          \frac{\delta}{C_{1}}(R-r)^{N_0}.
      \end{equation}
      Then
      \begin{equation}\label{houC_1}
          C_{1}(R-r)^{-N_{0}}\varepsilon=\delta.
      \end{equation}

      Moreover, by the form of the interpolation inequality, 
      there exists
      \(q>0\)   such that
      \[
      C_\varepsilon\le \varepsilon^{-q}.
      \]

      Hence
      \begin{equation}\label{houC_2}
          C_{2}(R-r)^{-N_{0}}C_\varepsilon
       \le
          C_\delta (R-r)^{-\widetilde N},
     \end{equation}
      where
      $
      \widetilde N:=(1+q)N_0
      $
      and
      $
      C_\delta
      :=
      C_{2}\left(\frac{C_1}{\delta}\right)^q.
      $
      Thus by~\eqref{houCzong}, \eqref{houC_1} and~\eqref{houC_2}, we obtain
      \[
      A(r)
      \le
      \delta A(R)
      +
      C_\delta (R-r)^{-\widetilde N}R^\alpha M_0.
      \]

      We now fix \(\delta>0\) sufficiently small. More precisely, choose
      $
      \delta_0:=2^{-(\widetilde N+1)},
      $
      then
      $
      \delta_0 2^{\widetilde N}<1.
      $

      Let
      \[
      r_j:=\frac12+\frac32(1-2^{-j}),
      \qquad j=0,1,2,\dots.
      \]
      Then
      \[
      r_0=\frac12,
      \qquad
      r_j\uparrow 2,
      \]
      and
      \[
      r_{j+1}-r_j=\frac32\,2^{-(j+1)}.
      \]

        Applying the preceding estimate with \(r=r_j\) and \(R=r_{j+1}\), we get
        \[
        A(r_j)
        \le
        \delta_0 A(r_{j+1})
        +
        C_{\delta_0}2^{\widetilde N j}R^\alpha M_0.
        \]
        Iterating yields
        \[
        A(r_0)
        \le
        \delta_0^m A(r_m)
        +
        C_{\delta_0}R^\alpha M_0
        \sum_{j=0}^{m-1}
        \left(\delta_0 2^{\widetilde N}\right)^j.
        \]
        Since
        $
        \delta_0 2^{\widetilde N}<1,
        $
        the series is uniformly bounded. Moreover, since \(A(r_m)\le A(2)<\infty\), and hence
        \[
        \delta_0^m A(r_m)\to0
        \qquad\text{as }m\to\infty.
        \]
        Letting \(m\to\infty\), we conclude that
        \[
        A(r_0)
        \le
        C R^\alpha M_0.
        \]
        Since \(r_0=1/2\), this gives 
        \[
        \|\widetilde{w}\|_{C^{\beta+2s}(\overline{B}_{1/2})}
        \le
        C R^\alpha
        \left(
        [w]_{C^\alpha(\mathbb{R}^n)}
        +
        \|(-\Delta)^s_\lambda w\|^{(2s-\alpha)}_{\beta;U}
        \right),
        \]
        where the constant \(C\) depends only on \(n\), \(s\), \(\alpha\), \(\beta\), \(\lambda\), and  \(U\). 
        
        Combining the cases \(\beta\le \alpha\) and \(\beta>\alpha\), we conclude that
        \[
        \|\widetilde{w}\|_{C^{\beta+2s}(\overline{B}_{1/2})}
        \le
        C R^\alpha
        \left(
        [w]_{C^\alpha(\mathbb{R}^n)}
        +
        \|(-\Delta)^s_\lambda w\|^{(2s-\alpha)}_{\beta;U}
        \right).
        \]
        \color{black}
        Scaling back we obtain
        \begin{equation*}
            \sum_{l=1}^{k} R^{\,l-\alpha}\,\|D^l w\|_{L^\infty(B_{R/2}(x_0))}
            \;+\;
            R^{\,2s+\beta-\alpha}\,[w]_{C^{\beta+2s}(\overline{B}_{R/2}(x_0))}
        \end{equation*}  
        \begin{equation}\label{eq:tempered2.8}
            \le
            C\left(
            \|w\|_{C^\alpha(\mathbb{R}^n)}
            +
            \|(-\Delta)^s_\lambda w\|^{(2s-\alpha)}_{\beta;U}
            \right),
        \end{equation}
      where \(k\) denotes the greatest integer less than \(\beta+2s\). This bound holds, with the same constant \(C\), for each ball \(B_R(x_0)\), \(x_0\in U\), where \(R=d_{x_0}/2\).
     
        Next we claim that if \eqref{eq:tempered2.8} holds for each ball \(B_{d_x/2}(x)\), \(x\in U\), 
        then \eqref{eq:tempered2.7} holds. It is clear that this already yields
        \begin{equation}\label{eq:tempered2.9}
          \sum_{l=1}^{k}\sup_{x\in U}\left(d_x^{\,l-\alpha}|D^l w(x)|\right)
          \le
          C\left(\|w\|_{C^\alpha(\mathbb{R}^n)}+\|(-\Delta)^s_\lambda w\|^{(2s-\alpha)}_{\beta;U}\right),
        \end{equation}
        where \(k\) is the greatest integer less than \(\beta+2s\).

      To prove this claim we only have to control \([w]^{(-\alpha)}_{\beta+2s;U}\) --- 
      see Definition~\ref{Definition_tempered1.3}. Let \(\gamma\in(0,1)\) be such that \(\beta+2s=k+\gamma\).
      We next take any \(x,y\in U\). Without loss of generality, we may assume that \(d_x \le d_y\). Then
      $d_{x,y}=d_x$ and \(|x-y|\ge d_x/2\). We next estimate
      \[
      d_x^{\beta+2s-\alpha}\frac{|D^k w(x)-D^k w(y)|}{|x-y|^\gamma}.
      \] This will yield the bound for \([w]^{(-\alpha)}_{\beta+2s;U}\), because if \(|x-y|<d_x/2\) then \(y\in B_{d_x/2}(x)\), and that case is done.

      We proceed differently in the cases \(k=0\) and \(k\ge 1\). If \(k=0\), then
      \[
      d_x^{\beta+2s-\alpha}\frac{w(x)-w(y)}{|x-y|^{2s+\beta}}
      =
      \left(\frac{d_x}{|x-y|}\right)^{\beta+2s-\alpha}
      \frac{w(x)-w(y)}{|x-y|^\alpha}
      \le
      C\|w\|_{C^\alpha(\mathbb{R}^n)}.
      \]

      If \(k\ge 1\), then
      \begin{align*}
         d_x^{\beta+2s-\alpha}\frac{|D^k w(x)-D^k w(y)|}{|x-y|^\gamma}
      &\le
          \left(\frac{d_x}{|x-y|}\right)^\gamma
          d_x^{\beta+2s-\alpha-\gamma}|D^k w(x)-D^k w(y)|\\
      &\le
         C\|w\|^{(-\alpha)}_{k;U},
     \end{align*}
      where we have used that \(\beta+2s-\alpha-\gamma=k-\alpha\).

      Noting that for \(x\in B_R(x_0)\) 
      we have \(R\le d_{x}\le 3R\), \eqref{eq:tempered2.7} follows from \eqref{eq:tempered2.8}, \eqref{eq:tempered2.9}.
      This proves \eqref{eq:tempered2.7} in the smooth case.

      We now remove the auxiliary smoothness assumption by standard
      mollification \cite{EvansPDE,11}. Let
      \(\rho\in C_c^\infty(B_1)\) be a standard mollifier, set
      \(\rho_\varepsilon(x):=\varepsilon^{-n}\rho(x/\varepsilon)\), and
      define
      \[
          w_\varepsilon:=w*\rho_\varepsilon.
      \]
      Since the tempered fractional Laplacian is translation invariant,
      convolution commutes with the operator in the sense of
      distributions. Consequently, if
      \[
          U_\varepsilon
          :=
          \{x\in U:\operatorname{dist}(x,\partial U)>\varepsilon\},
      \]
      then
      \[
          (-\Delta)^s_\lambda w_\varepsilon
          =
          h*\rho_\varepsilon
          =:h_\varepsilon
          \qquad\text{in }U_\varepsilon.
      \]
      Moreover,
      \[
          \|w_\varepsilon\|_{C^\alpha(\mathbb R^n)}
          \le
          \|w\|_{C^\alpha(\mathbb R^n)}.
      \]

      Fix \(x_0\in U\), put \(d_0:=\operatorname{dist}(x_0,\partial U)\),
      and take \(0<\varepsilon<d_0/8\). On
      \(B_{3d_0/4}(x_0)\), the distance to \(\partial U\) is comparable
      with \(d_0\). The standard local estimates for convolution in
      noninteger H\"older spaces therefore give, with
      \(m=\lfloor\beta\rfloor\),
      \[
      \begin{aligned}
          &\sum_{l=0}^{m}
          d_0^{\,l+2s-\alpha}
          \|D^l h_\varepsilon\|_{L^\infty(B_{3d_0/4}(x_0))}
          +
          d_0^{\,\beta+2s-\alpha}
          [h_\varepsilon]_{C^\beta(B_{3d_0/4}(x_0))}
          \\
          &\qquad\le
          C\|h\|^{(2s-\alpha)}_{\beta;U},
      \end{aligned}
      \]
      with \(C\) independent of \(\varepsilon\) and \(x_0\). Applying
      the already proved smooth local estimate to \(w_\varepsilon\) on
      these balls yields the estimates corresponding to
      \eqref{eq:tempered2.8}, with a constant independent of
      \(\varepsilon\).

      Since \(w\in C^\alpha(\mathbb R^n)\),
      \(w_\varepsilon\to w\) uniformly in \(\mathbb R^n\). The uniform
      local \(C^{\beta+2s}\)-bounds and the lower semicontinuity of
      H\"older seminorms under uniform convergence imply
      \[
          w\in C^{\beta+2s}_{\mathrm{loc}}(U)
      \]
      and give the same local estimates for \(w\). Finally, applying
      the near-pair/far-pair argument used above, now to \(w\), gives
      \[
          \|w\|^{(-\alpha)}_{\beta+2s;U}
          \le
          C\left(
          \|w\|_{C^\alpha(\mathbb R^n)}
          +
          \|h\|^{(2s-\alpha)}_{\beta;U}
          \right).
      \]
      This proves the lemma for distributional solutions.
  \end{proof}

Finally, we prove Theorem~\ref{Pro:1.4}.
    \begin{proof}[Proof of Theorem~\ref{Pro:1.4}]
      The equation in \eqref{eq:1.1} gives
      \[
          (-\Delta)^s_\lambda u=f
          \qquad\text{in }\mathcal D'(\Omega).
      \]
      Since \(0<s<1\), the choice \(\alpha=s\) satisfies
      \(0<\alpha<1\) and \(\alpha<2s\). Hence
      Lemma~\ref{tempered2.10}, with \(U=\Omega\), \(w=u\), and \(h=f\),
      yields
      \[
          \|u\|_{\beta+2s;\Omega}^{(-s)}
          \le
          C\left(
          \|u\|_{C^s(\mathbb R^n)}
          +
          \|f\|_{\beta;\Omega}^{(s)}
          \right),
      \]
      which is the desired estimate.
    \end{proof}

\section{Higher-order weighted estimates: proof of Theorem~\ref{lemma1.6-fourth}}\label{sec:higher-order}

This final section completes the proof of the higher-order regularity
result. We apply the weighted Schauder estimate from
Section~\ref{sec:schauder} with the
exponent prescribed in Theorem~\ref{lemma1.6-fourth}, verify that the
relevant H\"older orders avoid the integer endpoints, and then read off
the pointwise decay of the derivatives of \(u\) and \(f\) in terms of the
distance to \(\partial\Omega\).

      \begin{proof}[Proof of Theorem~\ref{lemma1.6-fourth} ]
      We first check the ranges of \(\beta\) and \(\beta+2s\). Since
      $
          \beta=4-2s+\gamma,
      $
      we have
      $
          \beta+2s=4+\gamma.
      $

      If \(0<s\le \frac12\), then \(0<\gamma<2s\). Hence
      $
          \beta=4-2s+\gamma>4-2s\ge 3.
      $
      Since \(\gamma>0\), in fact
      $
          \beta>3.
      $
      Moreover, from \(\gamma<2s\), we get
      $
          \beta=4-2s+\gamma<4.
      $
      Therefore,
      $
          3<\beta<4.
      $
      Also, since \(0<\gamma<2s\le 1\),
      $
          4<\beta+2s=4+\gamma<5.
      $

      If \(\frac12<s<1\), then \(0<\gamma<2s-1\). Therefore
      $
          \beta=4-2s+\gamma>4-2s>2,
      $
      and
      $
          \beta=4-2s+\gamma<4-2s+2s-1=3.
      $
      Thus
      $
          2<\beta<3.
      $
      Moreover, since \(0<\gamma<2s-1<1\), we have
      $
          4<\beta+2s=4+\gamma<5.
      $

      Hence, in both cases, our hypotheses imply that $\beta>2$, and $4<\beta+2s<5.$
      In particular, neither \(\beta\) nor \(\beta+2s\) is an integer.

      By Theorem~\ref{Pro:1.4}, we have
      \[
          \|u\|_{\beta+2s;\Omega}^{(-s)}
          \le
          C\left(
              \|u\|_{C^s(\mathbb R^n)}
              +
              \|f\|_{\beta;\Omega}^{(s)}
          \right).
      \]
      Moreover, by~\ref{pro:1.1}, we obtain that
      \[
          \|u\|_{\beta+2s;\Omega}^{(-s)}
          \le
          C\left(
              \|f\|_{L^\infty(\Omega)}
              +
              \|f\|_{\beta;\Omega}^{(s)}
          \right).
      \]
      % Set
      % \[
      %     M:=
      %     \|f\|_{L^\infty(\Omega)}
      %     +
      %     \|f\|_{\beta;\Omega}^{(s)}.
      % \]

      Since \(4<\beta+2s<5\) and the Definition~\ref{Definition_tempered1.3}  of the weighted H\"older norm gives
      \[
          \sum_{1\le |\nu|\le 4}
          \sup_{x\in\Omega}
          \left(
              \delta(x)^{|\nu|-s}|D^\nu u(x)|
          \right)
          \le
          C\|u\|_{\beta+2s;\Omega}^{(-s)}
          \le
          C( \|f\|_{L^\infty(\Omega)}
          +
          \|f\|_{\beta;\Omega}^{(s)}).
      \]
      Consequently, for every multi-index \(\nu\) with \(1\le |\nu|\le 4\),
      \[
          |D^\nu u(x)|
          \le
          C( \|f\|_{L^\infty(\Omega)}
          +
          \|f\|_{\beta;\Omega}^{(s)})\delta(x)^{s-|\nu|}
          \le
          C\delta(x)^{s-|\nu|}
          \qquad \text{for }x\in\Omega.
      \]

      It remains to treat the case \(|\nu|=0\). Since \(u=0\) in
      \(\mathbb R^n\setminus\Omega\) and \(u\in C^s(\mathbb R^n)\), for every
      \(x\in\Omega\), we may choose \(x_0\in\partial\Omega\) such that
      \[
          |x-x_0|=\delta(x).
      \]
      Then
      \[
          |u(x)|
          =
          |u(x)-u(x_0)|
          \le
          [u]_{C^s(\mathbb R^n)}|x-x_0|^s
          \le
          C\|f\|_{L^\infty(\Omega)}\delta(x)^s.
      \]

      Finally, since \(\beta>2\) and the Definition~\ref{Definition_tempered1.3}  of the weighted H\"older norm gives
      \[
          \sum_{0\le |\nu|\le 2}
          \sup_{x\in\Omega}
          \left(
              \delta(x)^{|\nu|+s}|D^\nu f(x)|
          \right)
          \le
          C\|f\|_{\beta;\Omega}^{(s)}.
      \]
      Hence, for every multi-index \(\nu\) with \(0\le |\nu|\le 2\),
      \[
          |D^\nu f(x)|
          \le
          C\|f\|_{\beta;\Omega}^{(s)}
          \delta(x)^{-s-|\nu|}
          \le C \delta(x)^{-s-|\nu|}
          \qquad \text{for }x\in\Omega.
      \]
      The proof is complete.
      \end{proof}
\section*{Acknowledgments}
% =========================================================

M.H. was supported by  Distinguished Young Scholars of Gansu Province (23JRRA1020), 
Longyuan Young Talents Project of Gansu Province and NSFC 12471381.
J.Zhao was supported by Gansu Provincial Department of Education: 2026 Graduate Innovation Star Project (2026CXZX-152).

% =========================================================
% Declaration of competing interest
% =========================================================
\section*{Declaration of competing interest}

The authors declare that they have no known competing financial
interests or personal relationships that could have appeared to
influence the work reported in this paper.

% =========================================================
% Data availability
% =========================================================
\section*{Data availability}

No data was used for the research described in the article.

\newpage

\end{document}